\documentclass[12pt,oneside]{amsart}

\usepackage{setspace}

\usepackage{amssymb}   

\usepackage{amsmath}

\usepackage{mathrsfs}

\usepackage{stmaryrd}

\usepackage{enumitem}

\usepackage[all]{xy}

\usepackage{textcomp}

\usepackage{amsbsy}

\usepackage{graphicx}

\usepackage{datetime}
\renewcommand{\dateseparator}{-}
\renewcommand{\today}{\the\year \dateseparator \twodigit\month
\dateseparator \twodigit\day}

\usepackage
[pdfauthor={Firstname X Lastname},
 pdftitle={The Title},
 bookmarks=false]
{hyperref}

\title{The Title} 
\author{Firstname X Lastname}

\newtheorem{thm}{Theorem}[section]

\newtheorem{theorem}[thm]{Theorem}
\newtheorem{lemma}[thm]{Lemma}
\newtheorem{proposition}[thm]{Proposition}

\newtheorem{condition}[thm]{Condition}
\newtheorem{assumption}[thm]{Assumption}

\theoremstyle{definition}
\newtheorem{definition}[thm]{Definition}

\newtheorem{example}{Example}

\theoremstyle{remark}
\newtheorem{remark}[thm]{Remark}

\numberwithin{equation}{section}

\begin{document}

\pagenumbering{roman}


\thispagestyle{empty}

\vspace*{\fill}

\begin{center}
Formulation and properties of a divergence used to compare probability measures without absolute continuity and its application to uncertainty quantification\\
\vspace{0.2in}
A dissertation presented\\
\vspace{0.2in}
by\\
\vspace{0.2in}
Yixiang Mao\\
\vspace{0.2in}
to\\
\vspace{0.2in}
The Department of Mathematics\\
\vspace{0.2in}
in partial fulfillment of the requirements\\
for the degree of\\
Doctor of Philosophy\\
in the subject of\\
Mathematics\\
\vspace{0.2in}
Harvard University\\
Cambridge, Massachusetts\\
\vspace{0.2in}
August 2020\\
\end{center}

\vspace*{\fill}

\pagebreak






\thispagestyle{empty} 

\vspace*{\fill}

\begin{center}
\copyright \, 2020 -- Yixiang Mao \\
All rights reserved.
\end{center}

\vspace*{\fill}

\pagebreak




\doublespacing


\noindent Dissertation Advisor: Professor Paul Dupuis \hfill
Yixiang Mao

\vspace{0.5in}

\centerline{Formulation and properties of a divergence used to compare probability measures}
\centerline{without absolute continuity and its application to uncertainty quantification}

\vspace{0.8in}

\centerline{Abstract}

\vspace{0.3in}

This thesis develops a new divergence that generalizes relative entropy and can be used to compare probability measures without a requirement of absolute continuity. We establish properties of the divergence, and in particular derive and exploit a representation as an infimum convolution of optimal transport cost and relative entropy. We include examples of computation and approximation of the divergence, and its  applications in uncertainty quantification in discrete models and Gauss-Markov models.

\pagebreak





\tableofcontents

\pagebreak





\section*{Acknowledgements}

First and foremost, I would like to express my sincere gratitude to my advisor Professor Paul Dupuis for his continuous support, encouragement and patience throughout the entire period of my graduate studies. Without his generous sharing of knowledge and constant communication, I would have no chance of finishing my thesis.

I want to thank Professor Markos Katsoulakis for helpful discussion on various topics about my research. I want to thank Professor H.-T. Yau for sharing his knowledge on my work and allowing me to attend his research seminars, where I have benefitted a lot. I want to thank Professor S.-T. Yau for helpful discussion and help during my job search.

I would like to thank the professors and staff in the Department of Mathematics for all the support during my Ph.D. years. Especially, I want to thank Larissa Kennedy and Bell Marjorie, who helped me so much during my last year. I want to thank Professor S.-T. Yau and Tsinghua University for the generous financial support throughout the years, and my advisor Professor Paul Dupuis for financially supporting my last year. I want to give my special thanks to Jameel Al-Aidroos and Robin Gottlieb for their sharing of knowledge on teaching calculus and accessibility on answering my question of how to deliver mathematics idea more effectively.

I owe a lot to my colleagues and friends Arka Adhikari, Amol Aggarwal, Ziliang Che, Jun Hou Fung, Meng Guo, Jiaoyang Huang, Chi-Yun Hsu, Ben Landon, Yusheng Luo, Danny Shi, Koji Shimizu, Dennis Tseng, Ziquan Yang, Zijian Yao, Chenglong Yu, Boyu Zhang, Jonathan Zhu, for helpful discussion on mathematics, as well as all the good memories throughout my graduate studies. 

Last but not the least, I want to thank my family for their unconditional love and support throguhout writing this thesis, and my life in general.

\pagebreak




\pagenumbering{arabic}

\section{Introduction}
To compare different probabilistic models for a given application, one
needs a notion of \textquotedblleft distance\textquotedblright\ between
the distributions. The specification of this distance is a subtle
issue. Probability models are typically large or infinite dimensional, and the
usefulness of the distance will depend on its mathematical properties. Is it
convenient for analysis and optimization? Does it scale well with system size?

For situations that require an analysis of (probabilistic) model form uncertainly, the
quantity known as relative entropy (or Kullback-Leibler divergence)
is the most widely used such distance. This is true because relative entropy
has all the attractive properties asked for in the last paragraph, and
many more. (Relative entropy is not a true metric since it is not symmetric in
its arguments, but owing to its other attributes it is more widely
used for these purposes than any legitimate metric.)

The definition of relative entropy is as follows. Suppose $S$ is a Polish
space with metric $d(\cdot,\cdot)$ and associated Borel $\sigma$-algebra
$\mathcal{B}$. Let $\mathcal{P}(S)$ be the space of probability measures over
$(S,\mathcal{B})$. If $\mu,\nu\in\mathcal{P}(S)$ and $\mu$ is absolutely continuous with respect to $\nu$ (denoted $\mu
\ll\nu$), then
\begin{align}\label{RE_exp}
R(\mu\lVert\nu)\doteq\int_{S}\left(  \log\frac{d\mu}{d\nu}\right)  d\mu
\end{align}
(even though $\log{d\mu}/{d\nu}$ can take both positive and negative values,
as we discuss in the beginning of section \ref{section_2}, the definition is never ambiguous). Otherwise, we set $R(\mu\lVert\nu)=\infty$. 

While we cannot go into
all the reasons why relative entropy is so useful, it is essential that we
describe why it is convenient for the analysis of 
model form uncertainty. This is due to a
dual pair of variational formulas which relate $R(\mu\lVert\nu)$, integrals
with respect $\mu$, and what are called \textbf{risk-sensitive} integrals with respect
to $\nu$. Let $M_{b}(S)$ denote the set of bounded and measurable functions on $S$. Then
\cite[Proposition 1.4.2 and Lemma 1.4.3]{dupell4} gives

\begin{equation}
R(\mu\left\Vert \nu\right.  )=\sup_{g\in M_{b}(S)}\left\{  \int_{S}gd\mu
-\log\int_{S}e^{g}d\nu\right\},   \label{eqn:var_forms}%
\end{equation}
and for any $g\in M_{b}(S)$,
\begin{align}\label{dual_logexp_relative_entropy}
\log\int_{S}e^{g}d\nu=\sup_{\mu
\in\mathcal{P}(S)}\left\{  \int_{S}gd\mu-R(\mu\left\Vert \nu\right.
)\right\}  .
\end{align}
It is immediate from either of these that for $\mu,\nu\in\mathcal{P}(S)$
and $g\in M_{b}(S)$,%
\[
\int_{S}gd\mu\leq R(\mu\left\Vert \nu\right.  )+\log\int_{S}e^{g}d\nu.
\]
If we interpret $\nu$ as the \textbf{nominal}\ or \textbf{design }model
(chosen perhaps on the basis of data or for computational tractability) and
$\mu$ as the \textbf{true} model (or at least a {more accurate} model),
then according to the last display one obtains a bound on an integral with respect to the true model. 
(In fact by introducing a parameter one can obtain bounds that are in some sense optimal \cite{dupkatpanple}.) 
We typically interpret the integral $\int_{S}gd\mu$ as a \textbf{performance measure}, and so we have a bound
on the performance of the system under the true distribution in terms of the
relative entropy distance $R(\mu\left\Vert \nu\right.  )$, plus a
risk-sensitive performance measure under the design model. From this elementary
but fundamental inequality, and by exploiting the helpful qualitative and
quantitative properties of relative entropy, there has emerged a set of tools
that can be used to answer many questions where probabilistic model form
uncertainty is important, including 
\cite{baylov, brecsi,chodup,dupjampet,dupkatpanple,glaxu,hansar,lam,limshawat, arnlau, petjamdup}.

However, relative entropy has one important shortcoming: for the bound to be
meaningful we must have $R(\mu\left\Vert \nu\right.  )<\infty$, which imposes
the requirement of absolute continuity of the true model with respect to the
design model. For various uses, such as model building and model
simplification, this restriction can be significant. In the context of model
building, it can happen that one attempts to fit distributions to data by
comparing an empirical measure constructed using data with the elements of a
parameterized family, such as a collection of Gaussian distributions. In this
case the two distributions one would compare are singular, and so relative
entropy cannot be used. A second example, and one that occurs frequently in
the physical sciences, operations research and elsewhere,
is that a detailed model (such as the population process of a
chemical reaction network, which takes values in a lattice) is approximated by
a simpler process that takes values in the continuum (for example a diffusion
process). For exactly the same reason as in the previous example, these
processes, as well as their corresponding stationary distributions, are not
absolutely continuous.

Because relative entropy is not directly applicable to such problems,
significant effort has been put into investigating alternatives (\cite{blakanmur,blamur} and references therein). 
A class that has attracted some
attention (e.g., in the machine learning community) are the \textit{Wasserstein} or, more generally, 
\textit{optimal transport} distances \cite{kolparthosleroh,racrus,vil}. These distances, which are true
metrics, have certain attractive properties but also some 
shortcomings. The most important shortcomings are: (a) Wasserstein distances
do not in general scale well with respect to system dimension, and (b) such 
distances do not have an interpretation as the dual of a strictly convex function. To
be a little more concrete about point (b), it is
the strict concavity of the mapping
\[
g\rightarrow\int_{S}gd\mu-\log\int_{S}e^{g}d\nu
\]
in the variational representation for $R(\mu\left\Vert \nu\right.  )$ that
leads to tight bounds when applied to problems of control or optimization of
uncertain stochastic systems. In contrast, the analogous variational representation for
Wasserstein type distances involves the mapping $g\rightarrow\int_{S}%
gd\mu-\int_{S}gd\nu$. Point (a) is an issue in applications
to problems from the physical sciences, where large time horizons and large
dimensions are common.

Rather than give up entirely the attractive features of the dual pair
($R(\mu\left\Vert \nu\right.  )$, $\log\int_{S}e^{g}d\nu$), an alternative is to
be more restrictive regarding the class of costs or performance measures for
which bounds are required.  Indeed, the requirement of absolute
continuity in relative entropy is entirely due to the very large class of
functions, $M_{b}(S)$, appearing in (\ref{eqn:var_forms}). For a collection
$\Gamma\subset M_{b}(S)$ one can consider in lieu of $R(\mu\left\Vert
\nu\right.  )$ what we call the $\Gamma$-divergence, which is defined by
\[
G_{\Gamma}(\mu\left\Vert \nu\right.  )\doteq \sup_{g\in\Gamma}\left\{  \int_{S}%
gd\mu-\log\int_{S}e^{g}d\nu\right\}  .
\]
By imposing regularity conditions on $\Gamma$ (e.g., Lipschitz continuity,
additional smoothness) one generates (under mild additional conditions on $\Gamma$)
divergences which relax the absolute continuity condition. Thus one is trading
restrictions on the class of performance measures or observables for which
bounds are valid, for the enlargement of the class of distributions to which
the bounds apply. These divergences are of course not as nice as relative
entropy, but one can prove that they retain versions of its most important
properties. In addition, the dual function (which serves as the cost to be
minimized when considering problems of optimization or control) remains
$\log\int_{S}e^{g}d\nu$. This is important because the corresponding risk-sensitive
optimization and optimal control problems are well studied in the literature.

In our formulation of the $\Gamma$-divergence the underlying idea is that to extend the range of probability measures that can be compared, one must restrict the class of integrands that will be considered. 
However, this leads directly to an interesting connection with the Wasserstein distance
mentioned previously, which is that for suitable collections $\Gamma$ we will 
prove the inf-convolution expression
\begin{align*}
G_{\Gamma}(\mu\left\Vert \nu\right.  )&=\inf_{\gamma\in\mathcal{P}(S)}\left\{
W_{\Gamma}(\mu-\gamma)+R(\gamma\left\Vert \nu\right.  )\right\}  ,
\end{align*}
where $W_{\Gamma}$ is the Wasserstein metric whose dual (sup) formulation uses
the set of functions $\Gamma$. Moreover one recovers relative entropy by taking the
limit $b\rightarrow\infty$ in $G_{b\Gamma}(\mu\left\Vert \nu\right.  )$, which
may be useful if one wants to allow relatively small violations of the absolute
continuity restriction, while at the same time taking advantage of simple
approximations for the Wasserstein distance in the high transportation cost limit.

The organization of this thesis is as follows.
In Section 2 we define the $\Gamma$-divergence, and prove the first main result of this paper, which is the inf-convolution formula described above (Theorem \ref{thm:main}). In the same section, we show several properties of the $\Gamma$-divergence, and establish a convex duality formula for the $\Gamma$-divergence. 
In Section 3, we investigate the connection between $\Gamma$-divergence and optimal transport theory through investigation of a special choice of $\Gamma$, which are sets of bounded Lipschitz continuous functions. For this choice, we establish a relation between $\Gamma$-divergence and optimal transport cost, and prove existence and uniqueness for optimizers of variational representations of $\Gamma$-divergence (Theorem \ref{optimizer}), and also a formula for directional derivatives of the $\Gamma$-divergence (Theorem \ref{first_variation}). 
In Section 4, we look at several explicit examples where we draw intuition how relative entropy and Wasserstein metric interact with each other within the $\Gamma$-divergence. We also consider limits for $\Gamma$-divergence in this section.
Section 5 and Section 6 consider the application of $\Gamma$-divergence in uncertainty quantification, where Section 5 focuses on the application in the case where the reference measure $\nu$ is supported on discrete point settings and Section 6 focuses on the application in dealing with model uncertainty in stochastic differential equations.

\section{Formulation and Basic Properties of $\Gamma$-divergence}\label{section_2}
In this section, we rigorously define $\Gamma$ divergence and derive its basic properties. It will be seen that $\Gamma$ divergence is a way of generalizing relative entropy, which will be described first.

Throughout this section, $S$ is a Polish space with metric $d(\cdot,\cdot)$
and associated Borel $\sigma$-alegra $\mathcal{B}$. $C_{b}(S)$ denotes the
space of all bounded continuous functions from $S$ to $\mathbb{R}$, and
$M_{b}(S)$ denotes the space of all bounded measurable functions from $S$ to
$\mathbb{R}$. Let $\mathcal{P}(S)$ be the space of probability measures over
$(S,\mathcal{B})$, $\mathcal{M}(S)$ be the space of finite signed (Borel)
measures over $(S,\mathcal{B})$, and $\mathcal{M}_{0}(S)$ be the subspace of
$\mathcal{M}(S)$ whose total mass is $0$. $\overline{\mathbb{R}}%
\doteq\mathbb{R}\cup\{\infty\}$ is the extended real numbers. Throughout this
section, we consider $C_{b}(S)$ equipped with weak topology induced by
$\mathcal{M}(S)$. Thus for $f_n, f \in C_b(S)$, $f_{n}\rightarrow f$ if $\int_{S}f_{n}d\mu
\rightarrow\int_{S}fd\mu$ for all $\mu\in\mathcal{M}(S)$.

\subsection{Definition of the $\Gamma$-divergence }
We recall that for $\mu,\nu\in\mathcal{P}(S)$, relative entropy of $\mu$ with respect $\nu$ is defined as
\[
R(\mu\lVert\nu)\doteq\int_{S}\left(  \log\frac{d\mu}{d\nu}\right)  d\mu,
\]
whenever $\mu$ is absolutely continuous with respect to $\nu$. For $t\in\mathbb{R}$, define $t^{-}%
\doteq-(t\wedge0)$. Since the function $s(\log s)^{-}$ is bounded for
$s\in\lbrack0,\infty)$, whenever $\mu\ll\nu$,%

\[
\int_{S} \left(  \log\frac{d\mu}{d\nu}\right)  ^{-}d\mu=\int_{S} \frac{d\mu
}{d\nu}\left(  \log\frac{d\mu}{d\nu}\right)  ^{-}d\nu<\infty.
\]
Thus $R(\mu\lVert\nu)$ is always well defined.

We recall the Donsker-Varadhan variational representation (\ref{eqn:var_forms})
for relative
entropy.
We will use equation (\ref{eqn:var_forms}) as an equivalent
characterization of $R(\cdot\lVert\nu)$ on $\mathcal{P}(S)$, and consider an extension to
$\mathcal{M}(S)$ in the following lemma. With an abuse of notation, we will also call the extended function $R$.  
To set up the functionals of interest on a space with the proper structure (locally convex Hausdorff space),
we will use that 
\begin{equation}
\sup_{g\in C_{b}(S)}\left\{  \int_{S}gd\mu-\log\int_{S}e^{g}d\nu\right\}
=\sup_{g\in M_{b}(S)}\left\{  \int_{S}gd\mu-\log\int_{S}e^{g}d\nu\right\}
\label{eqn:altchar}%
\end{equation}
\cite[Lemma 1.4.3(a)]{dupell4}(It is worth notating in this reference, (\ref{eqn:altchar}) is only proved for $\mu,\nu\in \mathcal{P}(S)$. However, the exact same argument applies for $\mu,\nu\in\mathcal{M}(S)$, and we are using the latter version here). The fact that one obtains the same value
when supremizing over the smaller class $C_{b}(S)$ is closely related to the fact that 
$R(\mu\lVert\nu)$ is finite only when $\mu\ll\nu$.

\begin{lemma}
\label{lem:RE}Consider $R:\mathcal{M}(S)\times\mathcal{P}(S)\rightarrow
(-\infty,\infty]$ defined by (\ref{eqn:var_forms}). Then

\begin{enumerate}
\item $R(\mu\lVert\nu)\geq0$ and $R(\mu\lVert\nu)=0$ if and only if $\mu=\nu$,

\item $R(\cdot\lVert\cdot)$ is convex,

\item $R(\mu\lVert\nu)=\infty$ if $\mu\in\mathcal{M}(S)\backslash
\mathcal{P}(S)$.
\end{enumerate}
\end{lemma}

\begin{proof}
If we prove item 3, then items 1 and 2 will follow from the corresponding
statements when $\mu$ is restricted to $\mathcal{P}(S)~$\cite{dupell4}. If
$m=\mu(S)\neq1$, then taking $g(x)\equiv c$ a constant,
\[
\int_{S}gd\mu-\log\int_{S}e^{g}d\nu=c\mu(S)-c=c(m-1).
\]
Since $m\neq1$ and $c\in\mathbb{R}$, the right hand side of equation
(\ref{eqn:var_forms}) is $\infty$.

Suppose next that $\mu(S)=1$ but $\mu\in\mathcal{M}(S)\backslash
\mathcal{P}(S)$. Then there exist sets $A,B\in\mathcal{B}$ such that $A\cap
B=\varnothing,A\cup B=S,\mu(A)<0$ and $\mu(B)>0$. For $c>0$, let $g(x)=-c$ for
$x\in A$ and $g(x)=0$ for $x\in B$. Then
\[
\int_{S}gd\mu-\log\int_{S}e^{g}d\nu=c\left\vert \mu(A)\right\vert -C_{c},
\]
where $C_{c}\in(\log\nu(B),0)$ for all $c$. Letting $c\rightarrow\infty$ and
using (\ref{eqn:altchar}) shows $R(\mu\lVert\nu)=\infty$.
\end{proof}

\vspace{.5\baselineskip}
Though relative entropy has very attractive regularity and optimization
properties, as noted $R(\mu\lVert\nu)$ is finite if and only if $\mu\ll\nu$.
As such, it cannot be used to give a meaningful notion of ``distance'' without this absolute continuity restriction. In
order to define a meaningful divergence for a pair of probability measures
that are not mutually absolute continuous, but at the same time not to lose the
useful properties of the \textquotedblleft dual\textquotedblright\ function
$g\rightarrow\log\int_{S}e^{g}d\nu$ appearing in (\ref{eqn:var_forms}), a natural
approach is to restrict the set of test functions in the variational formula.
We define a criterion for the classes of \textquotedblleft
admissible\textquotedblright\ test functions we want to use.

\begin{definition}
\label{access} Let $\Gamma$ be a subset of $C_{b}(S)$ endowed with the inherited weak topology. We call $\Gamma$
\textbf{admissible} if the following hold.

1) $\Gamma$ is convex and closed.

2) $\Gamma$ is symmetric in that $g\in\Gamma$ implies $-g\in\Gamma$, and
$\Gamma$ contains all constant functions.

3) $\Gamma$ is determining for $\mathcal{P}(S)$, i.e., for any $\mu,\nu
\in\mathcal{P}(S)$ with $\mu\neq\nu$, there exists $g\in\Gamma$ such that%
\[
\int_{S} gd\mu\neq\int_{S} gd\nu.
\]

\end{definition}

We next define a new divergence by restricting the class of test functions in
the definition of relative entropy.

\begin{definition}
\label{def:defofV}Fix $\nu\in \mathcal{P}(S)$. For $\mu\in\mathcal{M}(S)$, we define the
$\mathbf{\Gamma}$\textbf{-divergence} associated with the admissible set $\Gamma$ by%
\[
G_{\Gamma}(\mu\lVert\nu)\doteq\sup_{g\in\Gamma}\left\{  \int_{S} gd\mu-\log\int_{S}
e^{g}d\nu\right\}  .
\]
We also define the following related quantity. For $\eta\in\mathcal{M}(S)$ let%
\[
W_{\Gamma}(\eta)\doteq\sup_{g\in\Gamma}\left\{  \int_{S} gd\eta\right\}
=\sup_{g\in\mathcal{C}_{b}(S)}\left\{  \int_{S} gd\eta-\infty1_{\{g\in\Gamma^{c}%
\}}\right\}  .
\]

\end{definition}

When $\Gamma$ is clear based on context, we will drop the subscript from
$G_{\Gamma}$ and $W_{\Gamma}$. Using a similar argument as in Lemma
\ref{lem:RE}, one can show that $G_{\Gamma}(\mu\lVert\nu)=\infty$ if $\mu(S)\neq 1$. The next theorem
states an important property of the ${\Gamma}$-divergence, which is that it
can be written as a convolution involving relative entropy and $W_{\Gamma}$.

\begin{theorem}
\label{thm:main} Assume $\Gamma$ is an admissible set. Then for $\mu\in \mathcal{M}(S)$, $\nu\in\mathcal{P}(S)$,
\[
G_{\Gamma}(\mu\lVert\nu)=\inf_{\gamma\in\mathcal{P}(S)}\left\{  R(\gamma
\lVert\nu)+W_{\Gamma}(\mu-\gamma)\right\}
\]

\end{theorem}

\begin{remark}
The theorem tells us that by restricting the set of test functions in the
variational representation of relative entropy, we get a quantity which is an
inf-convolution of relative entropy and a metric. It will be pointed out in
Section \ref{Wass} that by restricting $\Gamma$ to Lipschitz functions
with respect to a cost function $c(x,y)$ that satisfies some specified
conditions, $W_{\Gamma}(\mu-\nu)$ will be the corresponding optimal transport
cost from $\mu$ to $\nu$.
\end{remark}

The rest of this section is focused on the proof of Theorem \ref{thm:main}. In
order to do this, we need a few definitions and also will find it convenient
to consider a more general setting.

\begin{definition}
Points $x$ and $y$ in a topological space $Y$ can be \textbf{separated} if
there exists an open neighborhood $U$ of $x$ and an open neighborhood $V$ of
$y$ such that $U$ and $V$ are disjoint ($U\cap V=\varnothing$). $Y$ is a
$\mathbf{Hausdorff}$ space if all distinct points in $Y$ are pairwise separable.
\end{definition}

\begin{definition}
A subset $C$ of a topological vector space $Y$ over the number field
$\mathbb{R}$ is

1. $\mathbf{convex}$ if for any $x,y\in C$ and any $t\in\lbrack0,1]$,
$tx+(1-t)y\in C$,

2. $\mathbf{balanced}$ if for all $x\in C$ and any $\lambda\in\mathbb{R}$ with
$|\lambda|\leq1$, $\lambda x\in C$,

3. $\mathbf{absorbant}$ if for all $y\in Y$, there exists $t>0$ and $x\in C$
such that $y=tx$.

A topological vector space $Y$ is called $\mathbf{locally}$ $\mathbf{convex}$
if the origin has a local topological basis of convex, balanced and absorbent sets.
\end{definition}

\begin{definition}
For a topological vector space $Y$ over the number field $\mathbb{R}$, its
$\mathbf{topological\ dual\ space}$ $Y^{*}$ is defined as the space of all
continuous linear functionals ${\displaystyle \varphi:Y\to{\mathbb{R} }}$.

The $\mathbf{weak^{\ast}\ topology}$ on $Y^{\ast}$ is the topology induced by
$Y$. In other words, it is the coarsest topology such that functional
$y:Y^{\ast}\rightarrow\mathbb{R}$, $y(\varphi)=\varphi(y)$ is continuous in
$Y^{\ast}$.

For $y\in Y$ and $\varphi\in Y^{\ast}$, we also write $\langle y,\varphi
\rangle\doteq\varphi(y)=y(\varphi)$.
\end{definition}

Now let $Y$ be a Hausdorff locally convex space with $Y^{\ast}$ being its
topological dual space and endowed with the weak* topology.

\begin{definition}\label{convex_dual_def}
For a function $f:Y\rightarrow\overline{\mathbb{R}}$, its $\mathbf{convex}%
\ \mathbf{dual}$ $f^{\ast}:Y^{\ast}\rightarrow\overline{\mathbb{R}}$ is
defined by%
\[
f^{\ast}(z)=\sup_{y\in Y}\left\{  \langle y,z\rangle-f(y)\right\}  .
\]

\end{definition}

\begin{definition}\label{inf_conv_def}
Let $f_{1},f_{2}:Y\rightarrow\overline{\mathbb{R}}$ be two functions. We
define the inf-convolution of $f_{1}$ and $f_{2}$ by%
\[
\left[  f_{1}\Box f_{2}\right]  (y)\doteq\inf_{y_{1}\in Y}\{f_{1}(y_{1}%
)+f_{2}(y-y_{1})\}.
\]

\end{definition}

\begin{definition}\label{lsc_hull_def}
For a function $f:Y\rightarrow\overline{\mathbb{R}}$ the
$\mathbf{lower\ semicontinuous\ hull}$ $\overline{f}$ is defined by%
\[
\overline{f}(x)\doteq\sup\{g(x):g\leq f,g:Y\rightarrow\overline{\mathbb{R}%
}\ is\ continuous\}.
\]

\end{definition}

\begin{definition}\label{proper&domain_def}
A convex function $f:Y\rightarrow\overline{\mathbb{R}}$ is \textbf{proper} if
there exists $y\in Y$ such that $f(y)<\infty$. The \textbf{domain} of a
convex, proper funciton $f$ is defined by%
\[
\mathrm{dom}(f)\doteq\{y\in Y:f(y)<\infty\}.
\]

\end{definition}

Now let us introduce an important lemma. 

\begin{lemma}
\label{inf-cov} \cite[Theorem 2.3.10]{botgrawan} Let $f_{i}:Y\rightarrow
\overline{\mathbb{R}}$ be convex, proper and lower-semicontinuous functions
fulfilling $\bigcap_{i=1}^{m}\mathrm{dom}(f_{i})\neq\varnothing$. Then one has%
\[
\left(  \sum_{i=1}^{m}f_{i}\right)  ^{\ast}=\overline{f_{1}^{\ast}\Box
\cdots\Box f_{m}^{\ast}}.
\]

\end{lemma}

In our use we take $Y=C_{b}(S)$ equipped with topology induced by
$\mathcal{M}(S)$, i.e., the topological basis around $g\in Y$ is taken as sets
of the form%

\[
\left\{  f\in Y:\int_{S} fd\mu_{k}\in\left(  \int_{S}gd\mu_{k}-\epsilon_{k}%
,\int_{S}gd\mu_{k}+\epsilon_{k}\right)  ,k=1,2,\dots,m\right\}  ,
\]
where $m\in\mathbb{N},\{\mu_{k}\}_{k=1,2,\dots,m}\subset\mathcal{M}(S)$ and
$\epsilon_{k}>0,k=1,2,\dots,m$ are arbitrary. It can be easily verified that
under this topology, $C_{b}(S)$ is a Hausdorff locally convex space, with
$C_{b}(S)^{\ast}=\mathcal{M}(S)$ \cite[Theorem 3.10]{rud}. 
For $g\in C_{b}(S)$ and $\mu\in
\mathcal{M}(S)$, we define the bilinear form%

\[
\langle g,\mu\rangle\doteq\int_{S}gd\mu.
\]

We are now ready to prove the main theorem.

\vspace{.5\baselineskip}
\begin{proof}
[Proof of Theorem \ref{thm:main}]Define $H_{1},H_{2}:C_{b}(S)\rightarrow
\overline{\mathbb{R}}$ by%
\[
H_{1}(g)\doteq\log\int_{S}e^{g}d\nu\text{ and }H_{2}(g)\doteq\infty
1_{\Gamma^{c}}(g).
\]
Then
\begin{align*}
G_{\Gamma}(\mu\lVert\nu) &  =\sup_{g\in\Gamma}\left\{  \int_{S}gd\mu-\log
\int_{S}e^{g}d\nu\right\}  \\
&  =\sup_{g\in C_{b}(S)}\left\{  \int_{S}gd\mu-\log\int_{S}e^{g}d\nu
-\infty1_{\Gamma^{c}}(g)\right\}  \\
&  =\left(  H_{1}+H_{2}\right)  ^{\ast}(\mu).
\end{align*}

Notice that $\{0\}\in\mathrm{dom}(H_{1})\cap\mathrm{dom}(H_{2})\neq
\varnothing$, and both $H_{1}$ and $H_{2}$ are proper and convex. For
lower-semicontinuity, under the topology induced by $\mathcal{M}(S)$, $H_{1}$ is
lower semicontinuous because of (\ref{dual_logexp_relative_entropy}) and the fact that supremum of continuous functions are lower semicontinuous, and $H_{2}$ is lower semicontinuous since $\Gamma$ is
closed. Thus,
by Lemma \ref{inf-cov}%
\[
G_{\Gamma}(\mu\lVert\nu)=(H_{1}+H_{2})^{\ast}(\mu)=[\overline{H_{1}^{\ast}\Box
H_{2}^{\ast}}](\mu).
\]
By equation (\ref{eqn:var_forms}) and the definition of $W_{\Gamma}$, we know that%
\[
R(\mu\lVert\nu)=H_{1}^{\ast}(\mu)\text{ and }W_{\Gamma}(\eta)=H_{2}^{\ast
}(\eta).
\]
In the following display, the first equality is due to the definition of
inf-convolution, and the second is since $R(\gamma\lVert\nu)<\infty$ only when
$\gamma\in\mathcal{P}(S)$:%
\begin{align*}
H_{1}^{\ast}\Box H_{2}^{\ast}(\mu)  &  =\inf_{\gamma\in\mathcal{M}(S)}\left\{
R(\gamma\lVert\nu)+W_{\Gamma}(\mu-\gamma)\right\} \\
&  =\inf_{\gamma\in\mathcal{P}(S)}\left\{  R(\gamma\lVert\nu)+W_{\Gamma}%
(\mu-\gamma)\right\}  .
\end{align*}

Thus the last thing we need to prove is that $H_{1}^{\ast}\Box H_{2}^{\ast}$
is lower semicontinuous. Note that relative entropy is lower semicontinuous in
the first argument in the weak topology \cite[Lemma 1.4.3 (b)]{dupell4}, and
$W_{\Gamma}$ is lower semicontinuous in the weak topology since
it is the supremum of a collection of linear functionals. Let%
\[
F(\mu)\doteq H_{1}^{\ast}\Box H_{2}^{\ast}(\mu)=\inf_{\gamma\in\mathcal{P}%
(S)}\left\{  R(\gamma\lVert\nu)+W_{\Gamma}(\mu-\gamma)\right\}  .
\]
\noindent Consider any sequence $\mu_{n}\Rightarrow\mu$ with $\mu_n,\mu\in\mathcal{M}(S)$. Here ``$\Rightarrow$'' means convergence in the weak$^*$ topology, i.e., for any $f\in C_b(S)$, $\int f d\mu_n\to\int f d\mu$. Let $\varepsilon>0$, and
for each $\mu_n$ let $\gamma_n$ satisfy%
\[
R(\gamma_{n}\lVert\nu)+W_{\Gamma}(\mu_{n}-\gamma_{n})\leq F(\mu_{n}%
)+\varepsilon.
\]
\noindent We want to show that%
\begin{equation}
\liminf_{n\rightarrow\infty}F(\mu_{n})\geq F(\mu). \label{eqn:lsc}%
\end{equation}
If $\liminf_{n\rightarrow\infty} F(\mu_n)=\infty$, the inequality above holds
automatically. Assuming $\liminf_{n\rightarrow\infty} F(\mu_n)<\infty$, let
$n_k$ be a subsequence such that%

\[
\lim_{k\rightarrow\infty}F(\mu_{n_{k}})=\liminf_{n\rightarrow\infty}F(\mu
_{n}).
\]
Notice that%
\[
R(\gamma_{n_{k}}\lVert\nu)\leq R(\gamma_{n_{k}}\lVert\nu)+W_{\Gamma}%
(\mu_{n_{k}}-\gamma_{n_{k}})\leq F(\mu_{n_{k}})+\varepsilon.
\]
Since $\{F(\mu_{n_{k}})\}_{k\geq1}$ is bounded, we know that $\{\gamma_{n_{k}%
}\}_{k\geq1}$ is tight \cite[Lemma 1.4.3(c)]{dupell4}. Then we can take a further
subsequence that converges weakly. For simplicity of notation, let $n_{k}$
denote this subsequence, and let $\gamma_{\infty}$ denote the weak limit of
$\gamma_{n_{k}}$. Then using the lower semicontinuity of $R(\cdot\lVert\nu)$
on $\mathcal{P}(S)$ and the lower semicontinuity of $W_{\Gamma}$ on
$\mathcal{M}(S)$,%
\begin{align*}
\liminf_{n\rightarrow\infty}F(\mu_{n})+\varepsilon &  =\lim_{k\rightarrow
\infty}F(\mu_{n_{k}})+\varepsilon\\
&  \geq\lim_{k\rightarrow\infty}\left[  R(\gamma_{n_{k}}\lVert\nu
)+W(\mu_{n_{k}}-\gamma_{n_{k}})\right] \\
&  \geq R(\gamma_{\infty}\lVert\nu)+W(\mu-\gamma_{\infty})\\
&  \geq\inf_{\gamma\in\mathcal{P}(S)}\left\{  R(\gamma\lVert\nu)+W_{\Gamma
}(\mu-\gamma)\right\} \\
&  =F(\mu).
\end{align*}
Since $\varepsilon>0$ is arbitrary this establishes (\ref{eqn:lsc}), and thus
$F$ is lower semicontinuous in $\mathcal{M}(S)$. The theorem is proved.
\end{proof}

\subsection{Properties of the $\Gamma$-divergence}

Theorem \ref{thm:main} provides an interesting characterization of the $\Gamma$-divergence. Before we
continue to specific choices of $\Gamma$, we first state some general
properties associated with $\Gamma$-divergence. Throughout this section we fix an
admissible set $\Gamma$, and thus drop the subscript from $G_{\Gamma}$ and
$W_{\Gamma}$ in this section. Also, now that we have established the
expression for $G$ as an inf-convolution as in Theorem \ref{thm:main}, we no
longer need to consider $G$ as a function on $\mathcal{M}(S)\times
\mathcal{P}(S)$, and instead can consider it just on $\mathcal{P}(S)\times
\mathcal{P}(S)$, since we want to use $G$ as a measure of how two probability distributions differ.

\begin{lemma}
\label{basic} For $(\mu,\nu)\in\mathcal{P}(S)\times\mathcal{P}(S)$ define
$G(\mu\lVert\nu)$ by Definition \ref{def:defofV} and assume $\Gamma$ is
admissible. Then the following properties hold.

1) $G(\mu\lVert\nu)\geq0$, with $G(\mu\lVert\nu)=0$ if and only if $\mu=\nu$.

2) $G(\mu\lVert\nu)$ is a convex and lower semicontinuous function of
$(\mu,\nu)$. In particular, $G(\mu\lVert\nu)$ is a convex, lower
semicontinuous function of each variable $\mu$ or $\nu$ separately.

3) $G(\mu\lVert\nu)\leq R(\mu\lVert\nu)$ and $G(\mu\lVert\nu)\leq W(\mu-\nu)$.
\end{lemma}

\begin{remark}
1) The first property justifies our calling $G$ a divergence as the term is used in information theory.

2) Relative entropy has the property that for each fixed $\nu\in\mathcal{P}(S)$,
$R(\cdot\lVert\nu)$ is strictly convex on $\{\mu\in\mathcal{P}(S):R(\mu
\lVert\nu)<\infty\}$. However,
$G(\cdot\lVert\nu)$ in general is not strictly convex.
\end{remark}

\begin{proof}
[Proof of Lemma \ref{basic}]
1) As noted in Lemma \ref{lem:RE}, $R(\cdot\lVert\cdot)$ is non-negative
\cite[Lemma 1.4.1]{dupell4}, and for any $\mu\in\mathcal{P}(S)$%
\[
W(\mu)=\sup_{g\in\Gamma}\left\{  \int_{S}gd\mu\right\}  \geq\int_{S}0d\mu=0.
\]
\noindent Thus%
\[
G(\mu\lVert\nu)=\inf\{R(\mu_{1}\lVert\nu)+W(\mu_{2}):\mu_{1}+\mu_{2}=\mu
\}\geq0.
\]
Also by Lemma \ref{lem:RE}, $R(\mu_{1}\lVert\nu)=0$ if and only if
$\mu_{1}=\nu$. Thus $G(\mu\lVert\nu)=0$ if and only if%
\[
W(\mu-\nu)=\sup_{g\in\Gamma}\left\{  \int_{S}gd(\mu-\nu)\right\}  =0,
\]
which tells us $\mu=\nu$ since $\Gamma$ is admissible.

2) This is a straightforward corollary of Theorem \ref{thm:main}, since the
supremum of a collection of linear and continuous functionals is both convex
and lower semicontinuous.

3) This follows from Theorem \ref{thm:main} and that $R(\nu\lVert\nu)=W(0)=0$.
\end{proof}

\vspace{.5\baselineskip}
For relative entropy we have the following lemma \cite[Proposition 1.4.2]{dupell4}.

\begin{lemma}\label{REvar} For all $g\in C_{b}(S)$%
\[
\log\int_{S}e^{g}d\nu=\sup_{\mu\in\mathcal{P}(S)}\left\{  \int_{S}gd\mu
-R(\mu\lVert\nu)\right\}  ,
\]
where the supremum is achieved uniquely at $\mu_{0}$ satisfying%
\[
\frac{d\mu_{0}}{d\nu}(x)\doteq\frac{e^{g(x)}}{\int_{S}e^{g}d\nu}.
\]
\end{lemma}

A similar duality formula holds for the $\Gamma$-divergence when $g\in\Gamma$.

\begin{theorem}
\label{incredible} If $\Gamma$ is admissible then for $g\in\Gamma$%
\[
\log\int_{S} e^{g}d\nu=\sup_{\mu\in\mathcal{P}(S)}\left\{  \int_{S} gd\mu-G(\mu
\lVert\nu)\right\}  .
\]
\end{theorem}

\begin{proof}
Using the definition of $\Gamma$-divergence 
\begin{align*}
\sup_{\mu\in\mathcal{P}(S)}\left\{  \int_{S}gd\mu-G(\mu\lVert\nu)\right\}   &
=\sup_{\mu\in\mathcal{P}(S)}\left\{  \int_{S}gd\mu-\sup_{f\in\Gamma}\left\{
\int_{S}fd\mu-\log\int_{S}e^{f}d\nu\right\}  \right\} \\
&  \leq\sup_{\mu\in\mathcal{P}(S)}\left\{  \int_{S}gd\mu-\left\{  \int
_{S}gd\mu-\log\int_{S}e^{g}d\nu\right\}  \right\} \\
&  =\log\int_{S}e^{g}d\nu.
\end{align*}
On the other hand, we know for relative entropy that%
\[
\log\int_{S}e^{g}d\nu=\sup_{\mu\ll\nu}\left\{  \int_{S}gd\mu-R(\mu\lVert
\nu)\right\}  .
\]
Since $G(\mu\lVert\nu)\leq R(\mu\lVert\nu)$,
\begin{align*}
\log\int_{S}e^{g}d\nu &  =\sup_{\mu\ll\nu}\left\{  \int_{S}gd\mu-R(\mu
\lVert\nu)\right\} \\
&  \leq\sup_{\mu\ll\nu}\left\{  \int_{S}gd\mu-G(\mu\lVert\nu)\right\} \\
&  \leq\sup_{\mu\in\mathcal{P}(S)}\left\{  \int_{S}gd\mu-G(\mu\lVert
\nu)\right\}.
\end{align*}
The statement of the theorem follows from the two inequalities.
\end{proof}

\vspace{.5\baselineskip}
The last theorem has two important implications.
The first is related to the fact that Lemma \ref{REvar} implies bounds for $\int_S g d\mu$ when $R(\mu\lVert\nu)$ is bounded,
an observation that has served as the basis for the analysis of various aspects of model form uncertainty \cite{chodup,dupkatpanple}. Using Theorem \ref{incredible}, we obtain analogous bounds on $\int_S g d\mu$ for $g\in\Gamma$ when  $G(\mu\lVert\nu)$ is bounded. Applications of these bounds will be further developed in Section \ref{application_static} and Section \ref{application_diffusion}.
The second is that for $g\in\Gamma$, if we take $\mu_{0}$ as defined in Lemma \ref{REvar}, then
\begin{align*}
\log\int_{S} e^{g} d\nu &  = \int_{S} g d\mu_{0} - R(\mu_{0}\lVert\nu)\\
&  \leq\int_{S} g d\mu_{0} - G(\mu_{0}\lVert\nu)\\
&  \leq\sup_{\mu\in \mathcal{P}(S)}\left\{  \int_{S} g d\mu- G(\mu\lVert\nu) \right\} \\
&  = \log\int_{S} e^{g} d\nu,
\end{align*}
where the first inequality comes from $G(\mu_{0}\lVert\nu)\leq
R(\mu_0\lVert\nu)$. Since both
inequalities above must be equalities, we must have%
\[
R(\mu_{0}\lVert\nu) = G(\mu_{0}\lVert\nu).
\]
The next lemma gives a more detailed picture of $G(\mu\lVert\nu)$ when $\mu\ll\nu$.

\begin{lemma}
\label{goodcase}
For $\mu,\nu\in \mathcal{P}(S)$, if $\mu\ll\nu$ then 
$$G(\mu\lVert\nu) = \sup_{\gamma\in\mathcal{A}(S)}\left\{\int_S \log\left(\frac{d\gamma}{d\nu}\right)d\mu\right\},$$
where 
\[
\mathcal{A}(S)\doteq\left\{\gamma\in\mathcal{P}(S): \gamma\ll\nu,\exists g \in \Gamma \mbox{ such that }\frac{d\gamma}{d\nu}(x) =e^{g(x)} \mbox{ for } x\in\mathrm{supp}(\nu) \right\}.
\]
\end{lemma}

\begin{proof}
We use the definition $$G(\mu\lVert \nu) = \sup_{g\in\Gamma} \left\{\int_S g d\mu - \log \int_S e^g d\nu \right\}$$
to prove this lemma.
For any $g\in\Gamma$, we define $\gamma_g \in \mathcal{P}(S)$ by the relation
$$\frac{d\gamma_g}{d\nu}(x)=\frac{e^{g(x)}}{\int_S e^g d\nu}$$
for $x\in \mathrm{supp}(\nu)$, and $\gamma_g(\mathrm{supp}(\nu)^c)=0$. Then for $x\in \mathrm{supp}(\nu)$,
$$\log\left(\frac{d\gamma_g}{d\nu}(x)\right) = g(x) - \log\int_S e^g d\nu.$$
Since $\mu\ll\nu$, we have
$$\int_S \log\left(\frac{d\gamma_g}{d\nu}\right)d\mu = \int_S gd\mu - \log\int_S e^g d\nu,$$
and thus
$$G(\mu\lVert \nu) = \sup_{g\in\Gamma} \left\{\int_S g d\mu - \log \int_S e^g d\nu \right\}\leq \sup_{\gamma\in\mathcal{A}(S)}\left\{\int_S \log\left(\frac{d\gamma}{d\nu}\right)d\mu\right\}.$$

On the other hand, for any $\gamma\in \mathcal{A}(S)$, by definition, we can find a $g_\gamma\in \Gamma$ such that
$$g_\gamma(x) = \log\left(\frac{d\gamma}{d\nu}(x)\right)$$
for $x\in\mathrm{supp}(\nu)$. Then
$$\int_S g_\gamma d\mu - \log\int_S e^{g_\gamma} d\nu = \int_S\log\left(\frac{d\gamma}{d\nu}\right)d\mu.$$
Thus
$$\sup_{\gamma\in \mathcal{A}(S)}\left\{\int_S \log\left(\frac{d\gamma}{d\nu}\right)d\mu\right\}\leq \sup_{g\in\Gamma} \left\{\int_S g d\mu - \log \int_S e^g d\nu \right\}=G(\mu\lVert \nu). $$
Combining  the two inequalities completes the proof.
\end{proof}

\begin{remark} 
When $\mu\in\mathcal{A}(S)$ we always have $G(\mu\lVert\nu)=R(\mu\lVert\nu)$.
This is because if $\gamma\in \mathcal{A}(S)$ then  $\mu\ll\gamma$, and therefore
$$\int_S\log\left(\frac{d\mu}{d\nu}\right)d\mu - \int_S\log\left(\frac{d\gamma}{d\nu}\right)d\mu =\int_S\log\left(\frac{d\mu}{d\gamma}\right)d\mu = R(\mu\lVert\gamma)\geq 0. $$
Rearranging gives 
\[
\int_S \log\left(\frac{d\gamma}{d\nu}\right)d\mu=R(\mu\lVert\nu)-R(\mu\lVert\gamma), 
\]
and so 
$$G(\mu\lVert\nu)= \sup_{\gamma\in\mathcal{A}(S)}\left\{\int_S \log\left(\frac{d\gamma}{d\nu}\right)d\mu\right\}= R(\mu\lVert\nu).$$

This statement is not valid when $\mu\ll\nu$ does not hold, since then $\log(d\gamma/d\nu)$ is not defined in $\mathrm{supp}(\mu)\backslash\mathrm{supp}(\nu)$, thus 
$$\int_S \log\left(\frac{d\gamma}{d\nu}\right)d\mu$$
is not well defined.
\end{remark}

\section{Connection with Optimal Transport Theory}\label{Wass}
In the proceeding sections, we discussed general properties for the $\Gamma$-divergence 
with an admissible set $\Gamma\subset C_{b}(S)$. In this section, we
discuss specific choices of $\Gamma$ which relate the $\Gamma$-divergence with
optimal transport theory. First we state some well known results in
optimal transport theory.

\subsection{Preliminary results from optimal transport theory}

The results in this section are from \cite[Chapter 4]{racrus}. The
general Monge-Kantorovich mass transfer problem with given marginals $\mu
,\nu\in\mathcal{P}(S)$ and cost function $c:S\times S\rightarrow\mathbb{R}%
_{+}$ is%
\[
\mathcal{C}(c;\mu,\nu)\doteq\inf_{\pi\in\Pi(\mu,\nu)}\left\{  \int_{S\times
S}c(x,y)\pi(dx,dy)\right\}  ,
\]
where $\Pi(\mu,\nu)$ denotes the collection of all probability
measures on $S\times S$ with first and second marginals being $\mu$ and $\nu$, respectively.

A natural dual problem with respect to this is%
\[
\mathcal{B}(c;\rho)\doteq \sup_{f\in\mathrm{Lip}(c,S;C_{b}(S))}\left\{  \int
_{S}f(x)\rho(dx)\right\}  ,
\]
where $\rho=\mu-\nu$, $C_b(S)$ denotes the set of bounded continuous functions mapping $S$ to $\mathbb{R}$ and%

\begin{align}\label{Lip_bdd}
\mathrm{Lip}(c,S;C_{b}(S))\doteq\left\{  f\in C_{b}(S):f(x)-f(y)\leq
c(x,y)\mbox{ for all } x,y\in S\right\}  .
\end{align}
We want to know when%
\begin{equation}
\mathcal{C}(c;\mu,\nu)=\mathcal{B}(c,\rho) \label{eqn:duality}%
\end{equation}
holds. The following is a necessary and sufficient condition.
As with many results in this section,
one can extend in a trivial way to the case where costs are bounded from below, rather than non-negative.
Recall that $S$ is a Polish space.

\begin{condition}
\label{con:crep}
There is a nonempty
subset $Q\subset C_{b}(S)$ such that the cost
$c:S\times S\rightarrow[0,\infty]$  has the representation%
\begin{equation}
c(x,y)=\sup_{u\in Q}\left(  u(x)-u(y)\right)  \quad\text{for\ all}\ (x,y)\in
S\times S.\label{eqn:crep}
\end{equation}
\end{condition}

\begin{theorem}
\cite[Theorem 4.6.6]{racrus}\label{massdual} Under Condition \ref{con:crep},  
 (\ref{eqn:duality}) holds.
\end{theorem}

\begin{remark}
Condition \ref{con:crep} implies that $c$ satisfies the triangle inequality, i.e., for all $x,y,z\in S$
$$c(x,z) \leq c(x,y)+c(y,z).$$
This follows easily from 
\begin{align*}
\sup_{u\in Q}\left(  u(x)-u(z)\right) &= \sup_{u\in Q}\left(  (u(x)-u(y)) +(u(y)-u(z))\right)\\
&\leq \sup_{u\in Q}\left(  u(x)-u(y)\right) +\sup_{u\in Q}\left(  u(y)-u(z)\right).
\end{align*}

On the other hand, Condition \ref{con:crep} also allows for a wide range of choices of
$c(x,y)$. For example, suppose that $c$ is a continuous metric on $S$, where
continuity is with respect to the underlying metric
of $S$. Then we can choose%
\[
Q=\left\{  \min(c(x,x_{0}),n):x_{0}\in S,n\in\mathbb{N}\right\}  .
\]
It is easily verified that $Q\subset C_{b}(S)$, and that with this choice of
$Q$ \eqref{eqn:crep} holds.
\end{remark}

\subsection{$\Gamma$-divergence with the choice $\Gamma=\mathrm{Lip}(c,S;C_{b}%
(S))$}

Suppose $\Gamma=\mathrm{Lip}(c,S;C_{b}(S))$, with $c:S\times S\rightarrow
[0, \infty]$ satisfying 
 Condition \ref{con:crep}. To make the presentation simple, we have assumed that $c$ is non-negative,
 and further assume it is symmetric, meaning $c(x,y) = c(y,x) \geq 0$ for any $x,y\in S$.
To distinguish from $W_\Gamma(\mu-\nu)$ for general $\Gamma$,
we denote the transport cost for $\mu, \nu\in \mathcal{P}(S)$ by
\[
W_{c}(\mu,\nu)\doteq\sup_{g\in \mathrm{Lip}(c,S;C_{b}(S))}\left\{  \int_{S}gd(\mu-\nu)\right\} .
\]
Then by Theorem \ref{massdual}%
\[
W_{c}(\mu,\nu)=\sup_{g\in\mathrm{Lip}(c,S;C_{b}(S))}\left\{  \int_{S}gd(\mu-\nu)\right\}
=\inf_{\pi\in\Pi(\mu,\nu)}\left\{  \int_{S\times S}c(x,y)\pi(dx,dy)\right\}
.
\]

\begin{condition}\label{mea-det}
Suppose $\mathrm{Lip}(c,S;C_{b}(S))$ is measure determining, i.e., 
for all $\mu,\nu\in\mathcal{P}(S)$, $\mu \neq \nu$, there exists $f\in \mathrm{Lip}(c,S;C_{b}(S))$ such that 
$$\int_S fd\mu \neq \int_S f d\nu.$$
\end{condition}
Under Condition \ref{mea-det}, $\Gamma$ is admissible (see Definition \ref{access}), and by
Theorem \ref{thm:main}%
\begin{equation}
G_{\Gamma}(\mu\lVert\nu)=\sup_{g\in\Gamma}\left\{  \int_{S}gd\mu-\log\int
_{S}e^{g}d\nu\right\}  =\inf_{\gamma\in\mathcal{P}(S)}\left\{  W_{c}(\mu,\gamma)+R(\gamma
\lVert\nu)\right\}  . \label{Vari}%
\end{equation}
Hence by choosing $\Gamma$ properly, we get that the $\Gamma$-divergence is an
infimal convolution of relative entropy, which is a convex function of likelihood ratios, and an
optimal transport cost, which depends on a cost structure on the space $S$.
Natural questions to raise here are the following. \quad

i) Do there exist optimizers $\gamma^{\ast}$ and $g^{\ast}$ in the variational
problem (\ref{Vari})? If so, are they unique?

ii) How can one characterize $\gamma^{\ast}$ and $g^{\ast}$?

iii) For a fixed $\nu\in\mathcal{P}(S)$, what is the effect of a perturbation of
$\mu$ on $G_{\Gamma}(\mu\lVert\nu)$?

\medskip We will address these questions sequentially in this section. From now on, we will drop the subscript $\Gamma$ in this section for the simplicity of writing. We consider the case where $G(\mu\lVert\nu)<\infty$. To
impose additional constraints on $\mu$ and $\nu$ such that $G(\mu\lVert\nu)<\infty$ holds, we make a further
assumption on $c$.

\begin{condition}
\label{finite} There exist $a:S\rightarrow\mathbb{R}_{+}$ such that%
\[
c(x,y)\leq a(x)+a(y).
\]
\end{condition}

Now consider $\mu,\nu\in L^{1}(a)\doteq\{\theta\in\mathcal{P}(S):\int
_{S}a(x)\theta(dx)<\infty\}$. Then
\begin{align*}
G(\mu\lVert\nu)  &  =\inf_{\gamma\in\mathcal{P}(S)}\left\{  W_{c}(\mu,\gamma)+R(\gamma\lVert
\nu)\right\} \\
&  \leq W_{c}(\mu,\nu)\\
&  =\inf_{\pi\in\Pi(\mu,\nu)}\left\{  \int_{S\times S}c(x,y)\pi(dx,dy)\right\}
\\
&  \leq\inf_{\pi\in\Pi(\mu,\nu)}\left\{  \int_{S\times S}\left[
a(x)+a(y)\right]  \pi(dx,dy)\right\} \\
&  =\int_{S}a(x)\mu(dx)+\int_{S}a(y)\nu(dy)\\
&  <\infty.
\end{align*}
We will assume the following mild conditions on the space $S$ and cost $c$ to make $\mathrm{Lip}(c,S;C_{b}(S))$ precompact.

\begin{condition}\label{Cond:c_cont}
There exists $\left\{K_m\right\}_{m\in\mathbb{N}}$ such that $K_m\subset S$ is compact, $K_{m}\subset K_{m+1}$ for all $m\in\mathbb{N}$, and $S = \cup_{m\in\mathbb{N}}K_m$. For each $m$, there exists $\theta_m: \mathbb{R}_+ \to \mathbb{R}_+$, such that $\lim_{a\to 0} \theta_m(a) = 0$, and $\delta_m >0$, such that for any $x,y\in K_m$ satisfying $d(x,y)\leq \delta_m$,
$$c(x,y)\leq \theta_m(d(x,y)).$$
\end{condition}
Recalling the definition (\ref{Lip_bdd}),
we define the unbounded version as follows
$$\mathrm{Lip}(c,S)\doteq\left\{  f\in C(S):f(x)-f(y)\leq
c(x,y)\mbox{ for all } x,y\in S\right\},$$
where $C(S)$ is the set of continuous functions mapping $S$ to $\mathbb{R}$. Before we proceed, we state the following lemma, which will be used repeatedly in this section.

\begin{lemma}\label{beyond_bounded}
If $g\in \mathrm{Lip}(c,S)$ and $\theta,\nu\in P(S)$ satisfy $\int_S |g|d\theta <\infty$,
then
$$\int_S gd\theta - \log \int_S e^g d\nu \leq G(\theta\lVert\nu)\leq R(\theta\lVert\nu).$$
\end{lemma}
\begin{proof}
We use a standard truncation argument. Since by Lemma \ref{basic} we already have $G(\theta\lVert\nu)\leq R(\theta\lVert\nu)$,
we only need to prove the first inequality in the statement of the lemma.
If $\int_S e^g d\nu = \infty$, then 
$$\int_S gd\theta - \log \int_S e^g d\nu = -\infty < 0\leq G(\theta\lVert\nu).$$
Hence we only need consider the case $\int_S e^g d\nu < \infty$. Let $g_n = \min(\max(g,-n),n)\in \mathrm{Lip}(c,S;C_b(S))=\Gamma$ for $n\in\mathbb{N}$. We have 
$|g_n(x)|\leq |g(x)|$
and 
$$\lim_{n\to\infty} g_n(x) = g(x) \quad x\in S.$$
Thus by the dominated convergence theorem
$$\lim_{n\to\infty}\int_S g_n d\theta = \int_S g d\theta.$$
Also we have 
$$e^{g_n(x)} \leq e^{g(x)}+1 \mbox{ and }
\lim_{n\to\infty} e^{g_n(x)} = e^{g(x)}.$$
Since we only consider the case $\int_S e^gd\nu < \infty$, again using the dominated convergence theorem we have, 
$$\lim_{n\to\infty}\int_S e^{g_n} d\nu = \int_S e^g d\nu.$$
Together with (\ref{eqn:var_forms}), this gives
\begin{align*}
\int_S gd\theta - \log\int_S e^gd\nu &= \lim_{n\to\infty} \left(\int_S g_nd\theta - \log\int_S e^{g_n} d\nu\right)\\
&\leq \sup_{f\in \Gamma}\left\{\int_S fd\theta - \log\int_S e^f d\nu\right\}\\
&=G(\theta\lVert\nu).
\end{align*}
\end{proof}

\vspace{\baselineskip}
Now we are ready to state the first main theorem of this section.

\begin{theorem}\label{optimizer} Suppose Conditions \ref{con:crep}, \ref{mea-det}, \ref{finite} and \ref{Cond:c_cont} are 
satisfied. Fix $\mu,\nu\in L^{1}(a)$.
Then the following conclusions hold.

\noindent1) There exists a unique optimizer $\gamma^{\ast}$ in the expression
(\ref{Vari}).

\noindent2) There exists an optimizer $g^{\ast}\in\mathrm{Lip}(c,S)$ in the
expression (\ref{Vari}), which is unique up to an additive constant in
$\mathrm{supp}(\mu)\cup\mathrm{supp}(\nu)$.

\noindent3) $g^{\ast}$ and $\gamma^{\ast}$ satisfy the following
conditions:

i)
\[
\frac{d\gamma^{\ast}}{d\nu}(x)=\frac{e^{g^{\ast}(x)}}{\int_{S} e^{g^{\ast}(y)}%
d\nu},\quad \nu-a.s.
\]

ii)
\[
W_{c}(\mu,\gamma^{*})=\int_{S} g^{*} d(\mu-\gamma^{*}).
\]

\end{theorem}
\begin{remark}
With many analogous expressions related to relative entropy, one can only conclude the uniqueness of $\gamma^*$ and $g^*$ (up to constant addition) almost everywhere according to either the measure $\mu$ or $\nu$. However, because of the regularity condition $g^*\in\mathrm{Lip}(c,S;C(S))$ and Condition \ref{Cond:c_cont}, the uniqueness of $g^*$ (up to constant addition) on $\mathrm{supp}(\mu)\cup\mathrm{supp}(\nu)$ will follow.
\end{remark}

\begin{proof}
For $n\in\mathbb{N}$ consider $\gamma_{n}\in\mathcal{P}(S)$ that satisfies%
\[
R(\gamma_{n}\lVert\nu)+W_{c}(\mu,\gamma_{n})\leq G(\mu\lVert\nu)+\frac{1}{n}.
\]
Then by \cite[Lemma 1.4.3(c)]{dupell4} $\{\gamma_{n}\}_{n\geq1}$ is precompact in the weak topology, and thus
has a convergent subsequence $\left\{\gamma_{n_{k}}\right\}_{k\geq 1}$. Denote $\gamma^{\ast}\doteq
\lim_{k\rightarrow\infty}\gamma_{n_{k}}$. Then by the lower semicontinuity of
both $R(\cdot\lVert\nu)$ and $W_{c}(\mu,\cdot)$, we have%
\[
R(\gamma^{\ast}\lVert\nu)+W_c(\mu,\gamma^{\ast})\leq\liminf_{k\rightarrow\infty
}\left(  R(\gamma_{n_{k}}\lVert\nu)+W_{c}(\mu,\gamma_{n_{k}})\right)  \leq
G(\mu\lVert\nu).
\]
Since
\[
G(\mu\lVert\nu)=\inf_{\gamma\in\mathcal{P}(S)}\left\{  R(\gamma\lVert
\nu)+W_{c}(\mu,\gamma)\right\}  \leq R(\gamma^{\ast}\lVert\nu)+W_c(\mu
,\gamma^{\ast})
\]
it follows that%
\[
G(\mu\lVert\nu)=R(\gamma^{\ast}\lVert\nu)+W_c(\mu,\gamma^{\ast}),
\]
which shows that $\gamma^{\ast}$ is an optimizer in expression (\ref{Vari}).
If there exist two optimizers $\gamma_{1}\neq\gamma_{2}$, the strict convexity
of $R(\cdot\lVert\nu)$ and convexity of $W_{c}(\mu,\cdot)$ imply that for
$\gamma_{3}=\frac{1}{2}(\gamma_{1}+\gamma_{2})$%
\begin{align*}
R(\gamma_{3}\lVert\nu)+W_{c}(\mu,\gamma_{3})  &  <\frac{1}{2}\left(  \left(
R(\gamma_{1}\lVert\nu)+W_{c}(\mu,\gamma_{1})\right)  +\left(  R(\gamma
_{2}\lVert\nu)+W_{c}(\mu,\gamma_{2})\right)  \right) \\
&  =G(\mu\lVert\nu)\leq R(\gamma_{3}\lVert\nu)+W_{c}(\mu,\gamma_{3}),
\end{align*}
a contradiction. Thus the existence and uniqueness of an optimizer
$\gamma^{\ast}$ of (\ref{Vari}) is proved, which establishes 1) in the statement of the theorem. Before proceeding, we establish the following lemma.

\begin{lemma}\label{opt_integrability}
If $g\in \mathrm{Lip}(c,S)$, then
$$\int_S g d\gamma^* <\infty.$$
\end{lemma}
\begin{proof}
This can be shown by contradiction. Assume there exists $h \in \mathrm{Lip}(c,S)$ such that $\int_S |h| d\gamma^*=\infty$. By symmetry, we can  just consider $h$ to be non-negative, since $\max (h,0)\in\mathrm{Lip}(c,S)$ and $h=\max(h,0) - \max(-h,0)$. Thus we can assume there exists non-negative $h\in \mathrm{Lip}(c,S)$ satisfying
$$\int_S h d\gamma^* = \infty,$$
and by the fact that $\mu\in L^1(a)$ together with Condition \ref{finite}, 
\begin{align*}
\int_S h d\mu &\leq \int_S \left[ h(0) + c(x,0)\right] \mu(dx)\\
&= h(0) + a(0) +\int_S a(x) \mu(dx)<\infty.
\end{align*}
Then
\begin{align*}
W_c(\mu,\gamma^*) & = \sup_{g \in \mathrm{Lip(c,S)}}\int_S g d(\mu - \gamma^*)\\
&\geq \limsup_{n\to\infty}\int_S \max(-h,-n) d(\mu-\gamma^*)\\
& = \limsup_{n\to\infty}\left[ \int_S \max(-h,-n) d\mu + \int_S \min(h,n)d\gamma^*\right] \\
& = \int_S -h d\mu + \int_S h d\gamma^*\\
&= \infty,
\end{align*}
where the second to last equation comes from dominated and monotone convergence theorems applied to the first and second terms respectively. However, since $\gamma^*$ is the optimizer, we have 
$$W_c(\mu,\gamma^*) \leq W_c(\mu,\gamma^*) + R(\gamma^*\lVert \nu) = G(\mu\lVert \nu) <\infty.$$
This contradiction shows the integrability of $\gamma^*$ with respect to any $\mathrm{Lip}(c,S)$ function. 
\end{proof}

\vspace{\baselineskip}
Now we consider the other variational representation of $G(\mu\lVert \nu)$, which is 
$$G(\mu\lVert \nu) = \sup_{g\in\mathrm{Lip}(c,S;C_b(S))}\left\{\int_S gd\mu-\log\int_S e^{g}d\nu\right\}.$$
Take $g_n\in\mathrm{Lip}(c,S;C_b(S))$ such that 
$$ G(\mu\lVert\nu) - 1/n \leq \int_S g_n d\mu-\log\int_S e^{g_n}d\nu \leq G(\mu\lVert \nu) .$$
Without loss of generality, we can  assume $g_n(x_0)=0$ for some fixed $x_0\in K_0\subset S$.  Since for any $m\in\mathbb{N}$ $K_m\subset S$  is compact, we have that $\left\{g_n\right\}_{n\in\mathbb{N}}$ is bounded and equicontinuous on $K_m$ by Condition \ref{Cond:c_cont}. By the  Arzel\`{a}-Ascoli theorem, there exists a subsequence of $\left\{g_n\right\}_{n\in\mathbb{N}}$ that converges uniformly in $K_m$. Using a diagonalization argument, by taking subsequences sequentially along $\left\{K_m\right\}_{m\in\mathbb{N}}$, where the next subsequence is a subsequence of the former one, and taking one element from each sequence, we conclude there exists a subsequence $\left\{g_{n_j}\right\}_{j\in\mathbb{N}}$, that converges uniformly in any $K_m$. Since $S=\cup_{m\in\mathbb{N}}K_m$, we conclude that $\left\{g_{n_j}\right\}_{j\in\mathbb{N}}$ converges pointwise in $S$. Denotes its limit by $g^*$. It can be easily verified that $g^*\in \mathrm{Lip}(c,S)$.
 
Since $g_{n_j}(x)\leq g_{n_j}(x_0) + c(x_0,x)\leq a(x_0)+a(x)$ and $\int_S\left( a(x_0)+a(x)\right)d\mu <\infty$, by the dominated convergence theorem 
$$\lim_{j\to\infty} \int_S g_{n_j} d\mu = \int_S g^* d\mu.$$
By Fatou's lemma, we have 
$$\liminf_{j\to\infty} \int_S e^{g_{n_j}}d\nu \geq \int e^{g^*}d\nu,$$
and therefore 
$$-\log\int e^{g^*}d\nu\geq\limsup_{j\to\infty} -\int_S e^{g_{n_j}}d\nu.$$

Putting these together, we have 
\begin{align*}
G(\mu\lVert \nu) &= \sup_{g\in\mathrm{Lip}(c,S;C_b(S))}\left\{\int_S gd\mu-\log\int_S e^{g}d\nu\right\}\\
&\leq \limsup_{j\to\infty} \left\{\int_S g_{n_j}d\mu-\log\int_S e^{g_{n_j}}d\nu\right\}\\
&\leq \int_S g^* d\mu - \log\int_S e^{g^*} d\nu\\
&= \left(\int_S g^* d\mu - \int_S g^* d\gamma^* \right)+ \left(\int_S g^* d\gamma^* - \log\int_S e^{g^*} d\nu\right).
\end{align*}
We can add and subtract $\int_S g^*d\gamma^*$ because we have proved in Lemma \ref{opt_integrability} that $\gamma^*$ is integrable with respect to functions in $\mathrm{Lip}(c,S)$, and $g^*\in \mathrm{Lip}(c,S)$. 
By Lemma \ref{beyond_bounded} we have
$$\int_S g^* d\gamma^* - \log\int_S e^{g^*} d\nu \leq R(\gamma^*\lVert\nu). $$
We also have 
$$\int_S g^* d\mu - \int_S g^* d\gamma^* \leq  W_c(\mu,\gamma^*),$$
which is due to 
\begin{align*}
W_c(\mu,\gamma^*)&= \sup_{g\in\mathrm{Lip}(c,S;C_b(S))}\int_S gd(\mu-\gamma^*)\\
&\geq \limsup_{n\to\infty}\int_S\max(\min(g^*,n),-n)d(\mu-\gamma^*)\\
&= \int_S g^*d(\mu-\gamma^*),
\end{align*}
where the last equality is because of the dominated convergence theorem and integrability of $|g^*|$ with respect to $\mu$ and $\gamma^*$ (Lemma \ref{opt_integrability}).
We can therefore continue the calculation above as 
\begin{align*}
 &\left(\int_S g^* d\mu - \int_S g^* d\gamma^* \right)+ \left(\int_S g^* d\gamma^* - \log\int_S e^{g^*} d\nu\right)\\
&\qquad\leq W_c(\mu,\gamma^*) + R(\gamma^*\lVert\nu)\\
&\qquad= G(\mu\lVert\nu).
\end{align*}

Since both the upper and lower bounds on the inequalities coincide, we must have all inequalities to be equalities, and therefore
$$G(\mu\lVert \nu) = \int_S g^*d\mu - \log\int_S e^{g^*} d\nu,$$
$$\int_S g^* d\mu - \int_S g^* d\gamma^* =  W_c(\mu,\gamma^*),$$
and 
$$\int_S g^* d\gamma^* - \log\int_S e^{g^*} d\nu = R(\gamma^*\lVert\nu). $$
The last equation gives us the relationship
$$\frac{d\gamma^*}{d\nu}(x) = \frac{e^{g^*(x)}}{\int_S e^{g^*}d\nu} \quad \nu-a.s.$$

 Thus we have shown the existence of optimizer $g^*\in\mathrm{Lip}(c,S)$ and its relationship with $\gamma^*$. Lastly, for any other optimizer $\bar{g}\in\mathrm{Lip}(c,S)$ the analogous argument shows
%
$$\frac{d\gamma^*}{d\nu}(x) = \frac{e^{\bar{g}(x)}}{\int_S e^{\bar{g}}d\nu} \quad \nu-a.s.$$
Hence  uniqueness of the optimizer $g^*$ in $\mathrm{supp}(\nu)$ up to $\nu-a.s.$ is also proved. 

To determine the uniqueness of the optimizer $g^*$ in $\mathrm{supp}(\mu)$, we take an optimal transport plan between $\mu$ and $\gamma^*$, $\pi^{*}\in
\Pi(\mu,\gamma^{*})$ for $W_{c}(\mu,\gamma^{*})$, which means
$$W_c(\mu,\gamma^*) = \inf_{\pi\in\Pi(\mu,\gamma^*)}\left\{\int_{S\times
S}c(x,y)\pi(dx,dy)\right\}=\int_{S\times
S}c(x,y)\pi^*(dx,dy).$$
(Note that $c$ satisfying Condition \ref{con:crep} is lower semicontinuous, and therefore   \cite[Theorem 1.5]{ambgig} shows the existence of an optimal transport plan $\pi^*$.)

Since $g^*(x)-g^*(y)\leq c(x,y)$, 
\begin{align*}
W_c(\mu,\gamma^*) &= \int_{S\times
S}c(x,y)\pi^*(dx,dy)\\
 &\geq \int_{S\times
 S}\left[ g^*(x)-g^*(y)\right] \pi^*(dx,dy)\\
&= \int_S g^*(x) (\mu-\gamma^*)(dx)\\
&= W_c(\mu,\gamma^*).
\end{align*}
Then the only inequality above must be equality, which implies that for $(x,y)\in\mathrm{supp}(\gamma^*)$, $g^*(x)-g^*(y)=c(x,y)$, $\pi^*- a.s.$ This is also true for any other optimizer $\bar{g}\in \mathrm{Lip}(c,S)$ for (\ref{Vari}). Thus we are able to determine $g^*$ uniquely in $\mathrm{supp}(\mu)$ $\mu-a.s.$ with the help of $\pi^*$ and data of $g^*$ in $\mathrm{supp}(\nu)$. Lastly, since $g^*\in\mathrm{Lip}(c,S)$ and by Condition \ref{Cond:c_cont}, we conclude the uniqueness of $g^*$ in $\mathrm{supp}(\mu)\cup\mathrm{supp}(\nu)$ by the continuity of $g^*$.
\end{proof}

\begin{remark}
When $\mu\ll\nu$ Theorem \ref{optimizer} implies that for some constant $c_0$%
\[
g^{\ast}(x)=\log\left(  \frac{d\gamma^{\ast}}{d\nu}(x)\right)  -c_{0}%
\quad \nu-a.s.
\]
Hence 
\[
G(\mu\lVert\nu)=\int_{S}g^{\ast}d\mu-\log\int_{S}e^{g^{\ast}}d\nu=\int_{S}
\log\left(  \frac{d\gamma^{\ast}}{d\nu}(x)\right)  d\mu,
\]
and so the $\Gamma$-divergence of $\mu$ with respect to $\nu$ looks like a \textquotedblleft modified\textquotedblright\ version of relative entropy.
\end{remark}

The next theorem tells us that 3) of Theorem \ref{optimizer} is not only a
description of of the pair of optimizer $(g^{*},\gamma^{*})$, but also a
characterization of it.

\begin{theorem}
\label{verif} 
Suppose Conditions \ref{con:crep}, \ref{mea-det}, \ref{finite} and \ref{Cond:c_cont} are 
satisfied. Fix $\mu,\nu\in L^{1}(a)$. If $g_{1}\in$\emph{ Lip}$(c,S)$ and
$\gamma_{1}\in\mathcal{P}(S)$ satisfy condition 3) in Theorem
\ref{optimizer}, then $(g_{1},\gamma_{1})$ are optimizers in the corresponding
variational problem (\ref{Vari}).
\end{theorem}

\begin{proof}
The theorem follows from the two variational characterization of $\Gamma$-divergence
in (\ref{Vari}).
Condition 3) of Theorem \ref{optimizer} implies
\[
R(\gamma_{1}\lVert\nu)=\int_{S}g_{1}d\gamma_{1}-\log\int_{S}e^{g_{1}}d\nu
\mbox{ and }
W_{c}(\mu,\gamma_{1})=\int_{S}g_{1}d(\mu-\gamma_{1}),
\]
and therefore 
\[
R(\gamma_{1}\lVert\nu)+W_{c}(\mu,\gamma_{1})=\int_{S}g_{1}d\mu-\log\int
_{S}e^{g_{1}}d\nu.
\]
This implies
\begin{align*}
G(\mu\lVert\nu)  &  =\inf_{\gamma\in\mathcal{P}(S)}\left\{  R(\gamma\lVert
\nu)+W_{c}(\mu,\gamma)\right\} \\
&  \leq R(\gamma_{1}\lVert\nu)+W_{c}(\mu,\gamma_{1})\\
&  =\int_{S}g_{1}d\mu-\log\int_{S}e^{g_{1}}d\nu\\
&  \leq G(\mu\lVert\nu).
\end{align*}
The first inequality comes from the fact that $\gamma_1\in \mathcal{P}(S)$, while the second needs a little more discussion, which will be given below. Assuming this, the last display shows that 
$(g_{1},\gamma_{1})$ are optimizers. The second inequality follows from Lemma \ref{beyond_bounded} and the fact that 

\begin{align*}
\int_S |g_1(x)|\mu(dx) &\leq \int_S |g_1(0)|+c(0,x)\mu(dx)\\
&\leq \int_S |g_1(0)|+a(0)+a(x)\mu(dx)<\infty.
\end{align*}
The proof is complete.
\end{proof}

\vspace{\baselineskip}
The last theorem answers questions i) and ii) raised earlier in this section, now we want to answer iii), which is to characterize the directional derivative of $G(\mu\lVert\nu)$ in the first variable when fixing the second one, i.e., 
$$\lim_{\varepsilon\to 0^+} \frac{1}{\varepsilon}\left(G(\mu+\varepsilon \rho\lVert \nu)-G(\mu\lVert\nu)\right)$$
for $\rho\in\mathcal{M}_0(S)$ which satisfies certain conditions. From Theorem \ref{optimizer} and remarks following it  we know that any optimizer $g^*$ of expression (\ref{Vari}) is unique in $\mathrm{supp}(\mu)\cup \mathrm{supp}(\nu)$, up to an addition constant. However, there is still freedom to choose $g^*$ in $S\backslash\left\{\mathrm{supp}(\mu)\cup \mathrm{supp}(\nu)\right\}$, since the variational problem in (\ref{Vari}) does not take into account of the information of $g^*$ outside $\mathrm{supp}(\mu)\cup \mathrm{supp}(\nu)$, other than requiring that $g^*$ belong to $\mathrm{Lip}(c,S)$. We will define a special $g^*$ that is uniquely defined not only in $\mathrm{supp}(\mu)$ and $\mathrm{supp}(\nu)$, but also on $S\backslash\left\{\mathrm{supp}(\mu)\cup \mathrm{supp}(\nu)\right\}$. For $x\in \left\{\mathrm{supp}(\mu)\cup \mathrm{supp}(\nu)\right\}$, we let $g^*$ be an optimizer of (\ref{Vari}). For $x\in S\backslash\left\{\mathrm{supp}(\mu)\cup \mathrm{supp}(\nu)\right\}$, set
\begin{align}\label{the_opt}
g^*(x)\doteq\inf_{y\in\mathrm{supp}(\nu)}\left\{g^*(y)+c(x,y)\right\}.
\end{align}
From now on we will use the notation $g^*$ for the function defined in \eqref{the_opt}. The following lemma 
confirms that this construction of $g^*$ still lies in $\mathrm{Lip}(c,S)$.

\begin{lemma}
The following two statements hold.

1) For $x\in \mathrm{supp}(\mu)$, the expression $(\ref{the_opt})$ also holds. In other words, for $x\in S\backslash\mathrm{supp}(\nu)$, we have 

$$g^*(x) = \inf_{y\in\mathrm{supp}(\nu)}\left\{g^*(y)+c(x,y)\right\}.$$

2) $g^*$ defined by equation (\ref{the_opt}) is in $\mathrm{Lip}(c,S)$. In addition, 
\begin{align}\label{biggest}
g^*(x) = \sup\{h(x): h\in \mathrm{Lip}(c,S), h(y) = g^*(y)\ \mathrm{for}\ y\in\mathrm{supp}(\nu) \}
\end{align}
\end{lemma}
\begin{proof}

1) For $x\in\mathrm{supp}(\mu)$, from an optimal transport plan between $\mu$ and $\gamma^*$, $\pi^*\in\Pi(\mu,\gamma^*)$ for $W_c(\mu,\gamma^*)$, we know there exists $y_x\in \mathrm{supp}(\nu)$ such that $(x,y_x)\in\mathrm{supp}(\pi^*)$. Thus by \cite{ambgig}[Remark 1.15],
$$g^*(x)= g^*(y_x) +c(x,y_x).$$
On the other hand, by Theorem \ref{optimizer}, $g^*|_{\mathrm{supp}(\nu)\cup\mathrm{supp}(\mu)}\in\mathrm{Lip}(c,S)$. Thus, for other $y\in\mathrm{supp}(\nu)$, $g^*(x)\leq c(x,y)+g^*(y)$, which in turn gives

$$g^*(x) \leq \inf_{y\in\mathrm{supp}(\nu)}\left\{g^*(y)+c(x,y)\right\}.$$
By combining the two expressions above, we have for $x\in\mathrm{supp}(\mu)$, (\ref{the_opt}) also holds. In other words, $g^*$ is totally characterized by $g^*|_\mathrm{supp}(\nu)$ and (\ref{the_opt}).

2) We check the Lipschitz condition for $g^*$ for pair of points according to whether they are in $\mathrm{supp}(\nu)$. First, since $g^*|_{\mathrm{supp}(\mu)\cup\mathrm{supp}(\nu)}$ is an optimizer for (\ref{Vari}), by Theorem \ref{optimizer}, $g^*|_{\mathrm{supp}(\mu)\cup\mathrm{supp}(\nu)}$ satisfies the Lipchitz condition, i.e., for $y_1,y_2\in\mathrm{supp}(\nu)$,
\begin{align}\label{Lip_in_supp_nu}
g^*(y_2)-c(y_1,y_2)\leq g^*(y_1)  \leq g^*(y_2)+c(y_1,y_2).
\end{align}

For $x\not\in\mathrm{supp}(\nu)$ and $y\in\mathrm{supp}(\nu)$, by (\ref{the_opt}) we have 
$$g^*(x) \leq g^*(y)+c(x,y).$$
On the other hand, for any $0<n<\infty$, there exists $y_1\in\mathrm{supp}(\nu)$ such that
$$g^*(x) \geq g^*(y_1) + c(x,y_1) -\frac{1}{n}.$$
Notice that both $y$ and $y_1$ are from $\mathrm{supp}(\mu)$, so from (\ref{Lip_in_supp_nu}), we have $g^*(y_1)\geq g^*(y) - c(y,y_1)$, thus we have 
\begin{align*}
g^*(x) &\geq g^*(y_1) + c(x,y_1) -\frac{1}{n}\\
&\geq g^*(y) - c(y,y_1) + c(x,y_1) -\frac{1}{n}\\
&\geq g^*(y) - c(y,x) -\frac{1}{n},
\end{align*}
where the last equation uses the triangle inequality property of $c$. Now since $n>0$ is arbitrary, by getting $n\to\infty$, we have 
$$g^*(x) \geq g^*(y) - c(x,y).$$
Combine both sides together, we have for $x\notin\mathrm{supp}(\nu)$, $y\in\mathrm{supp}(\mu)$,
$$g^*(y) - c(x,y) \leq g^*(x) \leq g^*(y) + c(x,y).$$

Lastly, we check for $x_1,x_2\not\in \mathrm{supp}(\nu)$ the Lipschitz constraint is satisfied. 
From  the definition (\ref{the_opt}),
we know for any $n<\infty$
there exists $y_1\in \mathrm{supp}(\nu)$ such that 
$$c(x_1,y_1) - 1/n \leq g^*(x_1)-g^*(y_1).$$
Also, because $y_1\in\mathrm{supp}(\nu)$,
$$g^*(x_2)-g^*(y_1) \leq c(x_2,y_1).$$
Therefore 
\begin{align*}
g^*(x_2)-g^*(x_1)&\leq (c(x_2,y_1)-c(x_1,y_1))+1/n\\
&\leq c(x_1,x_2)+1/n,
\end{align*}
where the last inequality uses the triangle inequality property of $c$. Since $n>0$ is arbitrary and we can swap the roles of $x_1$ and $x_2$, we have proved the Lipschitz condition of $g^*$ for $x_1,x_2\not\in\mathrm{supp}(\nu)$. Thus the statement that $g^*\in \mathrm{Lip}(c,S)$ is proven.

For (\ref{biggest}), notice that for $h\in\mathrm{Lip}(c,S)$,  $x\in S$ and $y\in\mathrm{supp}(\nu)$,
$$h(x) \leq h(y)+ c(x,y).$$
So if $h(y) = g^*(y)$ for $y\in\mathrm{supp}(\nu)$, then for $x\in S\backslash\mathrm{supp}(\nu)$,
$$h(x) \leq \inf_{y\in\mathrm{supp}(\nu)}\left\{h(y)+c(x,y)\right\} = \inf_{y\in\mathrm{supp}(\nu)}\left\{g^*(y)+c(x,y)\right\}=g^*(x).$$
Since $g^*$ is also in $\mathrm{Lip}(c,S)$, this proves (\ref{biggest}).
\end{proof}

\vspace{\baselineskip}
Then based on this construction, we have the following result.  

\begin{theorem}\label{first_variation}
Take $\Gamma=\mathrm{Lip}(c,S;C_b(S))$ where $c$ satisfies the conditions of Theorem \ref{optimizer} and $\mu,\nu\in L^1(a)$. Take $\rho=\rho_+-\rho_-\in\mathcal{M}_0(S)$ where $\rho_+,\rho_-\in\mathcal{P}(S)$ are mutually singular probability measures, $\rho_+\in L^1(a)$, and assume there exists $\varepsilon_0>0$ such that $\mu+\varepsilon\rho\in\mathcal{P}(S)$ for $0<\varepsilon\leq\varepsilon_0$. Then
$$\lim_{\varepsilon\to 0^+} \frac{1}{\varepsilon}\left(G(\mu+\varepsilon \rho\lVert \nu)-G(\mu\lVert\nu)\right)=\int_S g^* d\rho.$$
where $g^*$ is the optimizer found in (\ref{the_opt}). 
\end{theorem}

\begin{proof}
We use the variational formula (\ref{Vari}) for $G(\mu+\varepsilon\rho\lVert \nu)$, where $\mu+\varepsilon\rho\in\mathcal{P}(S)$ and $\rho_+\in L^1(a)$. Recall that $g^*$ is the optimizer for (\ref{Vari}). Using Lemma \ref{beyond_bounded} with $\theta = \mu+\varepsilon\rho$,
\begin{align*}
G(\mu+\varepsilon\rho\lVert \nu) 
&\geq \int_S g^*d(\mu+\varepsilon\rho) - \log\int_S e^{g^*} d\nu\\
&= \varepsilon\int_S g^*d\rho + \int_S g^*d\mu-\log\int_S e^{g^*}d\nu\\
&=\varepsilon \int_S g^*d\rho + G(\mu\lVert\nu). 
\end{align*}
Thus
\begin{equation}
    \liminf_{\varepsilon\to 0^+} \frac{1}{\varepsilon}\left(G(\mu+\varepsilon \rho\lVert \nu)-G(\mu\lVert\nu)\right)\geq\int_S g^* d\rho.
    \label{eqn:LB}
\end{equation}

The other direction is more delicate. Take $f(\varepsilon) = G(\mu+\varepsilon\rho\lVert\nu)$. From Lemma \ref{basic} we know that $f$ is convex, lower semicontinuous and finite on $[0,\varepsilon_0]$. Using a property of convex functions in one dimension, we know $f$ is differentiable on $(0,\varepsilon_0)$ except for a countable number of points. Take $\varepsilon\in(0,\varepsilon_0)$ to be a place where $f$ is differentiable, and $\delta>0$ small. Take $g^*_\varepsilon\in \mbox{Lip}(c,S)$ to be the optimizer for $G(\mu+\varepsilon\rho\lVert\nu)$ satisfying $g^*_\varepsilon(0)=0$, so that
$$G(\mu+\varepsilon\rho\lVert\nu)=\int_S g^*_\varepsilon d(\mu+\varepsilon\rho) - \log\int_S e^{g^*_\varepsilon} d\nu.$$
Then using an argument that already appeared in this proof, we have 
$$G(\mu+(\varepsilon+\delta)\rho\lVert\nu) - G(\mu+\varepsilon\rho\lVert\nu )\geq \delta\int_S g^*_\varepsilon d\rho,$$
and 
$$G(\mu+(\varepsilon-\delta)\rho\lVert\nu) - G(\mu+\varepsilon\rho\lVert\nu)\geq -\delta\int_S g^*_\varepsilon d\rho.$$
It follows that 
\begin{align*}
\int_S g^*_\varepsilon d\rho&\leq \lim_{\delta\to 0} \frac{1}{\delta} \left(G(\mu+(\varepsilon+\delta)\rho\lVert\nu) - G(\mu+\varepsilon\rho\lVert\nu ) \right)\\
&= f'(\varepsilon)\\
&= \lim_{\delta\to 0}\frac{1}{\delta} \left(G(\mu+\varepsilon\rho\lVert\nu) - G(\mu+(\varepsilon - \delta)\rho\lVert\nu ) \right) \\
&\leq \int_S g^*_\varepsilon d\rho.
\end{align*}
and therefore 
\begin{equation}
    f'(\varepsilon) = \int_S g^*_\varepsilon d\rho.\label{eqn:fderiv}
\end{equation}

If we denote 
$$f'_+(0) = \lim_{\varepsilon\to 0^+} \frac{1}{\varepsilon}(f(\varepsilon)-f(0)),$$
then by a property of convex functions \cite[Theorem 24.1]{roc},
for any sequence of $\left\{\varepsilon_n\right\}_{n\in\mathbb{N}}$ such that $\varepsilon_0 >  \varepsilon_n\downarrow 0$ and $f$ is differentiable at $\varepsilon_n>0$, we have 
$$f'_+(0)  = \lim_{n\to\infty}f'(\varepsilon_n) = \lim_{n\to\infty} \int_S g^*_{\varepsilon_n}d\rho.$$
By the same argument used in the proof of Theorem \ref{optimizer} (paragraphs following Lemma \ref{opt_integrability}), i.e., by applying the Arzel\`{a}-Ascoli theorem to $\{g_{\varepsilon_n}\}$ on each compact set $K_m\subset S$, and then doing a diagonalization argument, 
there exists a subsequence of $\left\{n_k\right\}_{k\geq 0} \subset \left\{n\right\}_{n\geq 0}$, such that $g_{\varepsilon_{n_k}}^*$ converges pointwise to a function that we denote by $g_0^*\in\mathrm{Lip}(c,S)$. To simplify the notation, let $n$ denote the convergent subsequence.

Since $\rho = \rho_+ - \rho_-$, where $\rho_+\in L^1(a)$ and $\mu+\varepsilon_0 \rho\in P(S)$, $\mu\in L^1(a)$, we have
$$0\leq \int a d(\mu+\varepsilon_0 \rho) = \int a d(\mu + \varepsilon_0\rho_+ - \varepsilon_0\rho_-).$$
Thus 
$$\int a \rho_- \leq \frac{1}{\varepsilon_0} (\int a d\mu + \varepsilon_0\int a d\rho_+) < \infty,$$
which implies $\rho_-\in L^1(a)$. Therefore
$$\int_S a d|\rho| < \infty.$$
Here $|\rho|=\rho_+ + \rho_-$. Recall that for any $\varepsilon\in(0,\varepsilon_0)$, $g^*_\varepsilon(0)=0$. For any $x\in S$,
$$g^*_\varepsilon(x)\leq g^*_{\varepsilon}(0) + c(0,x) \leq a(0) + a(x).$$
Thus by the dominated convergence theorem
$$f'_+(0) = \lim_{n\to\infty} \int _S g^*_{\varepsilon_{n}} d\rho = \int_S g^*_0 d\rho.$$

Lastly, to connect $g_0^*$ back to $g^*$ defined in (\ref{the_opt}), note that by the lower semicontinuity of $G(\cdot\lVert\nu)$,
\begin{align*}
G(\mu\lVert\nu)&\leq \liminf_{n\to\infty} G(\mu+\varepsilon_{n}\rho\lVert\nu)\\
& = \liminf_{n\to\infty}\left( \int_S g^*_{\varepsilon_{n}}d(\mu+\varepsilon_{n}\rho) - \log\int_S e^{g^*_{\varepsilon_{n}}}d\nu\right)\\
& = \liminf_{n\to\infty} \int_S g^*_{\varepsilon_{n}}d(\mu+\varepsilon_{n}\rho) - \limsup_{n\to\infty}\log\int_S e^{g^*_{\varepsilon_{n}}}d\nu\\
&\leq \int_S g^*_0 d\mu -\log\int_S e^{g_0^*}d\nu \\
&\leq G(\mu\lVert\nu).
\end{align*}
The second inequality uses dominated convergence, \eqref{eqn:fderiv},
and that by Fatou's lemma
$$\limsup_{n\to\infty} \int_S e^{g^*_{\varepsilon_{n}}}d\nu\geq\liminf_{n\to\infty} \int_S e^{g^*_{\varepsilon_{n}}}d\nu \geq \int_S e^{g^*_0}d\nu.$$
The third inequality uses Lemma \ref{beyond_bounded}.

Since both sides of the inequality coincide, $g^*_0$ must be the optimizer for variational expression (\ref{Vari}). By Theorem \ref{optimizer} and  equation (\ref{biggest}),
we have $g^*_0(x) \leq g^*(x)$ for all $x\in S$.
Thus
\begin{equation}
 f'_+(0)=\int_Sg^*_0d\rho \leq \int_Sg^*d\rho,  \label{eqn:UB}
\end{equation}
the other direction of the inequality is proved. 
Combining \eqref{eqn:UB} and \eqref{eqn:LB} gives 
$$\lim_{\varepsilon\to 0^+} \frac{1}{\varepsilon}\left(G(\mu+\varepsilon\rho\lVert\nu) - G(\mu\lVert\nu)\right)=\int_S g^*d\rho.$$
\end{proof}

\begin{remark}
When $\rho\in\mathcal{M}_0(S)$ is taken such that there exists $\varepsilon_0>0$ such that for $\varepsilon\in[-\varepsilon_0,\varepsilon_0]$, $\mu+\varepsilon\rho\in P(S)$, then by applying the above theorem to $\rho$ and $-\rho$ respectively, we can conclude $G(\mu+\varepsilon\rho\lVert\nu)$ as a function of $\varepsilon$ is differentiable at $\varepsilon = 0$ with derivative $\int_S g^*d\rho$.
\end{remark}

\begin{remark}
We call $g^*$ defined in (\ref{the_opt}) the unique potential associated with $G(\mu\lVert\nu)$. This $g^*$ is similar to the Kantorovich potential in the optimal transport literature. However, for the optimal transport cost $W_c(\mu,\nu)$ more conditions are needed(e.g. \cite{san1}[Proposition 7.18]) to ensure the uniqueness of the Kantorovich potential. Here under very mild conditions we are able to confirm the uniqueness of the potential, and prove that it is the directional derivative of the corresponding $\Gamma$-divergence, as is case of the Kantorovich potential for optimal transport cost when its uniqueness is established.  
\end{remark}

\section{Examples}
In this section, we present some explicit examples where we can compute the exact $\Gamma$ divergence. Throughout the exploration, we investigate the interaction between relative entropy and optimal transport cost within the outcome of $\Gamma$ divergence, and exploit some intuition based on the examples. In Section \ref{elementary_egs}, we investigate two elementary examples, where $\mu$ and $\nu$ are two continuous probability distribution while having different support. In Section \ref{general_egs}, we consider generalization of examples in Section \ref{elementary_egs}, as well as examples where $\mu$ and $\nu$ are discrete probability measures. Section \ref{more_eqs} considers even more generalization of examples in Section \ref{general_egs}. Section \ref{other_directions} discusses potential other directions in examples of $\Gamma$-divergence. Section \ref{limits_and_approximations} considers the scaling property of $\Gamma$-divergence when the corresponding set of functions $\Gamma$ changes.

Recall the definition of $G_\Gamma(\mu\lVert\nu)$ from Definition \ref{def:defofV},
\begin{align}\label{definition}
G_{\Gamma}(\mu\lVert\nu)\doteq\sup_{g\in\Gamma}\left\{  \int_{\mathbb{R}} gd\mu-\log\int_{\mathbb{R}}
e^{g}d\nu\right\},
\end{align}
and the alternative representation from Theorem \ref{thm:main},
\begin{align}\label{variational_expression}
G_\Gamma(\mu\lVert\nu) = \inf_{\gamma\in\mathcal{P}(\mathbb{R})}\left\{R(\gamma\lVert\nu) + W_\Gamma(\mu-\gamma)\right\},
\end{align}
where $W_\Gamma(\mu-\gamma) \doteq\sup_{g\in\Gamma}\left\{\int g (d\mu-d\gamma^*)\right\} $.
In this section, we only consider $\Gamma$ of the form $\mathrm{Lip}(c,S;C_b(S))$, where technical conditions of Theorem \ref{optimizer} are satisfied. We use Theorem \ref{optimizer} to verify the optimizer pair $(\gamma^*,g^*)$ for the variational expressions above, i.e., $(\gamma^*, g^*)$ is the optimizer for (\ref{variational_expression}) if and only if the following two rules hold.\\
(i) $g^*\in \mathrm{Lip}(c,S)$ and for $x\in\mathrm{supp}(\nu)$,
\begin{align}\label{verification_conditon_1}
\frac{d\gamma^*}{d\nu}(x) = \frac{e^{g^*(x)}}{\int e^{g^*(y)}\nu(dy)}.
\end{align}
(ii)  
\begin{align}\label{verification_condition_2}
W_\Gamma(\mu-\gamma^*) = \int g^* (d\mu-d\gamma^*).
\end{align}
For the following within this section, with an abuse of notation, we denote $W_\Gamma(\mu,\gamma) \doteq W_\Gamma(\mu-\gamma)$. Notice for admissible $\Gamma$, as defined in Definition \ref{access}, $g\in\Gamma$ implied $-g\in\Gamma$, which tells us that $W_\Gamma(\mu,\gamma) = W_\Gamma(\gamma,\mu)$. 

The author wants to mention here that we do not claim to have a formula to get the optimizer for any pair of $\mu$, $\nu$ given fixed $\Gamma$, rather under some cases based on intuition, we can "guess" the form of right $\gamma^*$ and $g^*$, and then verify they are indeed the optimizers. The intuition comes from the following observations. First, when $\Gamma = \mathrm{Lip}(c)$,  where $c: X\times X \to \mathbb{R}_+$ is a metric function on $X$, $W_\Gamma(\mu-\gamma)$ can be interpreted as the optimal transport cost between $\mu$ and $\gamma$ with cost function $c$. Let's consider the second variational expression of $G_\Gamma(\mu\lVert\nu)$ in (\ref{variational_expression}), from which one can interpret $G_{\Gamma}(\mu\left\Vert
\nu\right.  )$ as a two step procedure together with optimization. First move
mass from $\nu$ to $\gamma$ and pay a relative entropy cost, then move mass
from $\gamma$ to $\mu$ and pay the optimal transport cost, and finally
optimize over the intermediate measure $\gamma$ to get $\gamma^*$. Although one does not usually
interpret $R(\gamma\left\Vert \nu\right.  )$ in terms of \textquotedblleft
moving mass,\textquotedblright\ we will find it convenient to do so here,
since the properties of $G_{\Gamma}(\mu\left\Vert \nu\right.  )$ will reflect
how the two very different mechanisms provided by relative entropy and optimal
trasport interact to move mass \textquotedblleft cheaply\textquotedblright%
\ from $\nu$ to $\mu$. 

Note that relative entropy must move first, and that there is always the
absolute continuity restriction $\gamma\ll\nu$. One can interpret that
relative entropy will first rearrange mass subject to this constraint at
relative entropy cost, and then hand a new distribution $\gamma^{\ast}$ (the
optimizer) off to optimal transport for the final rearrangement. In the two
stages there are very different mechanisms at work. In particular, we note
that optimal transport is (by definition) sensitive to the distance that the
mass must travel (or more generally the cost it must incur) in moving a bit of
mass from one point to another. This contrasts sharply with relative entropy,
which is in a certain sense completely indifferent to any distance that mass
must travel when reshaping $\nu$ into $\gamma^{\ast}$. This is a crucial
point. It says when relative entropy is rearranging $\nu$ into $\gamma^{\ast}$
prior to handing off to optimal transport to finish the job, it can (and
indeed will) anticipate the distance sensitive nature of optimal transport.
This point will be made more precise as we explore various examples.

Before we proceed to examples, we use an expression of relative entropy when $\gamma^*\ll\mu$, which is an alternative but equivalent version of the expression (\ref{RE_exp}):

$$R(\gamma^*\lVert\nu) = \int_X \log\left(\frac{d\gamma^*}{d\nu}(x)\right)\frac{d\gamma^*}{d\nu}(x) \nu(dx).$$
It is a function of the Radon-Nikodym derivative of two distributions, which does not depend on the relative location of different points. Also, for optimal transport cost $W_\Gamma$, we have a dual representation of $W_\Gamma$ as 
\begin{align}\label{Mass_transport_characterization}
W_\Gamma(\mu,\gamma^*) = \inf_{\pi\in\Pi(\mu,\gamma^*)}\left\{\int_{X\times X} c(x,y)\pi(dxdy)\right\},
\end{align}
where $\Pi(\mu,\gamma^*)\doteq \{\pi\in \mathcal{P}(X\times X): \pi_x = \mu, \pi_y = \gamma^*\}$, $\pi_x$ and $\pi_y$ denotes the first the second marginal of $\pi$ respectively. Under mild condition, we can get the existence of $\pi^* \in \Pi(\mu,\gamma^*)$, such that 

$$W_\Gamma(\mu,\gamma^*) = \int_{X\times X} c(x,y) \pi^*(dxdy).$$
Notice that we also have from (\ref{verification_condition_2}),
\begin{align*}
W_\Gamma(\mu,\gamma^*)  &= \int_X g^* d(\mu - \gamma^*) \\
&= \int_X g^*(x) \mu(dx) - \int_X g^*(x) \gamma^*(dx)\\
& = \int_{X\times X}  g^*(x) \pi^*(dxdy) - \int_{X\times X} g^*(y) \pi^*(dxdy)\\
& = \int_{X \times X} (g^*(x) - g^*(y)) \pi^*(dxdy)\\
&\leq \int_{X\times X} c(x,y)\pi^*(dxdy)\\
& = W_\Gamma(\mu,\gamma^*).
\end{align*}
Since both ends match, the only inequality in this long expression must be equal. Thus we have for $(x,y) \in \mathrm{supp}(\pi^*)$,
\begin{align*}\label{mass_transfer_condition}
g^*(x) - g^*(y)=c(x,y).
\end{align*}
This piece of information will turn out to be valuable when making the guess for $g^*$ and $\gamma^*$ especially when the optimal coupling of $\pi^*$ of $(\mu,\gamma^*)$ is easy to guess. This can be seen in for examples in Section \ref{elementary_egs}.


In the next section, we focus on the case where space $X=\mathbb{R}$, and $\Gamma=\mathrm{Lip}(1;C_b(\mathbb{R}))$, which is the set of bounded Lipshitz functions with respect to $c(x,y) = |x-y|$ and Lipshitz constant 1. We also denote $\mathrm{Lip}(1)$ as the set of Lipshitz functions (not necessarily bounded) with Lipschitz constant 1 on $\mathbb{R}$. Before we proceed, we need a theorem to compute Wasserstein distance explicitly for this choice of $\Gamma$.

\begin{theorem}\cite{val}\label{1d_wass_computation}
For two probability distribution $P,Q\in \mathcal{P}(\mathbb{R})$, 

$$W_\Gamma(P,Q) = \int_{-\infty}^\infty |F(x) - G(x)|dx, $$
where $F$ and $G$ are the cumulative distribution function (c.d.f.) of the distributions $P$ and $Q$, respectively.
\end{theorem}

\subsection{Two elementary examples}\label{elementary_egs}
Let's first use two special examples to have a taste on how $G_\Gamma(\mu\lVert\nu)$ behaves compared to relative entropy and optimal transport cost alone.
\begin{example}\label{first_example}
$\mu = \mathrm{Unif}[0,1+c], \nu = \mathrm{Unif}[0,1]$ and $\Gamma = \mathrm{Lip}(1;C_b(\mathbb{R}))$.
\end{example}
Since the optimizing intermediate measure $\gamma^*$ must have the same support as $\nu$, by considering mass transfer from $\gamma^*$ to $\mu$, there should be "mass" transport from left to right. We want to make use of Theorem \ref{optimizer} to guess the optimizing $g^*$ and $\gamma^*$. First notice that from (\ref{definition}), we have 

\begin{align*}
G_{\Gamma}(\mu\lVert\nu)\doteq\sup_{g\in\Gamma}\left\{  \int_{\mathbb{R}} gd\mu-\log\int_{\mathbb{R}}
e^{g}d\nu\right\}.
\end{align*}
In this example, we can write explicitly for $G_\Gamma(\mu\lVert\nu)$ as 

\begin{align*}
G_\Gamma(\mu\lVert\nu) &= \sup_{g\in\Gamma}\left\{\int_0^{1+c} g(x) \mu(dx) - \log \int_0^1 e^{g(x)} \nu(dx)\right\}  \\
& = \sup_{g\in\Gamma} \left\{\frac{1}{1+c}\int_0^{1+c} g(x) dx - \log\int_0^1 e^{g(x)} dx\right\}.
\end{align*}
Since the value of $g$ in $[1,1+c]$ is only present in the first integral, for the optimizing $g^*$ for this variational problem, we will have $g^*(x) = g^*(1) + (x-1)$ for $x\in[1,1+c]$. From the intuition that relative entropy, considered as a cost for transferring the first measure into the second one, can relocate the mass without considering how far the mass has been transported, while optimal transport cost is sensitive to the distance of mass being transferred, we guess that $\gamma^*$ will allocate as much mass as possible to the right of interval $[0,1]$, up to the constraint (\ref{verification_conditon_1}), while remaining the same as $\mu$ within the left side of the interval $[0,1]$. Thus we guess there exists $b\in[0,1]$ such that

\begin{equation*}
g^*(x)=\left\{
\begin{aligned}
& 0 & 0\leq x\leq b,\\
& x-b & b< x \leq 1+c. \\
\end{aligned}
\right.
\end{equation*}
and 
\begin{equation*}
\gamma^*(dx)=\left\{
\begin{aligned}
& \mu(dx) & 0\leq x\leq b,\\
& e^{x-b}\mu(dx) & b< x \leq 1. \\
\end{aligned}
\right.
\end{equation*}
Here $g^*$ is only determined up to constant addition, so for simplicity, we fix the value of $g^*$ at $0$ as $0$. In order to make $\gamma^*$ a probability distribution, we need to have 

\begin{align*}
1 &= \int_0^1 \gamma^*(dx) = \int_0^b \gamma^*(dx) + \int_b^1 \gamma^*(dx)\\
& = \frac{b}{1+c} + \frac{1}{1+c}\int_b^1 \exp(x-b) dx\\
& = \frac{b}{1+c} + \frac{\exp(1-b) - 1}{1+c}.
\end{align*}
Thus $b$ solves
\begin{align}\label{condition_b}
\exp(1-b) - (1-b) = 1 + c.    
\end{align}
When $c=0$, $b=1$ solves the equation. When $c>0$ is small, $1-b$ is also close to $0$. From (\ref{condition_b}), we can show that $1-b$ can be written as an analytic function of $\sqrt{c}$ around $0$ by the following argument. By taking Taylor expansion for $\exp(1-b)$ in (\ref{condition_b}), we get
\begin{align}\label{condition_b_expansion}
(1-b)^2 F(1-b) = c,
\end{align}
where $F(x)$ is an analytic function, with $F(0)\neq 0$. Since we want to solve for the solution of (\ref{condition_b}) with $b<1$, by taking square root on both sides of (\ref{condition_b_expansion}), and denoting $z= 1-b$, $w = \sqrt{c}$, $G(x) = \sqrt{F(x)}$, we get
$$z G(z) = w.$$
Since $\frac{d}{dz} (zG(z))|_{z=0} = G(0) = \sqrt{F(0)} \neq 0$, by inverse function for analytic functions we know $z$ can also be written as an analytic function of $w$ around point $0$. Thus we have shown $1-b$ can be written as an analytic function of $\sqrt{c}$ around $0$. By writing out $1-b = a_0 + a_1 \sqrt{c} + a_2 c + O(c^{3/2})$, plugging this expression back in (\ref{condition_b}) and matching the coefficients, we can solve

\begin{align}\label{expansion_for_b}
1-b = \sqrt{2}\sqrt{c} - \frac{1}{3} c + O(c^{3/2})
\end{align}

Now we have shown for choosing $b$ satisfying (\ref{condition_b}), $\gamma^*$ is a probability distribution. Now we verify $(g^*$, $\gamma^*)$ pair satisfies (\ref{verification_conditon_1}) and (\ref{verification_condition_2}).\\
(i) To check for (\ref{verification_conditon_1}),
\begin{equation*}
\frac{d\gamma^*}{d\nu}(x)=\left\{
\begin{aligned}
& \frac{1}{1+c} & 0\leq x\leq b,\\
& \frac{1}{1+c} e^{x-b} & b< x \leq 1+c. \\
\end{aligned}
\right.
\end{equation*}
Notice that  
\begin{align*}
\int_0^{1} e^{g^*(x)} \nu(dx) &= b + \int_b^1 \exp(x-b) dx \\
&= b + \exp(1-b) - 1\\
& = 1+c.    
\end{align*}
where the last equation is because of the choice (\ref{condition_b}). So we can conclude for $x\in\mathrm{supp}(\nu) = [0,1]$,
$$\frac{d\gamma^*}{d\nu}(x)  = \frac{e^{g^*(x)}}{\int_\mathbb{R}e^{g^*(y)}\nu(dy)}.$$
(ii) To verify (\ref{verification_condition_2}), by Theorem \ref{1d_wass_computation}, we can compute 
\begin{align*}
W_\Gamma(\mu,\gamma^*) = \int_0^{1+c} |F_\mu(x) - F_{\gamma^*}(x)|dx.
\end{align*}
Here $F_\mu$ and $F_{\gamma^*}$ are c.d.f.'s for $\mu$ and $\gamma^*$ respectively. Notice that since for $x\in[0,1+c]$, $F_\mu(x) \leq F_{\gamma^*}(x)$, we have 

\begin{align*}
W_\Gamma(\mu,\gamma^*) &= \int_0^{1+c} |F_\mu(x) - F_{\gamma^*}(x)|dx\\
& = \int_0^{1+c} F_{\gamma^*}(x) - F_\mu(x) dx\\
& = \int_0^{1+c} \left(\int_0^{1+c} 1_{y\leq x} \gamma^*(dy) - \int_0^{1+c} 1_{y\leq x} \mu(dy)\right) dx\\
&= \int_0^{1+c} \left(\int_0^{1+c} 1_{y\leq x} \left(\gamma^*(dy) -  \mu(dy)\right)\right) dx\\
& = \int_0^{1+c} \left(\int_0^{1+c} 1_{x\geq y} dx\right) \left(\gamma^*(dy) - \mu(dy)\right)\\
& = \int _0^{1+c} (1+c-y) \left(\gamma^*(dy) - \mu(dy)\right)\\
& = \int_0^{1+c} (y-1-c) \left(\mu(dy) - \gamma^*(dy)\right)\\
& = \int_b^{1+c} (y-1-c) \left(\mu(dy) - \gamma^*(dy)\right)\\
& = \int_b^{1+c} g^*(y) \left(\mu^*(dy) - \gamma^*(dy)\right)\\
& = \int_0^{1+c} g^*(y) \left(\mu^*(dy) - \gamma^*(dy)\right).
\end{align*}
where the third last line and last line is because $\mu(dy) = \gamma^*(dy)$ for $x\in[0,b]$, and the second last line is because both $\mu$ and $\gamma^*$ are probability distributions, so adding constant to the function to be integrated will not change the value of the integral. \\
\\
Now both conditions (i) and (ii) are verified, thus $(g^*,\gamma^*)$ pair are the optimal choice. We can compute $\Gamma$ divergence between $\mu$ and $\nu$.

\begin{align*}
G_\Gamma(\mu\lVert\nu) &= R(\gamma^* \lVert \nu) + W_\Gamma(\mu,\gamma^*)\\
& = \int_0^1 \log \left(\frac{d\gamma^*}{d\nu}(x)\right)\gamma^*(dx) + \int_0^{1+c} g^*(x) (\mu-\gamma^*)(dx)\\
& = \int_0^b \log\left(\frac{1}{1+c}\right)\frac{1}{1+c}dx + \int_b^1 \left(\log\left(\frac{1}{1+c}\right) + x-b\right) \frac{\exp(x-b)}{1+c}dx \\
& + \int_b^{1+c}(x-b) \frac{1}{1+c} dx - \int_b^1 (x-b) \frac{\exp(x-b)}{1+c}dx\\
& = \log\left(\frac{1}{1+c}\right) + \frac{1}{2(1+c)}(1+c-b)^2
\end{align*}
To analyze the above quantity and compare with Wasserstein distance, we have 

$$W_\Gamma(\mu,\nu) = \sup_{h\in\Gamma}\left\{\int h d\mu - \int h d\nu\right\} = \int x d\mu - \int x d\nu = \frac{1}{2} (\frac{1}{1+c} (1+c)^2 - 1) = \frac{c}{2}.$$
For small $c$, from (\ref{expansion_for_b}), we have 

\begin{align*}
G_\Gamma(\mu\lVert\nu) &= -\log(1+c) + \frac{1}{2(1+c)}(c + \sqrt{2c} - \frac{1}{3}c + O(c^{3/2}))^2\\
& = - c + O(c^2) + \frac{1}{2}(1-c+O(c^2)) (2c + \frac{4\sqrt{2}}{3} c^{3/2} + O(c^2))\\
& = \frac{2\sqrt{2}}{3} c^{3/2} + O(c^2).
\end{align*}
Thus we can conclude for small $c>0$, $G_\Gamma(\mu\lVert\nu)$ goes to $0$ much faster than $W_\Gamma(\mu,\nu)$. 
The last part is to compute what's the portion of relative entropy and Wasserstein in the composition of $\Gamma$ divergence. Notice that

\begin{align*}
R(\gamma^*\lVert\nu) &= \int_0^b \log\left(\frac{1}{1+c}\right)\frac{1}{1+c}dx + \int_b^1 \left(\log\left(\frac{1}{1+c}\right) + x-b\right) \frac{e^{(x-b)}}{1+c}dx \\
& =\frac{b}{1+c}\log\left(\frac{1}{1+c}\right) + \frac{e^{(1-b)} - 1}{1+c}\log\left(\frac{1}{1+c}\right)\\ &\quad + \frac{1}{1+c}\left((1-b)e^{(1-b)} - e^{(1-b)} + 1\right)\\
& = \log\left(\frac{1}{1+c}\right) + \frac{1}{1+c}\left((1-b)e^{(1-b)} - e^{(1-b)} + 1\right)\\
& =\log\left(\frac{1}{1+c}\right) + \frac{1}{1+c}\left((1-b)(c+1+1-b) - (c+1+1-b) +1\right)\\
&= \log\left(\frac{1}{1+c}\right) + \frac{1}{1+c}\left((1-b)^2+c(1-b) - c\right).
\end{align*}
where the third and fourth equations use the fact that $\exp(1-b) - 1 = c+1+1-b$. Doing the expansion for the last line, and using (\ref{expansion_for_b}), we have 

\begin{align*}
R(\gamma^*\lVert\nu) &=  -\log(1+c) + \frac{1}{1+c}\left(( \sqrt{2c} - \frac{1}{3}c)^2 +c( \sqrt{2c} - \frac{1}{3}c )) -c+O(c^2))\right)  \\
&= - c + O(c^2) + (1-c + O(c^2))\cdot \left(c+\frac{\sqrt{2}}{3}c^{3/2} + O(c^2)) \right)\\
& = \frac{\sqrt{2}}{3}c^{3/2} + O(c^2).
\end{align*}
Similarly, one will get 
$$W_\Gamma(\mu,\gamma^*) = \frac{\sqrt{2}}{3} c^{3/2} + O(c^2).$$
However, if we compute $W_\Gamma(\gamma^*,\nu)$, with the help of Theorem \ref{1d_wass_computation}, we will get

\begin{align*}
W_\Gamma(\gamma^*,\nu) &= \sup_{h\in\Gamma}\left\{\int h d(\gamma^*-\nu)\right\}\\  
&= \int_0^1 -x d\gamma^* + \int_0^1 x d\nu\\
&= \int_0^1 x d\nu - \int_0^1 x d\mu + \int_0^{1+c} x d\mu - \int_0^{1+c} x d\gamma^* \\
& = W_\Gamma(\mu,\nu) - W_\Gamma(\mu,\gamma^*)\\
&= \frac{c}{2} + o(c).
\end{align*}
We notice that $R(\gamma^*\Vert\nu)$ is much smaller than $W_\Gamma(\gamma^*,\nu)$, which is the reason why $G_\Gamma(\mu\lVert\nu)$ is much smaller than $W_\Gamma(\mu,\nu)$. \\
Also notice that in this example, since $\mu\not\ll\nu$, $R(\mu\lVert\nu) = \infty$. So one can't use relative entropy as a measure to measure closeness between $\mu$ and $\nu$ in this example.

\begin{example}\label{second}
$\mu = \mathrm{Unif}[0,1-c]$, $\nu = \mathrm{Unif}[0,1]$, $\Gamma = \mathrm{Lip}(1;C_b(\mathbb{R}))$.
\end{example}
It's worth noticing that this example is similar to Example \ref{first_example}. However, in contrast to Example \ref{first_example} where $R(\mu\lVert\nu)$ is infinite, $R(\mu\lVert\nu)$ is finite in this example. Still, intuition about how relative entropy and optimal transport cost interact from Example \ref{first_example} carries through. 
Similar to the intuition in the first example, we can guess
\begin{equation*}
g^*(x)=\left\{
\begin{aligned}
& 0 & 0\leq x\leq b,\\
& -x+b & b< x \leq 1. \\
\end{aligned}
\right.
\end{equation*}
and 
\begin{equation*}
\gamma^*(dx)=\left\{
\begin{aligned}
& \frac{1}{1-c}dx & 0\leq x\leq b,\\
& \frac{1}{1-c}e^{-x+b}dx & b< x \leq 1. \\
\end{aligned}
\right.
\end{equation*}
To make $\gamma^*$ a probability distribution, we would need 

\begin{align*}
1 &= \int_0^1 \gamma^*(dx) = \frac{b}{1-c} + \frac{1}{1-c}\int_b^1 e^{-x+b}dx\\
&=\frac{b}{1-c} + \frac{1- e^{-1+b}}{1-c}.
\end{align*}
So $b\in(0,1)$ will need to solve 
\begin{align}\label{second_b}
e^{b-1} - 1 - (b-1) = c.
\end{align}
Similar to the solution of (\ref{condition_b}) Example \ref{first_example}, but noticing here $b-1<0$, we can get a similar expression as (\ref{expansion_for_b}) as
\begin{align}\label{expansion2_for_b}
1-b = \sqrt{2}\sqrt{c} + \frac{1}{3}c +O(c^{3/2}).
\end{align}
We now verify $g^*$ and $\gamma^*$ satisfies Theorem \ref{optimizer}.\\
(i) To show it satisfy (\ref{verification_conditon_1}), since $\nu(dx) = dx$ for $x\in[0,1]$, we have 
\begin{equation*}
\frac{d\gamma^*}{d\nu}(x)=\left\{
\begin{aligned}
& \frac{1}{1-c} & 0\leq x\leq b,\\
& \frac{1}{1-c}e^{-x+b} & b< x \leq 1. \\
\end{aligned}
\right.
\end{equation*}
Notice by equation (\ref{second_b}), 
$$\int_0^1 e^{g^*(x)}dx = \int_0^b dx + \int_b^1 e^{-x+b} dx = b + 1 - e^{-1+b} = 1-c.$$ So we can conclude for $x\in\mathrm{supp}(\nu)=[0,1]$,

$$\frac{d\gamma^*}{d\nu}(x) = \frac{e^{g^*(x)}}{\int_0^1 e^{g^*(y)}dy}.$$
(ii) To show it satisfies (\ref{verification_condition_2}), by Theorem \ref{1d_wass_computation}, we can get
\begin{align*}
W_\Gamma(\mu,\gamma^*)& = \int_0^1 |F_\mu(x) - F_\gamma^*(x)|dx\\
&= \int_0^1 F_\mu(x) - F_{\gamma^*}(x) dx\\
&= \int_0^1 \left(\int_0^1 1_{y\leq x} \mu(dy) - \gamma^*(dy)\right) dx\\
& = \int_0^1 \left(\int_0^1 1_{x\geq y} dx \right) \mu(dy) - \gamma^*(dy)\\
& = \int_0^1 (-y+1) \mu(dy) - \gamma^*(dy)\\
& = \int_0^1 g^*(y) \mu(dy) - \gamma^*(dy).
\end{align*}
\\
Here the second line follows from the fact that $F_\mu(x)\geq F_{\gamma^*}(x)$ for any $x\in[0,1]$ and the fourth line follows from Fubini's rule. Now by Theorem \ref{verif}, $(g^*,\gamma^*)$ is the optimizer pair for $G_\Gamma(\mu\lVert\nu)$. Similar to Example \ref{first_example}, we compute $G_\Gamma(\mu\lVert\nu)$ here. 
\begin{align*}
G_\Gamma(\mu\lVert\nu) &= \int g^* d\mu - \log \int e^{g^*} d\nu\\
&= \int_b^{1-c} (-x+b) \frac{1}{1-c}dx - \log \left(\int_0^b 1 dx + \int_b^1 e^{-x+b}dx\right)\\
& = -\frac{(1-c-b)^2}{2(1-c)} - \log(b+1- e^{b-1})\\
& = -\frac{(1-c-b)^2}{2(1-c)} - \log(1-c).
\end{align*}
The last equation comes from (\ref{second_b}). By putting the expansion (\ref{expansion2_for_b}) into the expression above, we can get 
$$G_\Gamma(\mu\lVert\nu)  = \frac{2\sqrt{2}}{3} c^{3/2} + O(c^2).$$
Compared to $W_\Gamma(\mu-\nu)$ which can be computed as 
$$W_\Gamma(\mu-\nu) = \sup_{h\in\Gamma} \left\{\int h d\mu - \int h d\nu\right\} = \int -x d\mu - \int -x d\nu = \frac{c}{2}.$$
So for small $c>0$, $G_\Gamma(\mu\lVert\nu)$ is much smaller than $W_\Gamma(\mu,\nu).$ One can also compute $R(\mu\lVert\nu)$ in this example to get
\begin{align*}
R(\mu\lVert\nu) &= \int \log\left(\frac{d\mu}{d\nu}\right)d\mu\\
& = \int_0^{1-c} \log \left(\frac{1}{1-c}\right)\frac{1}{1-c} dx\\
& = \log\left(\frac{1}{1-c}\right)\\
& = -\log(1-c) = c + o(c).
\end{align*}
From which we can also see $G_\Gamma(\mu\lVert\nu)$ is also much smaller than $R(\mu\lVert\nu)$.

Lastly, for the decomposition of $G_\Gamma(\mu\lVert\nu)$ as relative entropy, we can also compute the corresponding value.
\begin{align*}
R(\gamma^*\lVert\nu) &= \int g^* d\gamma^* - \log\int e^{g^*} d\nu\\
&= \int_b^1 (-x+b) \frac{1}{1-c} e^{-x+b} dx - \log(1-c)\\
& = (2-b)e^{-(1-b)} - 1 - \log(1-c)\\
& = \frac{\sqrt{2}}{3}c^{3/2} + O(c^2).
\end{align*}
Here the last line we use (\ref{second_b}), the expansion (\ref{expansion2_for_b}) and  Taylor expansion for logarithemic function. By $W_\Gamma(\mu-\gamma^*) = G_\Gamma(\mu\lVert\nu) - R(\gamma^*\lVert\nu)$, we can conclude 
$$W_\Gamma(\mu,\gamma^*) = \frac{\sqrt{2}}{3}c^{3/2} + O(c^2).$$
Similar computation as Example \ref{first_example} can derive that
$$W_\Gamma(\gamma^*,\nu) = \frac{c}{2} + o(c).$$
It's interesting to notice both $G_\Gamma(\mu\lVert\nu)$ and its relative entropy and optimal transport cost component are of the same behavior as the first example, although there is no direct symmetry between these two examples. And in both examples relative entropy plays a big role to make $G_\Gamma(\mu\lVert\nu)$ much smaller than $W_\Gamma(\mu,\nu)$. In this example we can also see with the help of optimal transport cost, $G_\Gamma(\mu\lVert\nu)$ is also much smaller than $R(\mu\lVert\nu)$.
\subsection{More General Examples}\label{general_egs}
In the following, we consider more general examples. Example \ref{example_same_density_extended_support} is a generalization of Example \ref{first_example}, in which we consider non-constant densities, while still retaining the structure that $\mu$ has a larger support than $\nu$. In Example \ref{discrete_pts_add_point}, we consider discrete measures, where both $\mu$ and $\nu$ are uniform distribution over its support, while $\mathrm{supp}(\mu)$ has exactly one more point than $\mathrm{supp}(\nu)$.
\begin{example}\label{example_same_density_extended_support}
$X = \mathbb{R}$, $\Gamma = \mathrm{Lip}(1;C_b(\mathbb{R}))$. Fix $f:\mathbb{R} \to \mathbb{R}_+$. $\nu(dx) = \frac{f(x)}{\int_0^1 f(y)dy}dx$ for $x\in[0,1]$, and $0$ elsewhere, and let $\mu(dx) = \frac{f(x)}{\int_0^{1+c}f(y)dy} dx$ for $x\in [0,1+c]$, and $0$ elsewhere. 
\end{example}
For this example, we apply the same idea as in Example \ref{first_example}. We know the existence of an optimal intermediate measure $\gamma^*$ and the corresponding function $g^*$. Since $\gamma^*$ has the same support as $\nu$, which is $[0,1]$, for mass of $\mu$ on the interval $[1,1+c]$, it must be transported to points in the support of $\gamma^*$, which tells $g^*(x) = g^*(1) + (x-1)$ for $x\in[1,1+c]$. One would guess that these mass transfer between $\mu$ and $\gamma^*$ happens only on $[b,1+c]$ for some $b\in (0,1)$, in which case $g^*$ and $\gamma^*$ should have the following representation:

\begin{equation*}
g^*(x)=\left\{
\begin{aligned}
& 0 & 0\leq x\leq b,\\
& x-b & b< x \leq 1+c. \\
\end{aligned}
\right.
\end{equation*}
and 
\begin{equation*}
\gamma^*(dx)=\left\{
\begin{aligned}
& \mu(dx) & 0\leq x\leq b,\\
& e^{x-b}\mu(dx) & b< x \leq 1. \\
\end{aligned}
\right.
\end{equation*}
The restriction which needs to be satisfied to make $\gamma^*$ a probability distribution is 
\begin{align*}
1 &= \int_0^1 \gamma^*(dx) = \int_0^b \gamma^*(dx) + \int_b^1 \gamma^*(dx)\\
& = \int_0^b \mu(dx) + \int_b^1 e^{x-b} \mu(dx)\\
& = \frac{\int_0^b f(x) dx}{\int_0^{1+c} f(x) dx} + \frac{\int_b^1 e^{x-b} f(x) dx}{\int_0^{1+c} f(x) dx}.
\end{align*}
After redistribution and cancellation, we conclude 
\begin{align}\label{example_4_condition}
\int_b^1 e^{x-b} f(x) dx = \int_b^{1+c} f(x) dx.    
\end{align}
We denote $H(b) = \int_b^1 e^{x-b} f(x) dx - \int_b^{1+c} f(x) dx$. Since $f(x)>0$ for all $x\in\mathbb{R}$, we have $H(1) = 0 - \int_1^{1+c} f(x)dx <0$, and $H$ is strictly decreasing on $[0,1]$ with 

$$H'(b) = -e^{b-b}f(b) - \int_b^1 e^{x-b}f(x) dx + f(x)= -\int_b^1 e^{x-b}f(x)dx <0.$$
In the following, we separate the cases depending on whether  $H(0)>0$ holds or not.\\
Case 1: $H(0) = \int_0^1 e^x f(x) dx -\int_0^{1+c} f(x) dx$ is positive. We can conclude by intermediate value theorem, there exists $b\in(0,1)$ such that $H(b) = 0$. Let's denote this solution as $b^*$. Now we use Theorem 4.8 to verify the $(\gamma^*,g^*)$ pair with $b^*$ inserted is indeed the optimal choice pair.\\
(i) To check for (\ref{verification_conditon_1}), 
\begin{equation*}
\frac{d\gamma^*}{d\nu}(x)=\left\{
\begin{aligned}
& \frac{d\mu}{d\nu}(x) &= \frac{\int_0^1 f(y)dy}{\int_0^{1+c}f(y)dy} \quad & 0\leq x\leq b^*,\\
& e^{x-b^*}\frac{d\mu}{d\nu}(x)&= e^{x-b^*} \frac{\int_0^1 f(y)dy}{\int_0^{1+c}f(y)dy}\quad & b^*< x \leq 1. \\
\end{aligned}
\right.
\end{equation*}
On the other hand, 
\begin{equation*}
e^{g^*(x)}=\left\{
\begin{aligned}
& 1 & 0\leq x\leq b^*,\\
& e^{x-b^*} & b^*< x \leq 1+c. \\
\end{aligned}
\right.
\end{equation*}
And 
\begin{align*}
\int_0^{1} e^{g^*(x)} \nu(dx) & = \int_0^{b^*} e^{g^*(x)} \nu(dx)  + \int_{b^*}^{1} e^{g^*(x)} \nu(dx)\\
& = \frac{\int_0^{b^*} f(x) dx}{\int_0^1 f(y) dy} + \frac{\int_{b^*}^{1} e^{x-b^*} f(x) dx}{\int_0^1 f(y) dy} \\
&= \frac{\int_0^{b^*} f(x) dx}{\int_0^1 f(y) dy} + \frac{\int_{b^*}^{1+c} f(x) dx}{\int_0^1 f(y) dy} \\
& = \frac{\int_0^{1+c} f(x) dx}{\int_0^1 f(y) dy}.
\end{align*}
Here the second last line uses (\ref{example_4_condition}). Thus, for $x\in [0,1+c]$, 
$$\frac{d\gamma^*}{d\nu}(x) = \frac{e^{g^*(x)}}{\int_0^1 e^{g^*(y)}\nu(dy)}.$$
(ii) To check for (\ref{verification_condition_2}), we use Theorem \ref{1d_wass_computation} to compute $W_\Gamma(\mu,\gamma^*)$. Denote $F_\mu$ and $F_{\gamma^*}$ as the cdf's for $\mu$ and $\gamma^*$ respectively. Notice that $F_\mu(x)\leq F_{\gamma^*}(x)$ for all $x\in [0,1+c]$.
\begin{align*}
W_\Gamma(\mu,\gamma^*) & = \int_0^{1+c} |F_\mu(x) - F_{\gamma^*}(x)| dx\\
& = \int_0^{1+c} (F_{\gamma^*}(x) - F_\mu(x)) dx\\
& = \int_0^{1+c} \left(\int_0^{1+c} 1_{y\leq x} \gamma^*(dy) - \int_0^{1+c} 1_{y\leq x} \mu(dy)\right) dx\\
&= \int_0^{1+c} \left(\int_0^{1+c} 1_{y\leq x} \left(\gamma^*(dy) -  \mu(dy)\right)\right) dx\\
& = \int_0^{1+c} \left(\int_0^{1+c} 1_{x\geq y} dx\right) \left(\gamma^*(dy) - \mu(dy)\right)\\
& = \int _0^{1+c} (1+c-y) \left(\gamma^*(dy) - \mu(dy)\right)\\
& = \int_0^{1+c} (y-1-c) \left(\mu(dy) - \gamma^*(dy)\right)\\
& = \int_{b^*}^{1+c} (y-1-c) \left(\mu(dy) - \gamma^*(dy)\right)\\
& = \int_{b^*}^{1+c} g^*(y) \left(\mu(dy) - \gamma^*(dy)\right)\\
& = \int_0^{1+c} g^*(y) \left(\mu(dy) - \gamma^*(dy)\right).
\end{align*}
Here the fifth line uses Fubini's theorem, the third last line and last line uses the fact $\mu(dx) = \gamma^*(dx)$ for $x\in [0,b^*]$, and the second last line uses the fact that $\mu$ and $\gamma^*$ have the same total mass on $[b^*,1+c]$. \\
Now both (i) and (ii) of Theorem 4.8 are checked, and it's verified that $(g^*,\gamma^*)$ are the optimal pair. Then one can compute $G_\Gamma(\mu\lVert\nu)$ as 
\begin{align*}
G_\Gamma(\mu\lVert\nu) &= \int g^*(x) \mu(dx) - \log\int e^{g^*(x)}\nu(dx) \\
& = \int_{b^*}^{1+c} (x-b^*) \mu(dx) - \log\left(\int_0^{b^*} \nu(dx) + \int_{b^*}^1 e^{x-b^*}\nu(dx)\right)\\
& = \frac{\int_{b^*}^{1+c}(x-b^*)f(x) dx}{\int_0^{1+c} f(x) dx} - \log\left(\frac{\int_0^{b^*}f(x) dx + \int_{b^*}^1 e^{x-b^*} f(x) dx}{\int_0^1 f(x) dx}\right)\\
& = \frac{\int_{b^*}^{1+c}(x-b^*)f(x) dx}{\int_0^{1+c} f(x) dx} - \log\left(\frac{\int_0^{b^*}f(x) dx + \int_{b^*}^{1+c} f(x) dx}{\int_0^1 f(x) dx}\right)\\
& = \frac{\int_{b^*}^{1+c}(x-b^*)f(x) dx}{\int_0^{1+c} f(x) dx} - \log\left(\frac{\int_0^{1+c}f(x) dx }{\int_0^1 f(x) dx}\right),
\end{align*}
where the fourth line uses (\ref{example_4_condition}).\\
Case 2: $H(0) \leq 0$. Then it actually tells us $b=0$, which is the case that for the optimal transport from $\mu$ to $\gamma^*$, the mass of $\mu$ on $[1,1+c]$ is actually transported to $\gamma^*$ over the whole interval $[0,1]$. Thus one would expect there exists some $C_0>0$, such that
\begin{equation*}
g^*(x)= x \quad 0\leq x \leq 1+c,
\end{equation*}
and 
\begin{equation*}
\gamma^*(dx)=
C_0 e^x \nu(dx) \quad 0\leq x \leq 1.
\end{equation*}
To make $\gamma^*$ a probability distribution, we must have $\int \gamma^*(dx) = 1$, which implies $C_0 = \frac{\int_0^1 f(x) dx}{\int_0^1 e^x f(x) dx}$. $H(0)\leq 0$ makes sure that $\frac{d\gamma^*}{d\mu}(x) \geq 1$ for $0\leq x\leq 1$. Similar to Case 1, one can easily verify that this pair of $(g^*, \gamma^*)$ satisfies condition (\ref{verification_conditon_1}) and (\ref{verification_condition_2}).

\begin{example}\label{discrete_pts_add_point}
$X = \mathbb{R}$, $\Gamma = \mathrm{Lip}(1;C_b(\mathbb{R}))$. $N \doteq \{x_1,x_2,\dots, x_n\} \subset X$, where we assume for simplicity that $x_i\neq x_j$ for any $i\neq j$. Let $\nu$ be uniform distribution over $N$. $\mu$ be uniform over $N\cup\{y\}$, where $y\notin N$. Without loss of generality, let's assume $x_1< x_2< \dots < x_n$.
\end{example}
Let's first consider the special case that $y> x_n$. By the intuition that optimal transport cost is "location-sensitive" and relative entropy is not, we guess that $\gamma^*$ is the same as $\mu$ for points far away from $y$, and accumulates more mass for points near $y$. So we guess there exists $k\in\{1,2,\dots,n\}$ such that 

\begin{equation*}
g^*(x)=\left\{
\begin{aligned}
& x-y & x_k\leq x\leq y,\\
& c_0 & x\leq x_{k-1}, \\
\end{aligned}
\right.
\end{equation*}
and 
\begin{equation*}
\gamma^*(x)=\left\{
\begin{aligned}
& \frac{e^{g^*(x)}}{\int e^{g^*(y)}\nu(dy)}\nu(x) & x\in \{x_k, x_{k+1}, \dots, x_n\}, \\
& \mu(x) & x\in \{x_1, x_2, \dots, x_{k-1}\}.\\
\end{aligned}
\right.
\end{equation*}
Here we fix the value of $g^*$ at $y$ as $0$. It's worth noticing that this will introduce $c_0$ as a variable to be determined, which is a little bit different from previous examples. This guess is based on the idea that for optimal transport from $\mu$ to $\gamma^*$, the mass of $\mu$ at the extra point $y$ is spread to $x_j$ where $k\leq j\leq n$. To make this pair $(g^*,\gamma^*)$ satisfy the condition from Theorem \ref{optimizer}, one needs to have \\
In the case k>1, \\
1) $g^* \in \mathrm{Lip}(1)$, where in this example is equivalent to  $|g^*(x_k) - g^*(x_{k-1})| = |x_k - y - c_0| \leq x_k - x_{k-1}$.\\
2) To satisfy (\ref{verification_conditon_1}) for $x\in\{x_1,x_2,\dots,x_{k-1}\}$,
$$\frac{1}{n+1} = \mu(x) = \gamma^*(x)= \frac{e^{g^*(x)}\nu(x)}{\int e^{g^*(y)}\nu(dy) } = \frac{e^{c_0}}{(k-1)e^{c_0}+\sum_{i=k}^n e^{g^*(x_i)}}.$$
After rearrangement, we get
$$c_0 = \log\left(\frac{1}{n+2-k}\sum_{i=k}^n e^{g^*(x_i)}\right) = \log\left(\frac{1}{n+2-k}\sum_{i=k}^n e^{x_i - y}\right).$$
3) To satisfy (\ref{verification_condition_2}), which is the optimal transport constraint, one needs to have $\gamma^*(x_k)\geq \gamma^*(x_{k-1})$, which combined with the last condition implies $c_0 \leq g^*(x_k) = x_k -y$.\\
In the case $k=1$, condition 1) and 2) above are automatically satisfied. Thus, the only requirement left is to satisfy (\ref{verification_condition_2}), which is
$$\gamma^*(x_1) = \frac{e^{x_1 - y}}{\sum_{i=1}^n e^{x_i - y}} \geq \mu(x_1) = \frac{1}{n+1},$$
which after rearrangement becomes
$$(n+1)e^{x_1 - y} \geq \sum_{i=1}^n e^{x_i - y}.$$
One can prove by induction that there exists $k\in\{1,2,\dots,n\}$ such that the above conditions hold. We do the induction as follows.\\
Step 1: If $k=n$ satisfies the conditions above, then we are done. If $k=n$ does not satisfy the condition, we denote the $c_0$ solved from $k=n$ as $c_0^{(n)}$. Then $c_0^{(n)} = \log\left(\frac{1}{2}e^{x_n - y}\right) = x_n - y - \log(2)< x_n - y$. So the only problem that $k=n$ does not satisfy is that $g^*$ defined is not in $\mathrm{Lip}(1)$, which is equivalent to $x_{n-1} - y> c_0^{(n)}$.\\
Step 2: If $k= m>2$ does not satisfy the condition above, then by step 1, we can conclude that $c_0^{(m)} < x_{m-1} - y$. Thus we will have 

\begin{align*}
c_0^{(m-1)} &= \log\left(\frac{1}{n+2-(m-1)}\sum_{i=m-1}^n e^{x_i - y}\right) \\
&= \log\left(\frac{1}{n+2-(m-1)}\left((n+2-m)e^{c_0^{(m)}} + e^{x_{m-1}-y}\right)\right)\\
&> \log\left(\frac{1}{n+2-(m-1)}\left((n+2-m)e^{c_0^{(m)}} + e^{c_0^{(m)}}\right)\right)\\
& = c_0^{(m)},
\end{align*}
and 
\begin{align*}
c_0^{(m-1)} &= \log\left(\frac{1}{n+2-(m-1)}\sum_{i=m-1}^n e^{x_i - y}\right) \\
&= \log\left(\frac{1}{n+2-(m-1)}\left((n+2-m)e^{c_0^{(m)}} + e^{x_{m-1}-y}\right)\right)\\
&< \log\left(\frac{1}{n+2-(m-1)}\left((n+2-m)e^{x_{m-1}-y} + e^{x_{m-1}-y)}\right)\right)\\
& = x_{m-1}-y,
\end{align*}
If $c_0^{(m-1)}$ satisfies $c_0^{(m-1)} > x_{m-2}-y$, then $k=m-1$ will make $(g^*,\gamma^*)$ defined in the beginning of this section the optimal pair we want. If not, then $c_0^{(m-1)}<  x_{m-2} - y$, and we can continue this step with $k=m-1$ as long as $m-1 >2$.\\
Step 3: If $k=2$ still does not satisfy the condition above. We have $c_0^{(2)} < x_1 - y$, which is equivalent to
\begin{align*}
c_0^{(2)} = \log\left(\frac{1}{n}\sum_{i=2}^n e^{x_i-y}\right) &\leq   x_1 - y\\
\Longleftrightarrow\frac{1}{n}\sum_{i=2}^n e^{x_i - y} & \leq  e^{x_1 - y}\\
\Longleftrightarrow\sum_{i=1}^n e^{x_i - y} &\leq (n+1) e^{x_1 - y},
\end{align*}
which tells us $k=1$ satisfies the condition.

Now let's consider the case if $y$ is not on the side, rather there exists $1\leq l <n$ such that $x_l < y < x_{l+1}$. Then through out intuition from the interaction between relative entropy and optimal transport cost, we can consider a new set of points by $\left\{x_1,x_2,\dots,x_l, 2y - x_{l+1}, 2y - x_{l+2},\dots, 2y - x_{n}\right\}$, where for all the point larger than $y$, we replace them with their symmetric point about $y$. Now for this new set of points, $y$ is larger than the largest point of them. We can redo what has been done in the beginning of this example, and notice that by mirroring these point around $y$, the cost for transporting mass from $y$ to these points are unchanged, and relative entropy cost does not care about location, so we get the exact $G_\Gamma(\mu\lVert\nu)$ from the new set of points, and repeat what has done before. 
\begin{remark}
The idea to find the optimal intermediate measure $\gamma^*$ in the last example can actually be carried to higher dimensional settings, where one can order the points according to their distance with respect to $y$. Example \ref{add_pt_multi_dim} discusses this extension.
\end{remark}

\subsection{More examples}\label{more_eqs}
In the following are more examples which are extensions of examples in section \ref{general_egs}. Based on intuition developed so far, we want to discuss how one can compute $\Gamma$ divergence for more general cases. 
\begin{example}[Extension of Example \ref{example_same_density_extended_support}]
$X = \mathbb{R}$, $\Gamma = \mathrm{Lip}(1;C_b(\mathbb{R}))$. Fix $f_1,f_2:\mathbb{R} \to (0,+\infty)$ as two continuous functions which are differentiable almost everywhere in $[0,1]$, such that $\log\left(\frac{f_1}{f_2}\right)$ as a function constrained at interval $[0,1]$, is in $\Gamma$. $\nu(dx) = \frac{f_2(x)}{\int_0^1 f_2(y)dy}dx$ for $x\in[0,1]$, and $0$ elsewhere. $\mu(dx) = \frac{f_1(x)}{\int_0^{1+c}f_1(y)dy} dx$ for $x\in [0,1+c]$, and $0$ elsewhere.
\end{example}
First we notice that if we multiply $f_1$ and $f_2$ by two different constant respectively, $\log\left(\frac{f_1}{f_2}\right)$ will only differ by a constant. Without loss of generality, let's assume $\int_0^{1+c} f_1(y)dy = \int_0^1 f_2(y)dy = 1$. It's worth noted that when $c=0$, $G_\Gamma(\mu\lVert\nu) = R(\mu\lVert\nu)$, since in this case, the only interval that matters is $[0,1]$, and it can be easily verified that $(g^*|_{[0,1]},\gamma^*) = \left(\log\left(\frac{f_1}{f_2}\right)|_{[0,1]},\mu\right)$ is the optimal pair for $\Gamma$-divergence. For $c>0$, similar to the intuition from Example \ref{example_same_density_extended_support}, we guess there exists $b\in[0,1)$, such that 
\begin{equation*}
g^*(x)=\left\{
\begin{aligned}
& \log\left(\frac{f_1(x)}{f_2(x)}\right) & 0\leq x\leq b,\\
& x-b + \log\left(\frac{f_1(b)}{f_2(b)}\right) & b< x \leq 1+c. \\
\end{aligned}
\right.
\end{equation*}
and 
\begin{equation*}
\gamma^*(dx)=\left\{
\begin{aligned}
& \mu(dx) & 0\leq x\leq b,\\
& e^{g^*(x)}\nu(dx) & b< x \leq 1. \\
\end{aligned}
\right.
\end{equation*}
It's obvious $g^* \in \mathrm{Lip}(1)$, and $\gamma^*$ is constructed so that (\ref{verification_conditon_1}) and (\ref{verification_condition_2}) are satisfied if $\gamma^*$ is a probability measure. So the only thing need to be satisfied is $\gamma^*$ being a probability measure, which is 
\begin{align*}
1 &= \int_0^b \mu(dx) + \int_b^1 e^{g^*(x)}\nu(dx)\\
&= \int_0^b f_1(x) dx + \int_b^1 e^{x-b} \frac{f_1(b)}{f_2(b)} f_2(x) dx.
\end{align*}
Notice that $\int_0^{1+c}f_1(x)dx =1$, so the condition that needs to be satisfied is in turn \begin{align}
\int_b^{1+c} f_1(x) dx - \int_b^{1}e^{x-b} \frac{f_1(b)}{f_2(b)}f_2(x) dx = 0.    
\end{align}
Denote $H(b) = \int_b^{1+c} f_1(x) dx - \int_b^{1+c}e^{x-b} \frac{f_1(b)}{f_2(b)}f_2(x) dx.$ Notice 
$$H(1) = \int_1^{1+c} f_1(x) dx > 0.$$ Since $\log\left(\frac{f_1}{f_2}\right)_{[0,1]}\in \mathrm{Lip}(1)$,  for any $x$ and $x+\Delta x\in[0,1)$, $\log\left(\frac{f_1(x+\Delta x)}{f_2(x+\Delta x)}\right)\leq \log\left(\frac{f_1(x)}{f_2(x)}\right) + \Delta x$, which is equivalent to $\frac{f_1(x+\Delta x)}{e^{x+\Delta x}f_2(x+\Delta x)}\leq \frac{f_1(x)}{e^x f_2(x)}$. By taking the difference of two sides of the inequality, divided by $\Delta x$, and taking $\Delta x \to 0$, we get $\left(\frac{f_1(x)}{e^x f_2(x)}\right)'\leq 0$. Thus 
\begin{align*}
H'(b) &= -f_1(b) + e^{b-b} \frac{f_1(b)}{f_2(b)} f_2(b) - \int_b^{1+c}e^x f_2(x) \left(\frac{f_1(b)}{e^b f_2(b)}\right)'dx\\
& = - \int_b^{1+c}e^x f_2(x) \left(\frac{f_1(b)}{e^b f_2(b)}\right)'dx\\
& \geq 0.
\end{align*}
Thus $H(b)$ is an increasing function. So depending on whether $H(0) <0$ hold or not, we will have two different scenarios.\\
1) $H(0)<0$, then by intermediate value theorem and the fact $H$ is monotone, there exists $b^*\in (0,1)$ such that $H(b^*) = 0$. There might be multiple $b$ such that $H(b) = 0$, but these $b$ must be on an interval where $\left(\frac{f_1(x)}{e^x f_2(x)}\right)'= 0$. Thus any $b^*$ in this interval will result in the same $\gamma^*$ and $g^*$, and nonetheless we can just choose $b^*$ as the right end of that interval.\\
2) $H(0)\geq 0$. Then similar to Example \ref{example_same_density_extended_support}, $b^*=0$ with $g^*$ can be specified as $g^*(x) = x$ for $x\in[0,1+c]$ and there will exists $C_0>0$ as a constant such that $\gamma^*(dx) = C_0 e^x\nu(dx)$ holds for $x\in [0,1]$.

In the next example, we will look at an alternate situation compared to Example \ref{discrete_pts_add_point}. When we compare $\nu$ to $\mu$, instead of adding an extra point in the support as in Example \ref{discrete_pts_add_point}, we will take an point away from the support.
\begin{example}[Alternation of Example \ref{discrete_pts_add_point}]\label{discrete_pts_minus_pt}
$X = \mathbb{R}, \Gamma = \mathrm{Lip}(1;C_b(\mathbb{R}))$. \\$N\doteq\{x_1, x_2, \dots, x_n\}\subset X$, where we assume that $x_i\neq x_j$ for $i\neq j$. Without loss of generality we assume that $x_1 < x_2 < \dots < x_n$. $\nu$ is the uniform distribution on $N$, and $\mu$ is the uniform distribution on $N\backslash \{x_j\}$ for some $j\in\{1,2,\dots,n\}$.
\end{example}
Similar to discussion in Example \ref{discrete_pts_add_point}, our intuition tells us that $\gamma^*$ should be the same as $\mu$ for points which are more than certain distance away from $x_j$, and has decreasing mass towards $x_j$ for points which are closer to $x_j$, where the log-likelihood of $\gamma^*$ with respect to $\nu$ has derivative $1$ for these closer points. Here we use a different notation as in Example \ref{discrete_pts_add_point} which can be generalized later on. Let's denote $d: X\times X \to \mathbb{R}$ as the distance function here. In this example $d(x,y)=|x-y|$. Denote $D_R \doteq \{x\in \mathbb{R}: d(x,x_j)\leq R\}$, and guess there exists $R>0$ and $c_0\in\mathbb{R}$ such that $g^*$ and $\gamma^*$ defined on $N$ are
\begin{equation}\label{discrete_minus_pt_g}
g^*(x)=\left\{
\begin{aligned}
& d(x,x_j) & x\in N\cap D_R,\\
& c_0 & x \in N\backslash D_R, \\
\end{aligned}
\right.
\end{equation}
and 
\begin{equation}\label{discrete_minus_pt_gamma}
\gamma^*(x)=\left\{
\begin{aligned}
& \frac{e^{g^*(x)}}{\int e^{g^*(y)}\nu(dy)}\nu(x) & x\in N \cap D_R, \\
& \mu(x) & x\in N\backslash D_R,\\
\end{aligned}
\right.
\end{equation}
These guesses are based on (\ref{verification_conditon_1}) and (\ref{verification_condition_2}). To make them finally satisfy 
these two conditions, we separate into two subcases: $N\backslash D_R \neq \varnothing$ or $N\backslash D_R = \varnothing$. \\
We first consider the case $N\backslash D_R\neq \varnothing.$\\
1) $g\in\Gamma$, where in this example we only need to check for any $x\in N\cap D_R$, and any $y\in N\backslash D_R$,
$$|g^*(x) - g^*(y)|\leq d(x,y).$$
By (\ref{discrete_minus_pt_g}), $d(x,y) = |x-y|$, we conclude that the condition above is implied by
$$\max_{x\in N\cap D_R} d(x,x_j) \leq c_0\leq \min_{x\in N\backslash D_R} d(x,x_j).$$
2) We need for $x\in N\backslash D_R$,
$$\frac{d\gamma^*}{d\nu}(x) = \frac{e^{g^*(x)}}{\int e^{g^*(y)}\nu(dy)}.$$
Together with (\ref{discrete_minus_pt_g}) and (\ref{discrete_minus_pt_gamma}), we have for $x\in N\backslash D_R$,
\begin{align*}
\frac{n}{n-1} &= \frac{d\mu}{d\nu}(x) = \frac{d\gamma^*}{d\nu}(x)\\
& = \frac{e^{g^*(x)}}{\int e^{g^*(y)}\nu(dy)}\\
&=\frac{e^{c_0}}{\left(\sum_{x\in N\cap D_R}e^{d(x,x_j)} + \sum_{x\in N\backslash D_R}e^{c_0}\right)/n}.
\end{align*}
Denote $k = \# (N\cap D_R)$ the number of points in $N\cap D_R$, then by solving the equation above we have 
\begin{align}\label{condition_constant_minus_pt_1d}
c_0 =\log\left( \frac{1}{k-1} \sum_{x\in N\cap D_R} e^{d(x,x_j)}\right).
\end{align}
3) To satisfy (\ref{verification_condition_2}), one needs to have 
$$W_\Gamma(\mu,\gamma^*) = \int g^* d(\mu-\gamma^*).$$
A sufficient condition is 
\begin{align}\label{condition_constant_2_minus_pt_1d}
c_0 > \max_{x\in N\cap D_R} d(x,x_j).
\end{align} 
This is because when (\ref{condition_constant_2_minus_pt_1d}) holds, we can pick a coupling $\pi^*$ of $\mu$ and $\gamma^*$, as $\pi^*(x_j,x) = \mu(x) - \gamma^*(x)$ for any $x\in N\backslash\{x_j\}$, $\pi^*(x,x) = \gamma^*(x)$ for any $x\in N\backslash\{x_j\}$. Notice that $(\ref{condition_constant_2_minus_pt_1d})$ makes sure the definition of $\pi^*$ above is a probability measure. By (\ref{Mass_transport_characterization}) and the cost here between points is $d(x,y) = |x-y|$, we have
\begin{align*}
W_\Gamma(\mu,\gamma^*) &= \inf_{\pi\in\Pi(\mu,\gamma^*)}\left\{\int_{\mathbb{R}\times \mathbb{R}} d(x,y)\pi(dxdy)\right\}\\
&\geq \int_{\mathbb{R}\times \mathbb{R}} d(x,y)\pi^*(dxdy)\\
&= \sum_{x\in N\backslash \{x_j\}}d(x,x_j)(\mu(x_j) - \gamma^*(x_j)).
\end{align*}
On the other hand, from the definition of $W_\Gamma(\mu,\gamma^*)$,
\begin{align*}
W_\Gamma(\mu,\gamma^*) &= \sup_{g\in\Gamma}\left\{\int_{\mathbb{R}} g d(\mu-\gamma^*)\right\}\\
&\leq \int_{\mathbb{R}} g^* d(\mu-\gamma^*)\\
&=\sum_{x\in N\cap D_R}g^*(x)(\mu(x) - \gamma^*(x))\\
&= \sum_{x\in N\backslash\{x_j\}}g^*(x)(\mu(x)-\gamma^*(x)).
\end{align*}
where the last line comes from the fact that for $x\in N\backslash D_R$, $\mu(x) = \gamma^*(x)$, and $g^*(x_j) = d(x_j,x_j) = 0$.\\
In the case $N\backslash D_R = \varnothing$, the conditions from (\ref{discrete_minus_pt_g}) and (\ref{discrete_minus_pt_gamma}) only consist of the first line, so the only sufficient condition here is for any $x\in N\backslash\{x_j\}$,
$$\gamma^*(x)<\mu(x),$$
which reduces to 
$$\max_{x\in N}\gamma^*(x) =\frac{e^{\max_{x\in N}d(x,x_j)}}{\sum_{x\in N}e^{d(x,x_j)}} \nu(x) <\mu(x),$$
which is equivalent to 
\begin{align}\label{discrete_minus_pt_large_R_condtion}
e^{\max_{x\in N}d(x,x_j)} < \frac{n}{n-1}\sum_{x\in N}e^{d(x,x_j)}.
\end{align}
Here we can use mathematical induction by increasing $R$ similar as in Example \ref{discrete_pts_add_point} to show there exists $R>0$ such that the above conditions are met. By using (\ref{condition_constant_minus_pt_1d}) to
define $c_0(R)$ as a function from $\mathbb{R}_+$ to $\mathbb{R}$, it can be easily checked that it's a piecewise constant and right continuous function, where the jumps happen at $R=d(x,x_j)$ for $x\in N$. By increasing $R$ to all through all these jump points, it can be checked by induction that either there exists $R>0$ such that $N\backslash D_R\neq \varnothing$, and the three conditions above are met, or for $R$ large enough such that $N\subset D_R$, the condition (\ref{discrete_minus_pt_large_R_condtion}) is met.

The intuition for what happens for $\Gamma$ divergence to compute the cost to go from $\mu$ to $\nu$ in this example is that $R$ determines in which neighbourhood there would be mass flowing from $\mu$ to the missing point $x_j$ to form $\gamma^*$, and relative entropy takes care of the cost between the intermediate measure $\gamma^*$ and $\nu$.

Next we consider a generalization of Example \ref{discrete_pts_add_point} to higher dimensions.
\begin{example}\label{add_pt_multi_dim}
$X = \mathbb{R}^m$, $\Gamma = \mathrm{Lip}(1;C_b(\mathbb{R}^m))$ with respect to Euclidean distance. $N \doteq \{x_1,x_2,\dots, x_n\} \subset X$, where we assume for simplicity that $x_i\neq x_j$ for any $i\neq j$. $\nu$ is uniform distribution over $N$. $\mu$ is uniform over $N\cup\{y\}$, where $y\notin N$. 
\end{example}
For this example, we apply the same notation as in Example \ref{discrete_pts_minus_pt}. To move from $\mu$ to $\nu$ through an intermediate measure $\gamma^*$, one expects the mass of $\mu$ at the new point $y$ to be transferred to nearby points of $y$ to form $\gamma^*$, and then let relative entropy take care of the cost between $\gamma^*$ and $\nu$. We denote $D_R = \{x\in X: d(x,y)\leq R\}$, and guess there exists $R>0$ and $c_0\in \mathbb{R}$ such that $g^*$ defined on $N\cup\{y\}$ is 
\begin{equation*}
g^*(x)=\left\{
\begin{aligned}
& 0 & x=y,\\
& -d(y,x) & x\in N\cap D_R,\\
& c_0 & x \in N\backslash D_R, \\
\end{aligned}
\right.
\end{equation*}
and 
\begin{equation*}
\gamma^*(dx)=\left\{
\begin{aligned}
& \frac{e^{g^*(x)}}{\int e^{g^*(z)}\nu(dz)}\nu(dx) & x\in N \cap D_R, \\
& \mu(dx) & x\in N\backslash D_R.\\
\end{aligned}
\right.
\end{equation*}
To make them satisfy 
(\ref{verification_conditon_1}) and (\ref{verification_condition_2}), we need \\
1) $g^*\in\Gamma$.\\
2) If $N\backslash D_R\neq \varnothing$, then we need for $x\in N\backslash D_R$
$$\frac{d\gamma^*}{d\nu}(x) = \frac{e^{g^*(x)}}{\int e^{g^*(z)}\nu(dz)},$$
which gives
$$\frac{n}{n+1} = \frac{d\mu}{d\nu}(x) = \frac{e^{c_0}}{\left(\sum_{x\in N\cap D_R}e^{-d(x,y)} + \sum_{x\in N\backslash D_R}e^{c_0}\right)/n}.$$
Denote $k = \# (N\cap D_R)$ the number of points in $N\cap D_R$, then by solving the equation above we have 
\begin{align}\label{condition_constant_add_pt_higher_d}
c_0 =\log\left( \frac{1}{k+1} \sum_{x\in N\cap D_R} e^{-d(x,y)}\right).
\end{align}\\
3)  Similar to Example \ref{discrete_pts_add_point}, a sufficient condition for (\ref{verification_condition_2}) to hold is for any $x\in N$, $\mu(x)\leq \gamma^*(x)$. If $N\backslash D_R \neq \varnothing$, this is equivalent to $c_0 < -\max_{x\in N\cap D_R}d(x,y)$.
If $N\backslash D_R = \varnothing$, this condition reduces to 
$$\min_{x\in N}\gamma^*(x) =\frac{e^{-\max_{x\in N}d(x,y)}}{\sum_{x\in N}e^{-d(x,y)}} \nu(x) >\mu(x),$$
which is equivalent to 
$$e^{-\max_{x\in N}d(x,y)} < \frac{n}{n+1}\sum_{x\in N}e^{-d(x,y)}.$$
Similar approach as in Example \ref{discrete_pts_add_point} and Example \ref{discrete_pts_minus_pt}, one can show there exists $R$ such that the conditions above are satisfied. We omit the details here.

\subsection{Other Directions}\label{other_directions}
There are many other directions which can be explored about $\Gamma$-divergence. Possible direction involve
\begin{itemize}
    \item both $\mu$ and $\nu$ are distributions over discrete points, where $\mathrm{supp}(\mu)$ has a few points more, or less than $\mathrm{supp}(\nu)$.
    \item $\mu$ is a discrete probability distribution, while $\nu$ is a continuous probability distribution, or vise versa.
    \item Numerical Estimation of $\Gamma$-divergence for general pair of probability measures.\\
    For example, one can make use of the first variational expression of (\ref{Vari}) to do numerical approximation of $G_\Gamma(\mu\lVert\nu)$. In \cite{belbarraj}, variational representation (\ref{eqn:var_forms}), together with neural networks, are used to get numerical estimates of mutual information, which is a specific type of relative entropy, of a joint probability distribution of two random variables from samples generated by the distribution. It is expected one can do estimation for $\Gamma$-divergence using similar approaches. 
\end{itemize}

\subsection{Limits and Approximations of $\Gamma$-divergence}\label{limits_and_approximations}

In this section, we consider limits that are obtained as the admissible set gets large or small,
and the $\Gamma$-divergence will be approximated by relative entropy or a transport distance, respectively. We also consider in special cases more informative expansions.
Throughout the section we assume the conditions of Theorem \ref{optimizer}.

Fix an admissible set of $\Gamma_0 \subset C_b(S)$, and take $\Gamma = b\Gamma_0 =\left\{b\cdot  f:f\in\Gamma_0\right\}$ for $b>0$. Then the following proposition holds.
\begin{proposition}\label{lim_RE}
For $\mu,\nu\in \mathcal{P}(S)$, 
$$\lim_{b\to\infty} G_{b\Gamma_0}(\mu\lVert\nu) = R(\mu\lVert\nu).$$
\end{proposition}
\begin{proof}
We separate the proof into two cases, $R(\mu\lVert\nu)<\infty$ and $R(\mu\lVert\nu)=\infty$.

\vspace{\baselineskip}\noindent
1) If $R(\mu\lVert\nu)<\infty$, then for any $b>0$, 
\begin{align}
G_{b\Gamma_0}(\mu\lVert\nu) = \inf_{\gamma\in\mathcal{P}(S)}\left\{W_{b\Gamma_0}(\mu,\gamma)+R(\gamma\lVert\nu)\right\}\leq R(\mu\lVert\nu)<\infty.\label{eqn:bbounds}
\end{align}
From Theorem \ref{optimizer} we know there exists a unique optimizer $\gamma^*$ for each $b$, which we write as $\gamma_b^*$. Note that
 $$R(\gamma_b^*\lVert \nu)
 \leq R(\mu\lVert\nu)<\infty,$$
and therefore  $\left\{\gamma_b^*\right\}_{b>0}$ is precompact in the weak topology \cite[Lemma 1.4.3(c)]{dupell4}.
 Given any subsequence $b_k$, there exists a further subsequence (again denoted by $b_k$) and $\gamma_{\infty}^*\in\mathcal{P}(S)$ such that $\gamma^*_{b_k}\Rightarrow\gamma^*_\infty$. On the other hand,
 \begin{align*}
 W_{b\Gamma_0}(\mu,\gamma^*_b) &= \sup_{f\in b\Gamma_0}\left\{\int_S f d(\mu-\gamma_b^*)\right\}\\
 &=b\sup_{f\in \Gamma_0}\left\{\int_S f d(\mu-\gamma_b^*)\right\}=b W_{\Gamma_0}(\mu,\gamma_b^*),
 \end{align*}
 and 
 $W_{b\Gamma_0}(\mu,\gamma_b^*)\leq G_{b\Gamma_0}(\mu\lVert\nu) \leq R(\mu\lVert\nu)<\infty.$
 Thus
 \begin{align*}
 W_{\Gamma_0}(\mu,\gamma_\infty^*)&\leq \liminf_{k\to\infty}W_{\Gamma_0}(\mu,\gamma_{b_k}^*) = \liminf_{k\to\infty}\frac{1}{b_k}W_{b_k\Gamma_0}(\mu,\gamma_{b_k}^*)\\
 &\leq \liminf_{k\to\infty}\frac{1}{b_k}R(\mu\lVert\nu)=0,
 \end{align*}
and since $\Gamma_0$ is admissible,
 $\gamma_\infty^*=\mu$.  We thus conclude that
 \begin{align*}
 \liminf_{k\to\infty}G_{b_k\Gamma_0}(\mu\lVert\nu)&=\liminf_{k\to\infty}\left(W_{b_k\Gamma_0}(\mu,\gamma_{b_k}^*)+R(\gamma_{b_k}^*\lVert \nu)   \right)\\
 &\geq \liminf_{k\to\infty}R(\gamma_{b_k}^*\lVert \nu)\\
 & \geq R(\mu\lVert\nu),
 \end{align*}
 and since the original subsequence was arbitrary  
 $$\liminf_{b\to\infty}G_{b\Gamma_0}(\mu\lVert\nu)\geq R(\mu\lVert\nu).$$
 On the other hand, we have by \eqref{eqn:bbounds} that
 $$\limsup_{b\to\infty}G_{b\Gamma_0}(\mu\lVert\nu)\leq R(\mu\lVert\nu),$$
 and the statement is proved.
 
\vspace{\baselineskip}\noindent
 2) $R(\mu\lVert\nu)=\infty.$ For this case, we want to prove that 
 $$\liminf_{b\to\infty}G_{b\Gamma_0}(\mu\lVert\nu)=\infty.$$
 If not, then there exists a subsequence $\left\{b_k\right\}_{b\in\mathbb{N}}$ such that 
 $$\lim_{k\to\infty}G_{b_k\Gamma_0}(\mu\lVert\nu)<\infty.$$
 For this subsequence, we can apply the argument used in part 1) to conclude there exists $\gamma_{b_k}^*$ such that 
 $$G_{b_k\Gamma_0}(\mu\lVert\nu)=W_{b_k\Gamma_0}(\mu,\gamma_{b_k}^*)+R(\gamma_{b_k}^*\lVert\nu)  .$$
Moreover there exists a further subsequence of this sequence, which for simplicity we also denote by $\left\{b_k\right\}_{k\in\mathbb{N}}$, which satisfies $\gamma_{b_k}^*\Rightarrow \mu$. Then by the same argument as in 1), we would conclude
$$\lim_{k\to\infty}G_{b_k\Gamma_0}(\mu\lVert\nu)\geq R(\mu\lVert\nu)=\infty.$$
This contradiction proves the statement.
\end{proof}

\vspace{\baselineskip}
On the other hand, if $\Gamma=\delta\Gamma_0$ for small $\delta>0$, we can approximate the $\Gamma$-divergence in terms of the $W_{\Gamma_0}$.

\begin{proposition}
For $\mu,\nu\in\mathcal{P}(S)$
$$\lim_{\delta\to 0}\frac{1}{\delta} G_{\delta\Gamma_0}(\mu\lVert\nu)=W_{\Gamma_0}(\mu,\nu).$$
\end{proposition}
\begin{proof}
For any $\delta>0$, Jensen's inequality implies
\begin{align*}
\frac{1}{\delta}G_{\delta \Gamma_0}(\mu\lVert\nu) &= \frac{1}{\delta} \sup_{g\in\delta\Gamma_0}\left\{\int_S g d\mu - \log\int_S e^gd\nu \right\}\\
&\leq \frac{1}{\delta} \sup_{g\in\delta\Gamma_0}\left\{\int_S g d\mu - \int_S g d\nu\right\}\\
&=\sup_{g\in\Gamma_0}\left\{\int_S g d\mu - \int_S g d\nu\right\}\\
&= W_{\Gamma_0}(\mu,\nu),
\end{align*}
and therefore 
$$\limsup_{\delta\to 0}\frac{1}{\delta}G_{\delta \Gamma_0}(\mu\lVert\nu)\leq W_{\Gamma_0}(\mu,\nu).$$

For the reverse inequality we consider two cases.

\vspace{\baselineskip}\noindent
1) $W_{\Gamma_0}(\mu,\nu)<\infty.$ For $0<\delta<1$ the argument used above shows  
$$G_{\delta\Gamma_0}(\mu\lVert\nu)\leq \delta W_{\Gamma_0}(\mu,\nu) \leq W_{\Gamma_0}(\mu,\nu)
<\infty.$$
By Theorem \ref{optimizer}, we know there exists $\gamma^*_\delta\in\mathcal{P}(S)$, such that 
$$G_{\delta\Gamma_0}(\mu\lVert\nu) = W_{\delta\Gamma_0}(\mu,\gamma^*_\delta)+R(\gamma^*_\delta\lVert\nu).$$
Since $R(\gamma^*_\delta\lVert\nu)<G_{\delta\Gamma_0}(\mu\lVert\nu)\leq W_{\Gamma_0}(\mu,\nu)$ for $\delta\in(0,1)$, for any sequence $\delta_k\subset (0,1)$ there a further a subsequence (again denoted $\delta_k$) such that $\delta_k$ is decreasing, $\lim_{k\to\infty}\delta_k=0$, and $\gamma^*_{\delta_k}$ converges weakly to a probability measure, which we denote as $\gamma_0^*$. Then
by the lower semicontinuity of $R(\cdot \lVert\nu)$
$$R(\gamma_0^*\lVert\nu)\leq \liminf_{k\to\infty} R(\gamma_{\delta_k}^*\lVert\nu)\leq \liminf_{k\to\infty}G_{\delta_k\Gamma_0}(\mu,\nu)\leq \lim_{k\to\infty}\delta_k W_{\Gamma_0}(\mu,\nu)=0.$$
Since $R(\gamma_0^*\lVert\nu)\geq 0$ with equality if and only if $\gamma_0^*=\nu$, we conclude $R(\gamma_0^*\lVert\nu)=0$ and $\gamma_0^*=\nu$. Therefore
\begin{align*}
\liminf_{k\to\infty}\frac{1}{\delta_k}G_{\delta_k \Gamma_0}(\mu\lVert\nu)&\geq \liminf_{k\to\infty}\frac{1}{\delta_k}W_{\delta_k\Gamma_0}(\mu,\gamma^*_{\delta_k})\\
&=\liminf_{k\to\infty}W_{\Gamma_0}(\mu,\gamma^*_{\delta_k})\\
&\geq W_{\Gamma_0}(\mu,\gamma^*_0)=W_{\Gamma_0}(\mu,\nu),
\end{align*}
and since the original sequence was arbitrary
$$\liminf_{\delta\to 0}\frac{1}{\delta}G_{\delta\Gamma_0}(\mu\lVert\nu)\geq W_{\Gamma_0}(\mu,\nu).$$

\vspace{\baselineskip}\noindent
2) $W_{\Gamma_0}(\mu,\nu)=\infty.$
If $\liminf_{\delta\to 0}\frac{1}{\delta}G_{\delta\Gamma_0} (\mu\lVert\nu)<\infty$, then there is a subsequence $\left\{\delta_l\right\}_{l\in\mathbb{N}}\subset (0,1)$ that achieves this $\liminf$. 
From essentially the same proof above applied to this subsequence, it can be shown there exists a further subsequence (again denoted $\left\{\delta_{l}\right\}$)  and  $\gamma_0^*\in\mathcal{P}(S)$ such that 
$$G_{\delta_{l}\Gamma_0}(\mu\lVert\nu) =W_{\delta_{l}\Gamma_0}(\mu,\gamma^*_\delta)+ R(\gamma^*_{\delta_{l}}\lVert\nu),$$
and 
$$\gamma^*_{l}\Rightarrow \gamma_0^*.$$
Denote $M\doteq \liminf_{\delta\to 0}\frac{1}{\delta}G_{\delta\Gamma_0} (\mu\lVert\nu)=\lim_{l\to\infty}\frac{1}{\delta_{l}}G_{\delta_{l}\Gamma_0} (\mu\lVert\nu)<\infty$. 
Since for $l$ large enough
$$R(\gamma^*_{\delta_{l}}\lVert\nu)\leq G_{\delta_{l}\Gamma_0}(\mu\lVert\nu)\leq \delta_{l}(M+1),$$
we have 
$$R(\gamma_0^*\lVert\nu)\leq \liminf_{l\to\infty}R(\gamma_{\delta_{l}}^*\lVert\nu)\leq \lim_{l\to\infty}\delta_{l}(M+1)=0,$$
and thus $\gamma_0^*=\nu$. However this leads to 
\begin{align*}
M=\lim_{l\to\infty}\frac{1}{\delta_{l}}G_{\delta_{l}\Gamma_0} (\mu\lVert\nu)&\geq \lim_{l\to\infty}\frac{1}{\delta_{l}}W_{\delta_{l}\Gamma_0} (\mu,\gamma_{\delta_{l}}^*)\\
&=\lim_{l\to\infty}W_{\Gamma_0}(\mu,\gamma_{\delta_{l}}^*)\geq W_{\Gamma_0}(\mu,\nu)=\infty.
\end{align*}
This contradiction implies $$\liminf_{\delta\to 0}\frac{1}{\delta}G_{\delta\Gamma_0} (\mu\lVert\nu)=\infty = W_{\Gamma_0}(\mu,\nu).$$
\end{proof}

\vspace{\baselineskip}
We now consider more refined approximations when $b$ is large.
Previously we described the limiting behavior when we vary the size of $\Gamma$. From Proposition \ref{lim_RE}, we know that when $\mu\not\ll\nu$, $\lim_{b\to\infty} G_{b\Gamma_0}(\mu\lVert\nu) = \infty$. In some applications one might use a large transport cost as ``penalty'' so that while allowing non-absolutely continuous perturbations, control on  $G_\Gamma(\mu\lVert\nu)$ will ensure that $\mu$ is not too far away from $\nu$. 

In the rest of this section, we investigate the behavior when $b\to\infty$, and in particular how $G_{b\Gamma_0}(\mu\lVert\nu)$ will behave for fixed $\mu$ and $\nu$. We only consider the case that $\Gamma_0=\mathrm{Lip}(c,S;C_b(S))$ for some function $c$ satisfies the condition of Theorem \ref{massdual}, Assumption \ref{mea-det} and Assumption \ref{finite}, and $\mu,\nu\in L^1(a)$ with $a$ in Assumption \ref{finite}. We separate the cases depending on whether $\mu$ and $\nu$ are discrete or continuous. The results presented here are only for special cases,
and further development of these sorts of expansions would be useful. 

\subsubsection{Finitely supported discrete measures}
We will consider the case where $\mathrm{supp}(\nu)$ has finite cardinality, and $\mu$ is also discrete with finite support. 

\begin{theorem}
Suppose $\nu$ and $\mu$ are discrete with finite support, where $\mathrm{supp}(\nu) = \left\{x_i\right\}_{1\leq i\leq N}$ and $\mathrm{supp}(\mu) = \left\{y_j\right\}_{1\leq j\leq M}$. Then there exists $\gamma^*\in\mathcal{P}(S)$ with $\gamma^*\ll\nu$ such that 
\begin{equation}
G_{b\Gamma_0}(\mu\lVert\nu) =bW_{\Gamma_0}(\mu,\gamma^*)+  R(\gamma^*\lVert
\nu)+o(b),
     \label{eqn:exp}
\end{equation}
where $o(b)\leq 0$ satisfies $o(b)\rightarrow 0$ as $b\rightarrow \infty$.
Furthermore, we can characterize $\gamma^*$ as the measure that minimizes $R(\gamma\lVert\nu)$ over the collection of $\gamma\in P(S)$ that satisfy the constraint 
\begin{align}\label{min_gamma}
W_{\Gamma_0}(\mu,\gamma) = \inf_{\theta\ll\nu}W_{\Gamma_0}(\mu,\theta).
\end{align}
If we further assume that 
$$c(y_j,x_i) \neq c(y_j,x_l)$$ 
for $1\leq j\leq M$ and $1\leq i\neq l\leq N$, which is to avoid ties, then $\gamma^*$ has the following form.
Let $S_i$ be the indicies $j$ in $\{1,\ldots,M\}$ for which $x_i$ is the point in $\left\{x_l\right\}_{1\leq l\leq N}$ closest to $y_j$.
Then for $1\leq i\leq N$,
$$\gamma^*(\{x_i\})= \sum_{j\in S_i}\mu(\{y_j\}).$$
\end{theorem}
\begin{remark}
In discrete case, is easily checked that the infimum in (\ref{min_gamma}) is achieved. Take a sequence of $\theta_n\ll\nu$ such that 

$$W_{\Gamma_0}(\mu,\theta_n)\leq \inf_{\theta\ll\nu}W_{\Gamma_0}(\mu,\theta) + 1/n.$$
Since $\theta_n$ is supported on the compact set $\mathrm{supp}(\nu)=\left\{x_i\right\}_{1\leq i\leq N}$
$\left\{\theta_n\right\}_{n\in\mathbb{N}}$ is compact, and hence there exist $\theta^*\ll\nu$ 
and a subsequence $\left\{\theta_{n_k}\right\}_{k\in\mathbb{N}}$
that converges to $\theta^*$ weakly.
By the lower semicontinuity of $W_{\Gamma_0}$
$$W_{\Gamma_0}(\mu,\theta^*) \leq \liminf_{n\to\infty}W_{\Gamma_0}(\mu,\theta_n) \leq \inf_{\theta\ll\nu}W_{\Gamma_0}(\mu,\theta),$$
and therefore $\theta^*$ achieves the infimum of (\ref{min_gamma}).
\end{remark}



\begin{proof}
We use the  representation $G_\Gamma(\mu\lVert\nu)=\inf_{\gamma\in\mathcal{P}(S)}\left\{  R(\gamma\lVert
\nu)+W_{\Gamma}(\mu,\gamma)\right\}$. 
First note that
\begin{align*}
G_{b\Gamma_0}(\mu\lVert\nu) &=\inf_{\gamma\in\mathcal{P}(S)}\left\{  R(\gamma\lVert
\nu)+W_{b\Gamma_0}(\mu,\gamma)\right\}\\
&\leq R(\gamma^*\lVert\nu)+ W_{b\Gamma_0}(\mu,\gamma^*)\\
&=R(\gamma^*\lVert\nu)+ bW_{\Gamma_0}(\mu,\gamma^*).
\end{align*}
Next, fix any $\varepsilon>0$, and take a near optimizer $\gamma_b$, so that for each $b$ 
$$G_{b\Gamma_0}(\mu\lVert\nu)\geq R(\gamma_b\lVert\nu) + W_{b\Gamma_0}(\mu,\gamma_b)-\varepsilon.$$
We must have $\gamma_b\ll\nu$. By (\ref{min_gamma}), we know 
$$W_{b\Gamma_0}(\mu,\gamma_b)= b W_{\Gamma_0}(\mu,\gamma_b) \geq b W_{\Gamma_0}(\mu,\gamma^*)= W_{b\Gamma_0}(\mu,\gamma^*).$$
Thus
\begin{align}
R(\gamma^*\Vert\nu)+W_{b\Gamma_0}(\mu,\gamma^*)&\geq \inf_{\gamma\in P(S)}\left\{  R(\gamma\lVert
\nu)+W_{b\Gamma_0}(\mu,\gamma)\right\} \nonumber\\
& = G_{b\Gamma_0}(\mu,\nu)\nonumber\\
&\geq R(\gamma_b\lVert\nu) + W_{b\Gamma_0}(\mu,\gamma_b)-\varepsilon\nonumber\\
&\geq R(\gamma_b\lVert\nu) + W_{b\Gamma_0}(\mu,\gamma^*)-\varepsilon.\label{eqn:ble}
\end{align}

Since $W_{b\Gamma_0}(\mu,\gamma^*)$ is finite we can subtract it on both sides, and get
$$R(\gamma_b\lVert\nu)\leq R(\gamma^*\lVert\nu)+\varepsilon$$
for any $b<\infty $. Then by \cite[Lemma 1.4.3(c)]{dupell4} $\left\{\gamma_b\right\}_{b\in(0,\infty)}$ is tight. Take a convergent subsequence $\left\{\gamma_{b_k}\right\}$, and denote its limit by $\gamma_{\infty}$. It is easily checked that $\gamma_{\infty}\ll\nu$, so $W_{\Gamma_0}(\mu,\gamma_\infty) \geq W_{\Gamma_0}(\mu,\gamma^*)$. On the other hand, by \eqref{eqn:ble}
\begin{align*}
W_{\Gamma_0}(\mu,\gamma_\infty)- W_{\Gamma_0}(\mu,\gamma^*)&\leq \liminf_{k\to\infty} W_{\Gamma_0}(\mu,\gamma_{b_k})- W_{\Gamma_0}(\mu,\gamma^*)\\
&=\liminf_{k\to\infty}\frac{1}{b_k}(W_{b_k\Gamma_0}(\mu,\gamma_{b_k})-W_{b_k\Gamma_0}(\mu,\gamma^*))\\
&\leq \liminf_{k\to\infty}\frac{1}{b_k}(R(\gamma^*\lVert\nu)-R(\gamma_{b_k}\lVert\nu)+\varepsilon)\\
&\leq \liminf_{k\to\infty}\frac{1}{b_k}(R(\gamma^*\lVert\nu)+\varepsilon)\\
& = 0.
\end{align*}

Thus we conclude that $W_{\Gamma_0}(\mu,\gamma_\infty)= W_{\Gamma_0}(\mu,\gamma^*)$. By the definition of $\gamma^*$  we must have $R(\gamma_\infty\lVert\nu)\geq R(\gamma^*\lVert\nu)$. Choose $k_0$ such that $b_{k_0}\geq 1$.
Then
\begin{align*}
&\liminf_{k\to\infty} \left(G_{b_k\Gamma_0}(\mu\lVert\nu)-[R(\gamma^*\lVert\nu)+b_k W_{\Gamma_0}(\mu,\gamma^*)]\right)\\
&\quad\geq \liminf_{k\to\infty}\left(R(\gamma_{b_k}\lVert\nu)+b_kW_{\Gamma_0}(\mu,\gamma_{b_k})-\varepsilon - (R(\gamma^*\lVert\nu) + b_k W_{\Gamma_0}(\mu,\gamma^*))\right)\\
&\quad\geq \liminf_{k\to\infty} (R(\gamma_{b_k}\lVert\nu)-R(\gamma^*\lVert\nu)) +\liminf_{k\to\infty}b_k(W_{\Gamma_0}(\mu,\gamma_{b_k})-W_{\Gamma_0}(\mu,\gamma^*))-\varepsilon\\
&\quad\geq (R(\gamma_{\infty}\lVert\nu)-R(\gamma^*\lVert\nu)) + \liminf_{k\to\infty}(W_{\Gamma_0}(\mu,\gamma_{b_k})-W_{\Gamma_0}(\mu,\gamma^*))-\varepsilon\\
&\quad\geq 0 + (W_{\Gamma_0}(\mu,\gamma_\infty) - W_{\Gamma_0}(\mu,\gamma^*))-\varepsilon\\
&\quad\geq -\varepsilon
\end{align*}
where the fourth inequality is because $R(\gamma_\infty\lVert\nu) \geq R(\gamma^*\lVert\nu)$ and the lower semi-continuity of $W_{\Gamma_0}(\mu,\cdot)$. Since $\varepsilon>0$ is arbitrary, this establishes \eqref{eqn:exp} along the given subsequence. For any other sequence $\{b_k\}_{k\in\mathbb{N}}$ along which \\$\lim_{k\to\infty}\left(G_{b_k\Gamma_0}(\mu\lVert\nu)-[R(\gamma^*\lVert\nu)+b_k W_{\Gamma_0}(\mu,\gamma^*)]\right)$ has a limit, we can also take a subsequence from it according to the discussion above. Thus the statement is proved.

The proof of the claimed form for $\gamma^*$ under the stated additional conditions is straightforward and omitted.
\end{proof}

\subsubsection{An example with $\nu$ is continuous}
To illustrate an interesting scaling phenomenon, here we consider the example with 
$S = \mathbb{R}$, $c(x,y)=|x-y|$,  $\nu=\mbox{Unif}([0,1])$, $\mu=\delta_0$.
Consider $\gamma^*(dx) = c_0 e^{-bx}dx$ and $g^*(x) = -bx$ for $0\leq x\leq 1$, where $c_0$ is the normalizing constant. 
For this example $\Gamma_0=\mathrm{Lip}(c,S;C_b(S))$ is the set of bounded functions over $\mathbb{R}$ with Lipschitz constant 1.  It is easily checked using Theorem \ref{verif} that $\gamma^*$ and $g^*$ are the optimizers in
$$G_{b\Gamma_0}(\mu\lVert\nu)=\inf_{\gamma\in\mathcal{P}(S)}\left\{ 
W_{b\Gamma_0}(\mu,\gamma)+R(\gamma\lVert\nu)\right\}
=\sup_{g\in b\Gamma_0} \left\{\int_S g d\mu-\log\int_S e^g\nu\right\}.$$
Thus we have 
$$G_{b\Gamma_0}(\mu\lVert\nu)=-\int_0^1 bx d\mu - \log\int_0^1 e^{-bx}d\nu =-\log\int_0^1 e^{-bx} dx= \log\left(\frac{b}{1-e^{-b}}\right),$$
and in this case, $G_{b\Gamma_0}(\mu\lVert\nu)$ scales as $\log(b)+o(\log(b))$.

For comparison we consider the optimal transport cost between $\mu$ and $\nu$.
We have
\begin{align*}
W_{bc}(\mu,\nu) &\doteq \sup_{g\in b\Gamma_0}\left\{\int_S g d\mu -\int_S gd\nu\right\} \\
&= b\sup_{g\in \Gamma_0}\left\{\int_S g d\mu -\int_S gd\nu\right\}=bW_c(\mu,\nu) 
\end{align*}
and one can calculate that $W_{c}(\mu,\nu)=1/2.$ Thus $W_{bc}(\mu,\nu)=b/2$, and so $G_{b\Gamma_0}(\mu\lVert\nu)$ gives a much smaller divergence between non absolutely continuous measures $\mu$ and $\nu$ than the corresponding optimal transport cost when the admissible $\Gamma = b\Gamma_0$ is becoming large. 

\section{Application to Uncertainty Quantification in Static Case}\label{application_static}
In this section, we consider the application of $\Gamma$-divergence in deriving uncertainty bounds in the static case. By using Definition \ref{def:defofV}, Theorem \ref{thm:main} and similar approaches as in \cite{dupkatpanple}, we derive the improved uncertainty information quantification upper and lower bounds (\ref{new_quant_upper}) and (\ref{new_quant_lower}). Next, we establish the the linearization bounds of (\ref{new_quant_upper}) and (\ref{new_quant_lower}) in Section \ref{uq_linearization}. Finally, in Section \ref{discrete}, we use a discrete example to investigate what the linearization bounds (\ref{sens_bound}) represent in the sensitivity analysis situation, and investigate the optimization problem associated with it.

\subsection{Static Case Setup}

Consider $S$ a Polish space, and $\mu,\nu$ two probability measures on $S$. The question here is to get a bound for 

\begin{align}\label{QoI}
\left|\int f d\mu - \int f d\nu\right|
\end{align}
where $f\in C_b(S)$. In \cite{dupkatpanple}, uncertainty quantification information inequalities (UQIIs) based on relative entropy between $\mu$ and $\nu$ are derived. However, in some
applications, $R(\mu\lVert\nu)$ is not guaranteed to be finite, for example when $\mu\ll\nu$ does not hold, then the UQIIs derived from \cite{dupkatpanple} only provide an ineffective bound. Enlightened by the idea from $\Gamma$-divergence, we want to restrict the test functions $f$ to be in a subset $\Gamma\subset C_b(S)$ and get a bound for (\ref{QoI}).

\subsection{classic method}
We consider $f\in\Gamma$, where $\Gamma$ is an admissible subset of $C_b(S)$ as defined in Definition \ref{access}. 
By similar ideas as in \cite{dupkatpanple}, we can do the following. According to Definition \ref{def:defofV}, 
$$G_\Gamma(\mu\lVert\nu) = \sup_{g\in\Gamma}\left\{\int g d\mu - \log \int e^gd\nu\right\},$$
we get
$$G_\Gamma (\mu\lVert\nu) \geq \int f d\mu - \log\int e^f d\nu,$$
which is equivalent to 
$$\int f d\mu \leq \log\int e^f d\nu + G_\Gamma(\mu\lVert\nu).$$
For $c>0$, we would have $c(f-\int f d\nu)\in c\Gamma\doteq\left\{cg: g\in \Gamma\right\}$. So by substituting $f$ with $c(f-\int fd\nu)$ and $\Gamma$ with $c\Gamma$, we have 

$$c(\int f d\mu - \int f d\nu) \leq \log \int e^{c(f- \int f d\nu)} d\nu + G_{c\Gamma} (\mu\lVert\nu),$$
which in turn is equivalent to 

\begin{align}\label{Gamma_div_bound}
\int f d\mu - \int f d\nu \leq \frac{1}{c}\log \int e^{c(f- \int f d\nu)} d\nu + \frac{1}{c}G_{c\Gamma}(\mu\lVert\nu).
\end{align}
By Theorem \ref{thm:main}, $$G_\Gamma(\mu\lVert\nu) = \inf_{\gamma\in P(X)}\left\{R(\gamma\lVert\nu)+W_\Gamma(\mu-\gamma)\right\},$$
where $W_\Gamma(\mu-\gamma) = \sup_{g\in\Gamma}\left\{\int g d(\mu-\gamma)\right\}$. From now on, we use the notation $W_\Gamma(\mu,\gamma) \doteq W_\Gamma(\mu-\gamma)$ to illustrate the symmetry of the roles of both measures within $W_\Gamma$. Then for any choice of fixed $\gamma$, we will have 

$$G_{c\Gamma}(\mu\lVert\nu)\leq R(\gamma\lVert\nu)+W_{c\Gamma}(\mu,\gamma)= R(\gamma\lVert\nu)+cW_{\Gamma}(\mu,\gamma).$$
Putting this inequality back to (\ref{Gamma_div_bound}), we have 

$$\int f d\mu - \int f d\nu \leq \inf_{c>0,\gamma\in P(X)}\left\{ \frac{1}{c}\log \int e^{c(f- \int f d\nu)} d\nu + \frac{1}{c}R(\gamma\lVert\nu)+W_\Gamma(\mu,\gamma)\right\}.$$
Using the same method, we can get a lower bound of the form
$$\int f d\mu - \int f d\nu \geq \sup_{c>0,\gamma\in P(X)}\left\{ -\frac{1}{c}\log \int e^{-c(f- \int f d\nu)} d\nu - \frac{1}{c}R(\gamma\lVert\nu) - W_\Gamma(\mu,\gamma)\right\}.$$
We call (\ref{new_quant_upper}) and (\ref{new_quant_lower}) the improved uncertainty quantification information bounds.
\begin{align}\label{new_quant_upper}
\inf_{c>0,\gamma\in P(X)}\left\{ \frac{1}{c}\log \int e^{c(f- \int f d\nu)} d\nu + \frac{1}{c}R(\gamma\lVert\nu)+W_\Gamma(\mu,\gamma)\right\}
\end{align}

\begin{align}\label{new_quant_lower}
\sup_{c>0,\gamma\in P(X)}\left\{ -\frac{1}{c}\log \int e^{-c(f- \int f d\nu)} d\nu - \frac{1}{c}R(\gamma\lVert\nu)-W_\Gamma(\mu,\gamma)\right\}
\end{align}

\begin{remark}
Letting $c\to 0$, we have that

$$\lim_{c\to 0}\frac{1}{c}\log\int e^{c(f-\int f d\nu)}d\nu=0,$$
and 
$$\lim_{c\to 0}\frac{1}{c}G_{c\Gamma}(\mu\lVert\nu) = W_\Gamma(\mu,\nu).$$
Thus the upper bound (\ref{new_quant_upper}) is always a better bound than the vanilla $W_\Gamma(\mu,\nu)$. 
On the other hand, when $\mu\ll\nu$,

\begin{align*}
&\inf_{c>0,\gamma\in P(X)}\left\{ \frac{1}{c}\log \int e^{c(f- \int f d\nu)} d\nu + \frac{1}{c}R(\gamma\lVert\nu)+W_\Gamma(\mu,\gamma)\right\} \\
&\leq \inf_{c>0}\left\{ \frac{1}{c}\log \int e^{c(f- \int f d\nu)} d\nu + \frac{1}{c}R(\mu\lVert\nu)+W_\Gamma(\mu,\mu)\right\} \\
&= \inf_{c>0}\left\{ \frac{1}{c}\log \int e^{c(f- \int f d\nu)} d\nu + \frac{1}{c}R(\mu\lVert\nu)\right\} 
\end{align*}
So the upper bound (\ref{new_quant_upper}) we get here is always not worse than inequalities derived by \cite{dupkatpanple} when $f\in\Gamma$ and $\mu\ll\nu$. It is typically strictly better, and in cases can be much better, even when $\mu\ll\nu$ holds.
\end{remark}

\subsection{linearization result}\label{uq_linearization}
As considered in \cite[section 2.3]{dupkatpanple}, we consider the linearization result here. We will first fix choices of $\gamma$ and optimize over $c$ to bring in the variance of $f$ as a crucial parameter when $R(\gamma\lVert\nu)$ is small, then move on to consider the optimum of $\gamma$. Fixing $\nu$, we consider $\mu\in P(S)$ such that $W_\Gamma(\mu,\nu)$ is finite. For fixed $\gamma\in P(S)$, we have 
$$\int f d\mu - \int f d\nu \leq \inf_{c>0}\left\{\frac{1}{c}\log \int e^{c(f- \int f d\nu)} d\nu + \frac{1}{c}R(\gamma\lVert\nu)\right\}+W_\Gamma(\mu,\gamma).$$
By \cite[Theorem 2.12]{dupkatpanple}, we can get when $f\neq E_\nu[f]$ $\nu-a.s.$,
\begin{align*}
\int f d\mu - \int f d\nu \leq \sqrt{2Var_\nu (f)}\sqrt{R(\gamma\lVert\nu)}+W_\Gamma(\mu,\gamma)+O(R(\gamma\lVert\nu)).
\end{align*}
By doing the same procedure to lower bound, we will also get 

\begin{align*}
\int f d\mu - \int f d\nu \geq -\sqrt{2Var_\nu (f)}\sqrt{R(\gamma\lVert\nu)}-W_\Gamma(\mu,\gamma)+O(R(\gamma\lVert\nu)).
\end{align*}
Thus by combining them together, we get

\begin{align}\label{sens_bound}
\left|\int f d\mu - \int f d\nu \right|\leq \sqrt{2Var_\nu (f)}\sqrt{R(\gamma\lVert\nu)}+W_\Gamma(\mu,\gamma)+O(R(\gamma\lVert\nu)).
\end{align}

To simplify our investigation, we only focus our attention on the main term, and neglect the higher order term $O(R(\gamma\lVert\nu))$, and consider the case that ${Var_\nu(f)}\neq 0$. First, let's establish the existence of an optimizer $\gamma^*$ that achieves the infimum 
\begin{align}\label{linearization_bound}
\inf_{\gamma\in P(X)}\left\{\sqrt{2Var_\nu (f)}\sqrt{R(\gamma\lVert\nu)}+W_\Gamma(\mu,\gamma)\right\}.
\end{align}
\begin{lemma}\label{linearized_opt}
Assume $W_\Gamma(\mu,\nu)<\infty$. Then there exists $\gamma^*\in P(X)$ such that 

$$\gamma^* = \arg\inf_{\gamma\in P(X)}\left\{\sqrt{2Var_\nu (f)}\sqrt{R(\gamma\lVert\nu)}+W_\Gamma(\mu,\gamma)\right\}. $$
\end{lemma}
\begin{proof}
For simplicity of writing, let's write $C_{\nu,f} = \sqrt{2Var_\nu(f)}$. The existence can be given by first taking a series of near optimizer $\gamma_n$, such that 

$$C_{\nu,f}\sqrt{R(\gamma_n\lVert\nu)}+W_\Gamma(\mu,\gamma_n)\leq \inf_{\gamma\in P(X)}\left\{C_{\nu,f}\sqrt{R(\gamma\lVert\nu)}+W_\Gamma(\mu,\gamma)\right\}+\frac{1}{n}.$$
Since $C_{\nu,f}> 0$ by our assumption, we would have 

\begin{align*}
\sqrt{R(\gamma_n\lVert\nu)}&\leq \frac{1}{C_{\nu,f}}\left(\inf_{\gamma\in P(X)}\left\{C_{\nu,f}\sqrt{R(\gamma\lVert\nu)}+W_\Gamma(\mu,\gamma)\right\}+\frac{1}{n}\right)\\
&\leq \frac{1}{C_{\nu,f}} (W_\Gamma(\mu,\nu)+1).
\end{align*}
So $\left\{R(\gamma_n\lVert\nu)\right\}_{n\in\mathbb{N}}$ is bounded. By \cite[Lemma 1.4.3(c)]{DupEll}, $\left\{\gamma_n\right\}_{n\geq 1}$ are tight, and we can extract a convergent subsequence, whose limit we denote as $\gamma^*$. By lower semi-continuity of $R(\cdot\lVert\nu)$ and $W_\Gamma(\mu,\cdot)$, we can conclude that $\gamma^*$ is the minimizer of the variational expression we consider.
\end{proof}

Next let's use a discrete example to show how the linearization bound (\ref{sens_bound}) can be used to obtain a sensitivity bound.
\subsection{A Discrete Example}\label{discrete}
Let's consider the following example in the space $S = \mathbb{R}^d$.

\begin{example}
$S = \mathbb{R}^d$, $\Gamma = \mathrm{Lip}(1)\cap C_b(S)$. Let $n\in\mathbb{N}$. For $\theta\in\mathbb{R}$ in some open neighborhood of $0$ and each $i\in \{1,\ldots,n\}$,
let $\theta \rightarrow p_i(\theta) \in (0,1)$ and $\theta \rightarrow x_i(\theta) \in S$ be smooth functions, 
with $\sum_{i=1}^n p_i(\theta)=1$. Denote $\mu_\varepsilon = \sum_{i=1}^n p_i(\varepsilon ) \delta_{x_i(\varepsilon)}$ and $\nu =\mu_0 = \sum_{i=1}^n p_i(0)\delta_{x_i(0)}$. For given $f\in\Gamma$, we are interested in getting upper bounds for $\lim_{\varepsilon\to 0^+}\frac{1}{\varepsilon}\left|\int f d\mu_\varepsilon- \int f d\nu\right|$.
\end{example}

We use the bound (\ref{sens_bound}). We assume that for each $i$, $x_i(\varepsilon)$ is not constant in $\varepsilon$. In order to make  relative entropy part finite, a general choice of $\gamma_\varepsilon$ is $\gamma_\varepsilon = \sum_{i=1}^n q_i(\varepsilon)\delta_{x_i(0)}$, where $q_i(0) = p_i(0)$, and $\theta \to q_i(\theta)$ is smooth for $i=1,2,\dots,n$. We do the Taylor expansion for $R(\gamma_\varepsilon\lVert\nu)$. 
In the fourth equation below we use the Taylor expansion  $\log(1+x) = x - \frac{1}{2}x^2 + o(x^2)$ for $|x|<1$. And the last equation is because $\sum_{i=1}^n q'_i(0) = \sum_{i=1}^n q_i''(0) = 0$, which is due to the fact that $\sum_{i=1}^n q_i(\varepsilon) = 1$ always holds.

\begin{align*}
R(\gamma_\varepsilon\lVert\nu) &= \sum_{i=1}^n q_i(\varepsilon)\log\left(\frac{q_i(\varepsilon)}{p_i(0)}\right)\\
&= \sum_{i=1}^n \left(q_i(0)+\varepsilon q_i'(0) + O(\varepsilon^2)\right)\log\left(\frac{q_i(0)+\varepsilon q_i'(0) + \frac{1}{2} \varepsilon^2 q_i''(0)+O(\varepsilon^3)}{p_i(0)}\right)\\
&= \sum_{i=1}^n \left(q_i(0)+\varepsilon q_i'(0) + O(\varepsilon^2)\right)\log\left(1+\frac{\varepsilon q_i'(0) + \frac{1}{2} \varepsilon^2 q_i''(0)+O(\varepsilon^3)}{p_i(0)}\right)\\
& = \sum_{i=1}^n \left(p_i(0) + \varepsilon q_i'(0) + O(\varepsilon^2) \right) \left(\varepsilon\frac{q_i'(0)}{p_i(0)}+\frac{1}{2}\varepsilon^2 \left(\frac{q_i''(0)}{p_i(0)}-\frac{q_i'(0)^2}{p_i(0)^2}\right)+O(\varepsilon^3)\right)\\
& = \varepsilon\sum_{i=1}^n q_i'(0) + \frac{1}{2}\varepsilon^2 \sum_{i=1}^n \left(q_i''(0)+\frac{q_i'(0)^2}{p_i(0)}\right) + O(\varepsilon^3)\\
&= \frac{1}{2}\varepsilon^2\sum_{i=1}^n \frac{q_i'(0)^2}{p_i(0)} + O(\varepsilon^3).
\end{align*}
Now we introduce a lemma for the expansion $W_\Gamma(\mu_\varepsilon,\gamma_\varepsilon)$.

\begin{lemma}\label{general_expansion_Wasserstein}
$$W_\Gamma(\mu_\varepsilon,\gamma_\varepsilon) = \varepsilon\sum_{i=1}^n p_i(0)\lVert x_i'(0)\lVert_2 + \varepsilon W_\Gamma(\rho,\tilde{\rho}) + o(\varepsilon),$$
where $\lVert\cdot\lVert_2$ is the Euclidean norm, $\rho = \sum_{i=1}^n p_i'(0)\delta_{x_i(0)}$ and $\tilde{\rho} = \sum_{i=1}^n q_i'(0) \delta_{x_i(0)}$ are two signed measures on $S$, and $W_\Gamma(\rho,\tilde{\rho})=\sup_{g\in\Gamma}\left\{\int g d(\rho-\tilde{\rho})\right\}$.
\end{lemma}
\begin{proof}
For simplicity, in the proof, let's denote $x_i(0)$, $x_i'(0)$, $p_i(0)$, $p_i'(0)$, $q_i'(0)$ as $x_i$, $x_i'$, $p_i$, $p_i'$, $q_i'$ respectively. Recalling $\Gamma = \mathrm{Lip}(1)\cap C_b(S)$, for $g\in\Gamma$, $|g(x_i)- g(x_i(\varepsilon))|\leq \lVert x_i - x_i(\varepsilon)\lVert_2= \varepsilon\lVert x_i'\lVert_2 + O(\varepsilon^2)$. Then we have
\begin{align*}
W_\Gamma(\mu_\varepsilon,\gamma_\varepsilon) & = \sup_{g\in\Gamma}\left\{\int g d(\mu_\varepsilon - \gamma_\varepsilon)\right\}\\
& = \sup_{g\in\Gamma}\left\{\sum_{i=1}^n g(x_i(\varepsilon)) (p_i + \varepsilon p_i')  - \sum_{i=1}^n g(x_i) (p_i + \varepsilon q_i') \right\} +O(\varepsilon^2)\\
& = \sup_{g\in\Gamma}\left\{\sum_{i=1}^n p_i \left(g(x_i(\varepsilon)) - g(x_i)\right) - \varepsilon \sum_{i=1}^n g(x_i) (p_i' - q_i') \right. \\
& \quad\left. + \varepsilon\sum_{i=1}^n p_i' \left(g(x_i(\varepsilon)) - g(x_i)\right)\right\}+ O(\varepsilon^2)\\
& \leq \sup_{g\in\Gamma}\left\{\sum_{i=1}^n p_i\left|g(x_i) - g\left(x_i(\varepsilon)\right)\right|\right\} + \varepsilon\sup_{g\in\Gamma}\left\{\sum_{i=1}^n g(x_i)(p_i'- q_i')\right\} \\
& \quad+ \varepsilon\sup_{g\in\Gamma}\left\{\sum_{i=1}^n p_i'|g(x_i)-g\left(x_i(\varepsilon)\right)| \right\} + O(\varepsilon^2)\\
&\leq \sum_{i=1}^n p_i \varepsilon\lVert x_i'\lVert_2 + \varepsilon W_\Gamma(\rho,\tilde{\rho}) + \varepsilon^2 \sum_{i=1}^n |p_i'| \lVert x_i'\lVert_2 + O(\varepsilon^2)\\
& = \varepsilon\sum_{i=1}^n p_i\lVert x_i'\lVert_2 + \varepsilon W_\Gamma(\rho,\tilde{\rho}) + O(\varepsilon^2).
\end{align*}
On the other hand, let's take $g^*=(g^*_1,g^*_2,\dots,g^*_n)\in\mathbb{R}^n$ to be the optimizer satisfying

$$\sum_{i=1}^n g^*_i(p_i'-q_i') = \sup_{g\in\Gamma}\left\{\sum_{i=1}^n g(x_i)(p_i'-q_i')\right\},$$
and $|g^*_i - g^*_j|\leq \lVert x_i - x_j\lVert_2$ for all $1\leq i<j\leq n$. The existence of $g^*$ is easy, which the author omits here. Now consider 
\begin{align}\label{Cons_g_star}
g^*_\varepsilon = \left(1 - 5\varepsilon\frac{\max_{1\leq i\leq n}\lVert x_i'\lVert_2}{\min_{1\leq i<j\leq n}\lVert x_i- x_j\lVert_2}\right)g^*.
\end{align}
When $\varepsilon>0$ is small enough, $1 - 5\varepsilon\frac{\max_{1\leq i\leq n}\lVert x_i'\lVert_2}{\min_{1\leq i<j\leq n}\lVert x_i- x_j\lVert_2}>0$. Since for any $1\leq i<j\leq n$, $|g_i^*-g_j^*|\leq \lVert x_i - x_j\lVert_2$, we have 
\begin{align*}
|g^*_{\varepsilon,i} - g^*_{\varepsilon,j}| &= \left(1 - 5\varepsilon\frac{\max_{1\leq i\leq n}\lVert x_i'\lVert_2}{\min_{1\leq i<j\leq n}\lVert x_i- x_j\lVert_2}\right)|g^*_i - g^*_j|\\
&\leq \left(1 - 5\varepsilon\frac{\max_{1\leq i\leq n}\lVert x_i'\lVert_2}{\min_{1\leq i<j\leq n}\lVert x_i- x_j\lVert_2}\right)\lVert x_i - x_j\lVert_2\\
&\leq \lVert x_i - x_j\lVert_2 - 5\varepsilon\max_{1\leq i\leq n}\lVert x_i'\lVert_2.
\end{align*}
Next we define function $h$ on points $\{x_i\}_{1\leq i\leq n}\cup \{x_i(\varepsilon)\}_{1\leq i\leq n}$ as follows: For $1\leq i\leq n$,
\begin{align}\label{def_h_0}
h(x_i) = g^*_{\varepsilon,i}\quad for\ 1\leq i\leq n,
\end{align}
and 
\begin{align}\label{def_h_varep}
h(x_i(\varepsilon)) = h(x_i) + \lVert x_i(\varepsilon) - x_i\lVert_2.
\end{align}
We will show that $h\in\Gamma$. To show that, the only thing we need to check is that for any two specified points of $h$, the value of $h$ satisfies the Lipshitz condition. For $1\leq i<j \leq n$, we have

$$|h(x_i)- h(x_j)| = |g^*_{\varepsilon,i} - g^*_{\varepsilon,j}|\leq \lVert x_i - x_j\lVert_2 - 5\varepsilon\max_{1\leq i\leq n}\lVert x_i'\lVert_2\leq \lVert x_i - x_j\lVert_2,$$

\begin{align*}
|h(x_i(\varepsilon)) - h(x_j)| &= |h(x_i) +\lVert x_i(\varepsilon) - x_i\lVert_2 - h(x_j)|\\
&\leq |h(x_i)-h(x_j)| + \lVert x_i(\varepsilon)-x_i\lVert_2\\
&\leq \lVert x_i - x_j\lVert_2 -5\varepsilon\max_{1\leq i\leq n}\lVert x_i'\lVert_2 + \lVert x_i(\varepsilon)-x_i\lVert_2\\
&\leq \lVert x_i -x_j\lVert_2 -\lVert x_i(\varepsilon) - x_i\lVert_2\\
&\leq \lVert x_i(\varepsilon)- x_j\lVert_2,
\end{align*}
Here the second to last inequality is because $\lVert x_i(\varepsilon) - x_i\lVert_2 = \varepsilon\lVert x_i'\lVert_2 + O(\varepsilon^2).$ Similar reasons also apply for the second to last inequality below.

\begin{align*}
|h(x_i(\varepsilon)) - h(x_j(\varepsilon))| &= |h(x_i) + \lVert x_i(\varepsilon) - x_i\lVert_2 - h(x_j) - \lVert x_j(\varepsilon)- x_j\lVert_2 |\\
&\leq |h(x_i) - h(x_j)| + \lVert x_i(\varepsilon)- x_i\lVert_2 + \lVert x_j(\varepsilon) - x_j\lVert_2\\
&\leq \lVert x_i - x_j\lVert_2 - 5\varepsilon\max_{1\leq i\leq n}\lVert x_i'\lVert_2 + \lVert x_i(\varepsilon)- x_i\lVert_2 + \lVert x_j(\varepsilon) - x_j\lVert_2\\
&\leq \lVert x_i - x_j\lVert_2 - \lVert x_i(\varepsilon)- x_i\lVert_2 - \lVert x_j(\varepsilon) - x_j\lVert_2\\
&\leq \lVert x_i(\varepsilon)-x_j(\varepsilon)\lVert_2.
\end{align*}
Now we can derive the other side of the inequality, where the fourth equation comes from (\ref{def_h_0}), (\ref{def_h_varep}) and (\ref{Cons_g_star}).
\begin{align*}
W_\Gamma(\mu_\varepsilon,\gamma_\varepsilon) &= \sup_{g\in\mathrm{Lip}(1)}\left\{\int g d(\mu_\varepsilon - \gamma_\varepsilon)\right\}\\
&\geq \int h d(\mu_\varepsilon-\gamma_\varepsilon)\\
& = \sum_{i=1}^n\left[ h(x_i(\varepsilon))(p_i+\varepsilon p_i') - h(x_i) (p_i + \varepsilon q_i')\right]+O(\varepsilon^2)\\
& = \sum_{i=1}^n p_i(h(x_i(\varepsilon)) - h(x_i)) +\varepsilon\sum_{i=1}^n h(x_i)(p_i' - q_i') \\
&\quad + \sum_{i=1}^n \left(h(x_i(\varepsilon)) - h(x_i)\right) \varepsilon p_i' + O(\varepsilon^2)\\
& = \sum_{i=1}^n p_i \lVert x_i(\varepsilon)- x_i\lVert_2 + \varepsilon (1-5\varepsilon\frac{\max_{1\leq i\leq n}\lVert x_i'\lVert_2}{\min_{1\leq i<j\leq n}\lVert x_i - x_j\lVert_2})\sum_{i=1}^n g^*_i(p_i'-q_i')\\
&\quad + \varepsilon^2\sum_{i=1}^n\lVert x_i(\varepsilon)-x_i\lVert_2 p_i' + O(\varepsilon^2)\\
& = \varepsilon\sum_{i=1}^n p_i\lVert x_i'\lVert_2 +\varepsilon W_\Gamma(\rho,\tilde{\rho}) + O(\varepsilon^2).
\end{align*}
This completes the other side of the inequality. This lemma is proved.
\end{proof}
Using the expansion of relative entropy  and Lemma \ref{general_expansion_Wasserstein}, bound (\ref{sens_bound}) becomes 

\begin{align*}
&\sqrt{2Var_\nu (f)}\sqrt{R(\gamma_\varepsilon\lVert\nu)}+W_\Gamma(\mu_\varepsilon,\gamma_\varepsilon)\\
= &\varepsilon\left(\sqrt{Var_\nu(f)}\sqrt{\sum_{i=1}^n\frac{q_i'(0)^2}{p_i(0)}} \right)+W_\Gamma(\mu_\varepsilon,\gamma_\varepsilon)+o(\varepsilon)\\
= &\varepsilon\left(\sqrt{Var_\nu(f)}\sqrt{\sum_{i=1}^n\frac{q_i'(0)^2}{p_i(0)}} +\sum_{i=1}^n p_i(0) \lVert x_i'(0)\lVert_2 + W_\Gamma(\rho,\tilde{\rho})\right) + o(\varepsilon),
\end{align*}
where $\rho=\sum_{i=1}^n p_i'(0)\delta_{x_i(0)}$ and $\tilde{\rho}=\sum_{i=1}^n q_i'(0)\delta_{x_i(0)}$. Since the only thing we can choose is the vector $q'(0)$, by getting rid of the constant term $\sum_{i=1}^n p_i(0)\lVert x_i'(0)\lVert_2$, the optimization problem reduces to

$$\inf_{q'(0)\in\mathbb{R}^n_0}\left\{\sqrt{Var_\nu(f)}\sqrt{\sum_{i=1}^n\frac{q_i'(0)^2}{p_i(0)}} + W_\Gamma \left(\sum_{i=1}^n p_i'(0)\delta_{x_i(0)},\sum_{i=1}^n q_i'(0)\delta_{x_i(0)}\right)\right\},$$
where $\mathbb{R}^n_0$ is the space of vectors in $\mathbb{R}^n$ with sum of all the coordinates being 0. Use the definition for $W_\Gamma(\mu'(0),\gamma'(0))$, we can transform the problem above into a min-max problem 
\begin{align}\label{minmax_target}
\inf_{q'(0)\in\mathbb{R}^n_0}\sup_{g\in\Gamma} \left\{\sqrt{Var_\nu(f)}\sqrt{\sum_{i=1}^n\frac{q_i'(0)^2}{p_i(0)}} + \sum_{i=1}^n g(x_i(0))(p'_i(0)-q_i'(0))\right\}.
\end{align}

\subsubsection{Investigate the min-max problem (\ref{minmax_target})}
Since in (\ref{minmax_target}) only values of $g$ on $x_i(0)$ $i=1,2,\dots,n$ are used, for simplicity, let's write $g = (g_1,g_2,\dots, g_n)$, where $g_i = g(x_i(0))$, $i=1,2,\dots,n$. Denote $p = (p_1,p_2,\dots,p_n)$, $p' = (p'_1, p'_2, \dots, p'_n)$, $q' = (q'_1,q'_2,\dots,q'_n)$, where $p_i = p_i(0)$, $p'_i= p'_i(0)$, $q'_i = q'_i(0)$ for $i= 1,2,\dots, n$ respectively. Furthermore, since the bound is invariant when each coordinate of $g$ is changed by the same constant, without loss of generality, we fix $g_1 = 0$. By Lipshitz condition on $\Gamma$, $|g_i |\leq \lVert x_i(0) - x_1(0)\lVert_2 + |g_1| = \lVert x_i(0) - x_1(0)\lVert_2$ is bounded. Without introducing further notation, let's call this set $\Gamma_0$, and notice that $\Gamma_0$ is a compact, convex subset of $\mathbb{R}^n$. Now the quantity we want to estimate is 
\begin{align}\label{real_target}
\inf_{q'\in\mathbb{R}^n_0}\sup_{g\in\Gamma_0} \left\{\sqrt{Var_\nu(f)}\sqrt{\sum_{i=1}^n\frac{q_i'^2}{p_i}} + \sum_{i=1}^n g_i(p'_i-q_i')\right\}.
\end{align}

We want to swap the order of the optimization above. Luckily, Sion's minimax theorem provides such a tool. To make the illustration self contained, let's state Sion's theorem as follows. In the following theorem, $M$ and $N$ are subsets of finite dimensional Euclidean spaces.

\begin{theorem}\cite[Corollary 3.5]{sio}\label{minimax}
Let $M$ and $N$ be convex spaces one of which is compact, and $f$ a function on $M\times N$ satisfying\\
1) $f$ is convex and lower semicontinuous in $M$,\\
2) $f$ is concave and upper semicontinuous in $N$.\\
Then the following holds.
$$\inf_{x\in M} \sup_{y\in N} f(x,y) = \sup_{y\in N}\inf_{x\in M} f(x,y).$$
\end{theorem}
In our application of the theorem, $M = \mathbb{R}_0^n$, $N = \Gamma_0$, and 

$$F(q',g) = \sqrt{Var_\nu(f)}\sqrt{\sum_{i=1}^n\frac{q_i'^2}{p_i}} + \sum_{i=1}^n g_i(p'_i-q_i') $$
is convex in $q'$, concave in $g$ and continuous in both coordinates. Notice that both $M$ and $N$ are convex with $N$ being compact, the conditions of Theorem \ref{minimax} is satisfied, thus 

\begin{align}
&\inf_{q'\in\mathbb{R}^n_0}\sup_{g\in\Gamma_0} \left\{\sqrt{Var_\nu(f)}\sqrt{\sum_{i=1}^n\frac{q_i'^2}{p_i}} + \sum_{i=1}^n g_i(p'_i-q_i')\right\}\label{target_minmax}\\
= &\sup_{g\in\Gamma_0}\inf_{q'\in\mathbb{R}^n_0} \left\{\sqrt{Var_\nu(f)}\sqrt{\sum_{i=1}^n\frac{q_i'^2}{p_i}} + \sum_{i=1}^n g_i(p'_i-q_i')\right\}\label{target_maxmin}\\
= & \sup_{g\in\Gamma_0}\left\{\sum_{i=1}^n g_i p'_i - \sup_{q'\in\mathbb{R}^n_0}\left\{\sum_{i=1}^n g_i q_i' - \sqrt{Var_\nu(f)}\sqrt{\sum_{i=1}^n\frac{q_i'^2}{p_i}}\right\}\right\}.\label{maxmin}
\end{align}
Now for any fixed $g\in\Gamma_0$, we want to first look at the inner optimization, which is  

\begin{align*}
&\sup_{q'\in\mathbb{R}^n_0}\left\{\sum_{i=1}^n g_i q_i' - \sqrt{Var_\nu(f)}\sqrt{\sum_{i=1}^n\frac{q_i'^2}{p_i}}\right\}\\
=& \sup_{q'\in\mathbb{R}^n}\left\{\sum_{i=1}^n g_i q_i' - \sqrt{Var_\nu(f)}\sqrt{\sum_{i=1}^n\frac{q_i'^2}{p_i}}-\infty 1_{(\sum_{i=1}^n q'_i \neq 0)}\right\}.
\end{align*}
What we have done is to transform the optimization question from $\sup$ over $\mathbb{R}^n_0$, which is a subset of $\mathbb{R}^n$ to $\sup$ over the whole space $\mathbb{R}^n$. Notice here now the question turns to be the convex dual of the sum of two convex functions. We first investigate the convex dual of each of these two functions. We introduce the following lemmas. 

\begin{lemma}\label{two_convex_dual}
Denote $\mathbf{1} = (1,1,\dots,1)\in\mathbb{R}^n$. For $h = (h_1,\dots,h_n)\in \mathbb{R}^n$, The following two statements holds.\\
1) 
$$\sup_{q'\in\mathbb{R}^n}\left\{\sum_{i=1}^n h_i q'_i - \infty 1_{(\sum_{i=1}^n q'_i \neq 0)}\right\} = \infty 1_{\left\{h\neq c\mathbf{1}, \forall c\in\mathbb{R}\right\}},$$
2)
$$\sup_{q'\in\mathbb{R}^n}\left\{\sum_{i=1}^n h_i q'_i - \sqrt{Var_\nu(f)}\sqrt{\sum_{i=1}^n\frac{q_i'^2}{p_i}}\right\} = \infty 1_{(\sum_{i=1}^n h_i^2 p_i> Var_\nu(f))}.$$
Moreover, for 2), the optimizing $q'$ exists when $\sum_{i=1}^n h_i^2 p_i\leq Var_\nu(f)$ as the following:\\
i) when $\sum_{i=1}^n h_i^2 p_i< Var_\nu(f)$, $q'_i=0$ for $i=1,2,\dots,n$;\\
ii) when $\sum_{i=1}^n h_i^2 p_i = Var_\nu(f)$, $q'_i = c \cdot h_ip_i$ for any fixed $c\geq 0$.
\end{lemma}
\begin{remark}
The reason that we want to investigate the optimizing $q'$ is that we will consider the saddle point of (\ref{real_target}).
\end{remark}
\begin{proof}
For 1), when there exists $c$ such that $h = c\mathbf{1}$,

\begin{align*}
&\sup_{q'\in\mathbb{R}^n}\left\{\sum_{i=1}^n h_i q_i' -\infty 1_{(\sum_{i=1}^n q_i'\neq 0)}\right\}\\
= &\sup_{q'\in\mathbb{R}^n}\left\{ c \sum_{i=1}^n q_i'  -\infty 1_{(\sum_{i=1}^n q_i'\neq 0)}\right\}\\
= & 0.
\end{align*}
On the other hand, if there exists $i\neq j$, such that $h_i\neq h_j$. Without loss of generality, let's assume that $h_1 > h_2$. Now by taking $q_1' = - q_2' = k>0 $ and $q_i' = 0$ for $i\geq 3$, we would have 

$$\sum_{i=1}^n h_i q_i' -\infty 1_{(\sum_{i=1}^n q_i'\neq 0)} = k(h_1 - h_2).$$
Thus,
\begin{align*}
&\sup_{q'\in\mathbb{R}^n}\left\{\sum_{i=1}^n h_i q_i' -\infty 1_{(\sum_{i=1}^n q_i'\neq 0)}\right\}\\
\geq &\sup_{k> 0}\left\{k(h_1 - h_2)\right\}\\
= &  \infty.
\end{align*}
Now 1) is proved. For 2), by a change of variable $$\tilde{q}' = \left(\sqrt{\frac{Var_\nu(f)}{p_1}}q_1',\sqrt{\frac{Var_\nu(f)}{p_2}}q_2',\dots,\sqrt{\frac{Var_\nu(f)}{p_n}}q_n'\right),$$ we have 
\begin{align*}
& \sup_{q'\in\mathbb{R}^n}\left\{\sum_{i=1}^n h_i q'_i - \sqrt{Var_\nu(f)}\sqrt{\sum_{i=1}^n\frac{q_i'^2}{p_i}}\right\} \\
= &\sup_{q'\in\mathbb{R}^n}\left\{\sum_{i=1}^n \sqrt{\frac{p_i}{Var_\nu(f)}}h_i \cdot \sqrt{\frac{Var_\nu(f)}{p_i}}q_i' - \sqrt{\sum_{i=1}^n \left(\sqrt{\frac{Var_\nu(f)}{p_i}}q_i'\right)^2}\right\} \\
= &\sup_{\tilde{q}'\in\mathbb{R}^n}\left\{\sum_{i=1}^n \sqrt{\frac{p_i}{Var_\nu(f)}}h_i \tilde{q_i}' - \sqrt{\sum_{i=1}^n\tilde{q}_i'^2}\right\}.
\end{align*}
Let's denote $\tilde{h}=(\sqrt{\frac{p_1}{Var_\nu(f)}}h_1,\sqrt{\frac{p_2}{Var_\nu(f)}}h_2,\dots, \sqrt{\frac{p_n}{Var_\nu(f)}}h_n)$. When $\tilde{q}'\neq \mathbf{0}$, the quantity being optimized is 

$$\tilde{h}\cdot \tilde{q}' - \sqrt{\tilde{q}'\cdot\tilde{q}'} = (\lVert\tilde{q}'\lVert_2)\left(\tilde{h}\cdot \frac{\tilde{q}'}{\lVert\tilde{q}'\lVert_2} -1\right).$$
When $\tilde{q}'=\mathbf{0}$, the expression is 0. Thus,
\begin{align*}
& \sup_{q'\in\mathbb{R}^n}\left\{\sum_{i=1}^n h_i q'_i - \sqrt{Var_\nu(f)}\sqrt{\sum_{i=1}^n\frac{q_i'^2}{p_i}}\right\} \\
= &\max\left(\sup_{\tilde{q}'\in\mathbb{R}^n,\tilde{q}'\neq \mathbf{0}}\left\{(\lVert\tilde{q}'\lVert_2)\left(\tilde{h}\cdot \frac{\tilde{q}'}{\lVert\tilde{q}'\lVert_2} -1\right)\right\},0\right) \\
= &\max\left(\sup_{r > 0}\left\{r \left(\lVert\tilde{h}\lVert_2 -1\right)\right\},0\right)\\
= & \infty 1_{(\lVert\tilde{h}\lVert_2> 1)}\\
= & \infty 1_{(\sum_{i=1}^n h_i^2 p_i> Var_\nu(f))}.
\end{align*}
When $\sum_{i=1}^n h_i^2 p_i < Var_\nu(f))$, the optimizer $\tilde{q}' \equiv 0$, which also gives $q' \equiv 0$. When $\sum_{i=1}^n h_i^2 p_i = Var_\nu(f))$, the optimizers $\tilde{q}' \propto \tilde {h}$ (including the case that $\tilde{q}'=0$), which gives that these exists $c\geq 0$ such that $q'_i = c\cdot h_ip_i$.
\end{proof}

Now, we need to use a lemma. Let's first recall Definition \ref{convex_dual_def}, Definition \ref{inf_conv_def}, Definition \ref{lsc_hull_def}, Definition \ref{proper&domain_def}. We make use of Lemma \ref{inf-cov}  with $Y = \mathbb{R}^n$, $m=2$,
and for $q'\in\mathbb{R}^n$, $F_1(q') = \infty 1_{(\sum_{i=1}^n q'_i \neq 0)}$, $F_2(q') = \sqrt{Var_\nu(f)}\sqrt{\sum_{i=1}^n\frac{q_i'^2}{p_i}}$. They satisfy the condition of Lemma \ref{inf-cov}, so we have 
\begin{align*}
&\sup_{q'\in\mathbb{R}^n}\left\{\sum_{i=1}^n g_i q_i' - \sqrt{Var_\nu(f)}\sqrt{\sum_{i=1}^n\frac{q_i'^2}{p_i}}-\infty 1_{(\sum_{i=1}^n q'_i \neq 0)}\right\}\\
= &\overline{\inf_{h\in\mathbb{R}^n}\left\{\infty 1_{(h\neq c\mathbf{1},\forall c\in\mathbb{R})} + \infty 1_{(\sum_{i=1}^n (g_i-h_i)^2 p_i > Var_\nu(f))}\right\}}\\
= &\infty 1_{(Var_\nu(g) > Var_\nu(f))}.
\end{align*}
Here the last equation is because we have to choose $h = c \mathbf{1}$ so that the quantity to be optimized in the second line is not $\infty$, and $c^* = \sum_{i=1}^n g_i p_i$ makes $\sum_{i=1}^n (g_i - c)^2 p_i$ the smallest, and when this $c$ is taken, $\sum_{i=1}^n (g_i - c^*)^2 p_i = Var_\nu(g)$.
\begin{remark}\label{optimizing_q}
Thus the optimizing $q'$ related to the above expression is the optimizing $q'$ of 
$$\sup_{q'\in\mathbb{R}^n} \left\{\sum_{i=1}^n (g_i - c^*)q_i' - \sqrt{Var_\nu(f)}\sqrt{\sum_{i=1}^n\frac{q_i'^2}{p_i}} \right\},$$
where $c^* = \sum_{i=1}^n g_i p_i$ as mentioned above. By Lemma \ref{two_convex_dual}, when $Var_\nu(g) \leq Var_\nu(f)$, the optimizing $q'$ have the form:\\
i) when $Var_\nu(g)< Var_\nu(f)$, $q'_i=0$ for $i=1,2,\dots,n$;\\
ii) when $Var_\nu(g)= Var_\nu(f)$, $q'_i = c \cdot (g_i-\sum_{i=1}^n g_ip_i)p_i$ for any fixed $c\geq 0$.
\end{remark}
Lastly, going back to (\ref{real_target}) and taking the result above, we have

\begin{align*}
&\inf_{q'\in\mathbb{R}^n_0}\sup_{g\in\Gamma_0} \left\{\sqrt{Var_\nu(f)}\sqrt{\sum_{i=1}^n\frac{q_i'^2}{p_i}} + \sum_{i=1}^n g_i(p'_i-q_i')\right\}\\
= & \sup_{g\in\Gamma_0}\left\{\sum_{i=1}^n g_i p'_i - \sup_{q'\in\mathbb{R}^n_0}\left\{\sum_{i=1}^n g_i q_i' - \sqrt{Var_\nu(f)}\sqrt{\sum_{i=1}^n\frac{q_i'^2}{p_i}}\right\}\right\}\\
= & \sup_{g\in\Gamma_0}\left\{\sum_{i=1}^n g_i p'_i - \infty 1_{(Var_\nu(g)> Var_\nu(f))}\right\}\\
= &\sup\left\{\sum_{i=1}^n g_i p_i': g\in\Gamma_0, Var_\nu(g)\leq Var_\nu(f)\right\}.
\end{align*}
Now to summarize, we have the following equations as the sensitivity (upper) bound for $\lim_{\varepsilon\to 0^+}\frac{1}{\varepsilon}|\int f d\mu_\varepsilon - \int f d\nu|$.
\begin{align}
&\inf_{q'\in\mathbb{R}^n_0}\sup_{g\in\Gamma_0} \left\{\sqrt{Var_\nu(f)}\sqrt{\sum_{i=1}^n\frac{q_i'^2}{p_i}} + \sum_{i=1}^n g_i(p'_i-q_i')\right\}\\
= & \sup_{g\in\Gamma_0}\inf_{q'\in\mathbb{R}^n_0} \left\{\sqrt{Var_\nu(f)}\sqrt{\sum_{i=1}^n\frac{q_i'^2}{p_i}} + \sum_{i=1}^n g_i(p'_i-q_i')\right\}\\
= & \sup_{g\in\Gamma_0}\left\{\sum_{i=1}^n g_i p_i': g\in\Gamma_0, Var_\nu(g)\leq Var_\nu(f)\right\}.\label{final_bound}
\end{align}

We want to interpret this sensitivity bound, especially what $q'$ is the best choice in (\ref{real_target}), which represent the best intermediate measure in the original expression (\ref{sens_bound}). First we establish the saddle point property for minimax problems in the following lemma. We assume $M$ and $N$ satisfy the setting of Theorem \ref{minimax}.

\begin{lemma}\label{saddle_pt}
Assume $f$ is a function on $M\times N$ satisfying condition of Theorem \ref{minimax}. If $x^*\in M$ achieves the infimum, i.e.

$$\sup_{y\in N} f(x^*,y) = \inf_{x\in M} \sup_{y\in N} f(x,y),$$
and $y^*\in Y$ achieves the superior, i.e.

$$\inf_{x\in M} f(x,y^*) = \sup_{y\in N} \inf_{x\in M} f(x,y),$$
then $(x^*,y^*)$ satisfies

$$f(x^*, y^*) = \inf_{x\in M}\sup_{y\in N} f(x,y) = \sup_{y\in N}\inf_{x\in M} f(x,y).$$
\end{lemma}
\begin{proof}
From Theorem \ref{minimax}, $\inf_{x\in M}\sup_{y\in N} f(x,y) = \sup_{y\in N}\inf_{x\in M} f(x,y)$, let's denote its value by $I$. Then by the first condition,

$$I = \sup_{y\in N} f(x^*,y) \geq f(x^*,y^*).$$
By the second condition, 

$$I = \inf_{x\in M} f(x,y^*) \leq f(x^*,y^*).$$
Combine the two inequalities above, we have $f(x^*,y^*) = I$.
\end{proof}
Based on Lemma \ref{minimax}, we can extract information on the optimizer $q'$ of (\ref{real_target}), if it exists, based on the information of optimizer $g^*$ of (\ref{final_bound}). Before we proceed, we first establish the existence of the optimizer $q'^*$ of (\ref{real_target}).

\begin{lemma}\label{inf_opt}
There exists optimizer $q'^*$ for (\ref{real_target}).
\end{lemma}
\begin{proof}
Since $\mathbb{R}^n_0 = \left\{x\in \mathbb{R}^n : \sum_{i=1}^n x_i = 0\right\}$ is not compact, we cannot directly conclude the existence of the optimizer $q'^*$ of (\ref{real_target}). Fortunately, we can change the space to be optimized over to a compact set without changing its value in the following way. Let's now also denote the value of (\ref{real_target}) by $I$. By choosing $q' = p'$ in the outside optimization step, we know 

$$I \leq \sqrt{Var_\nu(f)}\sqrt{\sum_{i=1}^n\frac{p_i'^2}{p_i}}.$$
Then for any $q'\in\mathbb{R}_0^n$, such that there exists $j\in\{1,2,\dots,n\}$, $|q'_j|> \sqrt{p_j \sum_{i=1}^n \frac{p_i'^2}{p_i}}$, 
\begin{align*}
&\quad\sup_{g\in\Gamma_0} \left\{\sqrt{Var_\nu(f)} \sqrt{\sum_{i=1}^n\frac{q_i'^2}{p_i}}+\sum_{i=1}^n g_i(p_i' - q_i')\right\}\\
&\geq \sqrt{Var_\nu(f)} \sqrt{\sum_{i=1}^n\frac{q_i'^2}{p_i}}\\
&\geq \sqrt{Var_\nu(f)} \sqrt{\frac{q_j'^2}{p_j}}\\
&> \sqrt{Var_\nu(f)}\sqrt{\sum_{i=1}^n\frac{p_i'^2}{p_i}} = I.
\end{align*}
So
\begin{align}
I &= \inf_{q'\in \mathbb{R}^n_0}\sup_{g\in\Gamma_0} \left\{\sqrt{Var_\nu(f)} \sqrt{\sum_{i=1}^n\frac{q_i'^2}{p_i}}+\sum_{i=1}^n g_i(p_i' - q_i')\right\}\\
& =\inf_{q'\in \mathbb{R}^n_0, |q'_j|\leq \sqrt{p_j \sum_{i=1}^n \frac{p_i'^2}{p_i}}}\sup_{g\in\Gamma_0} \left\{\sqrt{Var_\nu(f)} \sqrt{\sum_{i=1}^n\frac{q_i'^2}{p_i}}+\sum_{i=1}^n g_i(p_i' - q_i')\right\} \label{new_opt}.
\end{align}
Since $\mathbb{R}^n_0\cap \left\{q'\in\mathbb{R}^n: |q'_j|^2 \leq p_j \sum_{i=1}^n \frac{p_i'^2}{p_i}, j=1,2,\dots,n\right\}$ is compact, and 
$$\sup_{g\in\Gamma_0} \left\{\sqrt{Var_\nu(f)} \sqrt{\sum_{i=1}^n\frac{q_i'^2}{p_i}}+\sum_{i=1}^n g_i(p_i' - q_i')\right\}$$
as a function of $q'$ is lower semi-continuous (supremum of a collection of continous functions results in a lower semi-continuous function), we can conclude there is an optimizer $q'^*$ of the (\ref{new_opt}), which is also the optimization problem of our target.
\end{proof}

Now by Lemma \ref{saddle_pt} and Lemma \ref{inf_opt}, we know for any optimizer $g^*$ of (\ref{final_bound}), the optimizer $q'^*$ together with $g^*$ must be the saddle point of (\ref{real_target}), which requires $q'^*$ to have the structure as shown in Remark \ref{optimizing_q}.  Let's now look at the optimization question (\ref{final_bound}) carefully. 

$$\sup_{g\in\Gamma_0}\left\{\sum_{i=1}^n g_i p_i': g\in\Gamma_0, Var_\nu(g)\leq Var_\nu(f)\right\}$$
Since we can add any constant number to all coordinates of $g$ without changing the value of the optimization, instead of considering $\Gamma_0$, let's consider $\Gamma\cap \mathbb{R}^n_0 = \left\{g\in\Gamma:\sum_{i=1}^n g_i = 0\right\}$. Since $\Gamma$ is characterized by inequalities like $|g_i - g_j|\leq |x_i-x_j|$, $\Gamma\cap\mathbb{R}^n_0$ is also a space characterized by linear inequalities. We can imagine it as a "polytope" type space. On the other hand, let's denote $D \doteq \left\{g\in\mathbb{R}^n: Var_\nu(g) \leq Var_\nu(f)\right\}$. $D$ is characterized by a quadratic inequality, thus is a "ellipse" type space. We also consider $D\cap \mathbb{R}^n_0$, which can be viewed as an "ellipse" in the subspace $\mathbb{R}^n_0$. Recalling $p' = (p_1', p_2',\dots, p_n')\in\mathbb{R}^n_0$, we can now view (\ref{final_bound}) as 

\begin{align}\label{find_g_opt}
\sup\left\{g\cdot p': g\in(\Gamma\cap \mathbb{R}^n_0)\cap (D\cap \mathbb{R}^n_0)\right\},
\end{align}
which can be shown in Figure \ref{fig:my_label}.

\begin{figure}
    \centering
    \includegraphics[scale=0.5]{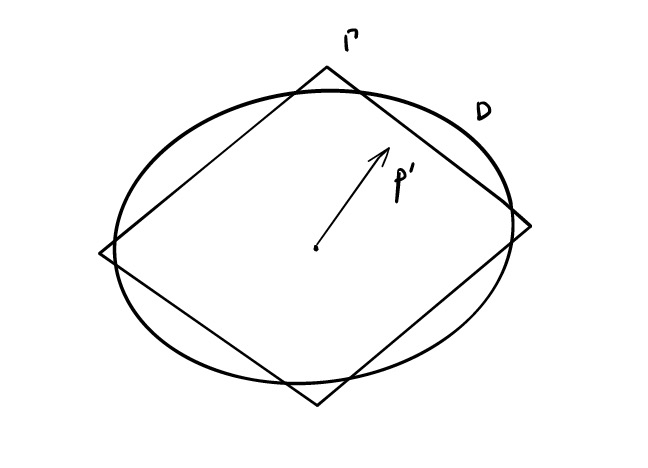}
    \caption{Picture of (\ref{find_g_opt})}
    \label{fig:my_label}
\end{figure}

Since the region $(\Gamma\cap \mathbb{R}^n_0)\cap(D\cap \mathbb{R}^n_0)$ is compact, the optimizing $g^*$ of (\ref{find_g_opt}) exists. Now, we can separate the cases according to where sup is achieves according to whether optimizer $g^*$ (not necessarily unique) is on the boundary of $\Gamma\cap\mathbb{R}^n_0$ or the boundary of $D\cap\mathbb{R}^n_0$. Here optimizer $g^*$ simply means $$g^*\cdot p' = \sup\left\{g\cdot p': g\in(\Gamma\cap \mathbb{R}^n_0)\cap (D\cap \mathbb{R}^n_0)\right\}.$$\\

Depending on the geometry of $(\Gamma\cap \mathbb{R}^n_0)\cap(D\cap \mathbb{R}^n_0)$ and $p'$, there are three different cases on where $g^*$ is located: being on the boundary of $\Gamma\cap \mathbb{R}_0^n$ and not on the boundary of $D\cap \mathbb{R}_0^n$, being on the boundary of $D\cap \mathbb{R}_0^n$ and not on the boundary of $\Gamma\cap \mathbb{R}_0^n$, and being on both the boundary of $\Gamma\cap\mathbb{R}_0^n$ and $D\cap\mathbb{R}_0^n$. It turns out these three cases represent when using only optimal transport cost, using only relative entropy or using both optimal transport cost and relative entropy components gives the best bound. They are discussed in detail in the rest of this section. 

\noindent $\mathbf{Case\ 1}$: $g^*$ is on the boundary of $\Gamma\cap \mathbb{R}^n_0$, and not on the boundary of $D\cap \mathbb{R}^n_0$.\\
In this case, it is implied that $Var_\nu(g^*)< Var_\nu(f)$, then by Lemma \ref{two_convex_dual}, Remark \ref{optimizing_q} and Lemma \ref{saddle_pt}, we have the corresponding optimizing $q'^*$ has to be $\mathbf{0}$, since it's the only $q'^*$ which optimizes 
$$\inf_{q'\in\mathbb{R}^n_0}\left\{\sqrt{Var_\nu(f)}\sqrt{\sum_{i=1}^n\frac{q_i'^2}{p_i}} + \sum_{i=1}^n g^*_i(p'_i-q_i')\right\},$$
with $g^*$ inserted in.

\noindent $\mathbf{Case\ 2}$: $g^*$ is on the boundary of $D\cap \mathbb{R}^n_0$, and not on the boundary of $\Gamma\cap\mathbb{R}^n_0$.\\
In this case, for any optimizing $q'^*$ of (\ref{real_target}), from Lemma \ref{saddle_pt}, we also need to have $g^*$ to be the optimizer for 

$$\sup_{g\in\Gamma\cap\mathbb{R}^n_0}\left\{\sqrt{Var_\nu(f)}\sqrt{\sum_{i=1}^n\frac{(q_i'^*)^2}{p_i}} + \sum_{i=1}^n g_i(p'_i-q'^*_i)\right\}.$$
Here we use $\Gamma\cap \mathbb{R}^n_0$ as the set to be optimized over instead of $\Gamma_0$ before, but since we can add any constant to all coordinates of $g$ without changing the value of the expression, it's valid to use this substitution. Now since $g^*$ is not on the boundary of $\Gamma\cap \mathbb{R}^n_0$, the only possible $q'^*$ for the statement above to be true is that $q'^* = p'$. This is because otherwise, $p'-q'^*\neq 0$, and then $g^*$ as a point in the interior of $\Gamma$ cannot be the optimizer for the expression above.

\noindent $\bold{Case\ 3}$: $g^*$ is on the boundary of $\Gamma\cap \mathbb{R}^n_0$ and the boundary of $D\cap \mathbb{R}^n_0$.\\
In this case, as in Case 2, we also require for any optimizer of (\ref{real_target}), $g^*$ is the optimizer for

$$\sup_{g\in\Gamma\cap\mathbb{R}^n_0}\left\{\sqrt{Var_\nu(f)}\sqrt{\sum_{i=1}^n\frac{(q_i'^*)^2}{p_i}} + \sum_{i=1}^n g_i(p'_i-q'^*_i)\right\}.$$
On the other hand, by Lemma \ref{two_convex_dual}, Remark \ref{optimizing_q} and Lemma \ref{saddle_pt}, we need to have $q_i'^* = c (g^*_i - \sum_{i=1}^n g^*_i p_i)p_i$, for some constant $c\geq 0$. Let's denote the subdifferential set of $g^*$ with respect to $\Gamma\cap\mathbb{R}^n_0$ as 

$$S\doteq \left\{h\in\mathbb{R}^n_0: l_h \cap \mathrm{int}(\Gamma\cap\mathbb{R}^n_0) = \varnothing,\ l_h\doteq \left\{x\in\mathbb{R}^n_0:h\cdot(x-g^*)=0\right\}\right\}.$$
Then from the discussion above, we have the sufficient condition for $q'^*$ to be the optimizer of (\ref{real_target}) is \\
i) There exists $c\geq 0$ such that $q_i'^* = c(g^*_i - \sum_{i=1}^n g^*_i p_i)p_i$,\\
ii) $p'-q'^*\in S$.

If $S$ only contains one element, which means $g^*$ is on the "flat" boundary of the "polytope" $\Gamma\cap\mathbb{R}^n_0$, $q'^*$ will be uniquely determined by the two conditions above. If $S$ contains more than one element, which happens when at $g^*$, multiple inequalities characterizing $\Gamma$ are satisfied, the set of $q'$ satisfying both i) and ii) may also contain multiple elements. However, if for some $c>0$, both conditions i) and ii) are satisfied, we can compute 

\begin{align*}
&\sqrt{Var_\nu(f)}\sqrt{\sum_{i=1}^n\frac{(q_i'^*)^2}{p_i}}\\
= & \sqrt{Var_\nu(f)} \sqrt{\sum_{i=1}^n\frac{c^2(g_i^*-\sum_{i=1}^n g_i^*p_i)^2p_i^2}{p_i}}\\
= & c\sqrt{Var_\nu(f)}\sqrt{\sum_{i=1}^n(g_i^*-\sum_{i=1}^n g_i^*)^2 p_i}\\
= & c \sqrt{Var_\nu(f)}\sqrt{Var_\nu(g^*)}\\
= & c Var_\nu(g^*),
\end{align*}
and
\begin{align*}
\sum_{i=1}^n g^*_i q_i'^* &= \sum_{i=1}^n g^*_i c(g^*_i - \sum_{j=1}^n g^*_j p_j )p_i \\
&= c\left(\sum_{i=1}^n (g^*_i)^2 p_i - (\sum_{i=1}^n g^*_i p_i)^2\right)\\
&= c Var_\nu(g^*).
\end{align*}
Thus for this choice of $q'^*$, 
\begin{align*}
&\sqrt{Var_\nu(f)}\sqrt{\sum_{i=1}^n\frac{(q_i'^*)^2}{p_i}} + \sum_{i=1}^n g^*_i(p'_i-q_i'^*)  \\
=& c Var_\nu(g) + \sum_{i=1}^n g_i^* p_i - c Var_\nu(g) \\
=& \sum_{i=1}^n g_i^* p_i,
\end{align*}
where the last value equals to the value of (\ref{real_target}). Thus condition i) and ii) are also necessary conditions for $q'^*$ to be the optimizer for (\ref{real_target}).

To interpret this result, we return to our consideration of tradeoff between using relative entropy and optimal transport cost. In Case 1, it's more efficient to transport mass using optimal transport cost, so just using the transport method gives the best bound. In Case 2, it's more efficient to use relative entropy, so using the most possible relative entropy mechanism to "move" mass gives the best bound. In Case 3, both relative entropy and optimal transport cost are involved to get the best bound. It is possible that in case 3 the optimizer $q'^*$ for (\ref{real_target}) is not unique. This is due to the fact that depending on specific situations, it's possible that within some neighborhood of one optimizer, the marginal relative entropy and optimal transport cost are exactly the same.
\newpage

\section{Application to Uncertainty Quantification in Diffusion Case}\label{application_diffusion}
In this section, we make use of $\Gamma$ divergence in a specific application of uncertainty quantification, i.e. the Gauss-Markov model case. Gauss-Markov model is a simple model, which are commonly used to approximate more complicated diffusion models (see for example \cite{kus84}) as well as other general models. The role of Gauss-Markov model in the modeling of stochastic process is analogous to the role of Gaussian random variable in the modeling of univariate random variable, and thus is important. For the sake of illustration, we look at the discrete version of the problem, and look at the long time (infinite time) horizon. Formulation of the problem will be given below.

The system we are considering is the 1-dimensional Gauss-Markov model, i.e.

\begin{align}\label{model}
dX_t = -aX_t dt + \sigma dW_t,
\end{align}
with initial condition $X_0 = x_0$, where $a>0$, $\sigma>0$ are constant and $W_t$ is a standard Brownian Motion. The perturbed model is 

\begin{align}\label{model_perturbed}
d\bar{X}_t = -a\bar{X_t} dt + \sigma u(\bar{X_t}) dt + \sigma v(\bar{X_t}) dW_t,
\end{align}
with initial condition $\bar{X}_0 = x_0$. Here $u,v$ are functions of the current state $\bar{X}_t$, and $v>0$. We make the following assumption.
\begin{assumption}\label{assumption_of_uv}
$u:\mathbb{R}\to \mathbb{R}$, $v:\mathbb{R} \to \mathbb{R}_+$ are bounded and continuous. Moreover, there exists a positive constant $\Delta>0$ such that $v(x)\geq \Delta$ holds for all $x\in\mathbb{R}$.
\end{assumption}
The performance measure/ cost we are interested in is 

$$\mathcal{F}(X) = \int_0^T k X_t^2 dt,$$
for large $T$. Take $\mathcal{L}$ as the generator for (\ref{model_perturbed}), where generators are defined in the following sense: For $h\in \mathcal{M}_b(\mathbb{R})$ being bounded measurable functions, 
$$\mathcal{L}h(x) \doteq \lim_{t\to 0^+} \frac{E_x[h(\bar{X}_t)] - h(x)}{t},$$
whenever the above limits exists. It is well known that for $h\in C^2(\mathbb{R})$, 
$$\mathcal{L}h(x) = h'(x) (-ax + \sigma u(x)) + \frac{1}{2}h''(x)\sigma^2v^2(x).$$
Take $V(x) = \frac{1}{2}x^2$, then 
$$\mathcal{L}V(x) = x(-ax + \sigma u(x)) + \frac{1}{2}\sigma^2v^2(x).$$
Under Assumption \ref{assumption_of_uv}, it can be shown there exists $M>0$, such that $\mathcal{L}V(x)<0$ holds for $|x|\geq M$. So $V(x)$ is a Lyapunov function for (\ref{model_perturbed}), then by classical argument (see for example the last paragraph of Proof of \cite[Theorem 2.6]{dupwil} in its appendix) one can show the existence of stationary distribution of (\ref{model_perturbed}). For uniqueness of stationary distribution, let's consider the Markov chain $\{\bar{X}_k\}_{k\in\mathbb{N}}$ where $\bar{X}_k$ is the random variable which is at time $k$ of the stochastic process (\ref{model_perturbed}). It can be easily checked that any stationary distribution of (\ref{model_perturbed}) is also a stationary distribution for $\{\bar{X}_k\}$. Under Assumption \ref{assumption_of_uv}, it can be checked that Markov chain $\{\bar{X}_k\}_{k\in\mathbb{N}}$ is indecomposable (see \cite[Definition 7.14]{bre}), thus the stationary distribution of $\{\bar{X}_k\}_{k\in\mathbb{N}}$ is unique (see \cite[Theorem 7.16]{bre}). So we can conclude the stationary distribution of (\ref{model_perturbed}) is also unique. Let's now denote $\pi_q$ as the stationary distribution of (\ref{model_perturbed}). For $T$ large, we know
$$\mathcal{F}(\bar{X}) = \int_0^T k \bar{X}_t^2 dt  = kT \int_\mathbb{R} x^2 d\pi_q + O(1).$$
Thus we are interested in getting bounds, especially upper bounds for 
\begin{align}\label{last_section_qoi}
\int_\mathbb{R} x^2 d\pi_q.
\end{align}

This type of problems are investigated by people using relative entropy as the type of divergence \cite{petjamdup}, however when the diffusion coefficient of (\ref{model}) is perturbed, like in (\ref{model_perturbed}), the bounds derived from relative entropy turn out to be not useful, since these bounds will be $\infty$. We will show in this section by using $\Gamma$ divergence using a specific $\Gamma$, one can get a meaningful upper bound for the quantity of interest (\ref{last_section_qoi}).

The following is the guideline for this section.
\begin{enumerate}
    \item $\Gamma$-divergence with $\Gamma$ being functions with Lipschitz first derivative.
    \item  Forward mapping specification and the scaling dependence of input cost.
    \item The optimal variational bound for perturbation.
\end{enumerate}

\subsection{$\Gamma$-divergence with a special choice of $\Gamma$}\label{new_Gamma_investigation}

To tackle the problem considering Brownian motion with different diffusion coefficient, we need to use a different $\Gamma$. The $\Gamma$ we are going to consider in this section will be mainly the following type. To illustrate the idea, let's take the space $S$ to be $\mathbb{R}$.
\begin{align}\label{choice_of_Gamma_diffusion}
\Gamma = \left\{g\in C_b(\mathbb{R}): g\in C^1(\mathbb{R})\ and\ g'\in\mathrm{Lip(1)}\right\}.
\end{align}
With this choice of $\Gamma$, we will be able to compare second moment information for two distributions.

\subsubsection{Admissibility of choice (\ref{choice_of_Gamma_diffusion})}
First, we need to show that this $\Gamma$ is admissible. We recall the definition of admissible here, which is the first defined at Definition \ref{access}.

\begin{definition}
Let $\Gamma$ be a subset of $C_b(\mathbb{R})$. 
We call $\Gamma$ $\mathbf{admissible}$ if  $\Gamma$ satisfies the following criteria:

1) $\Gamma$ is convex and closed.

2) $\Gamma$ is symmetric in that $g\in\Gamma$ implies $-g\in\Gamma$, and $\Gamma$ contains all contant functions.

3) $\Gamma$ is determining for $\mathcal{P}(\mathbb{R})$, i.e., $\forall \mu\neq\nu\in \mathcal{P}(\mathbb{R})$, there exists $g\in\Gamma$ such that 

$$\int g d\mu\neq \int g d\nu.$$
\end{definition}
The first two conditions can be checked easily. For the third condition, it can be proved by contradiction. Suppose for fixed two measures $\mu, \nu\in\mathcal{P}(\mathbb{R})$ and any $g\in \Gamma$, $\int g d\mu = \int g d\nu$. Notice that for any $h\in \mathrm{Lip(1)}\cap C_b(\mathbb{R})$, there exists a sequence of positive numbers $\{c_k\}_{k\in\mathbb{N}}$ and $\{g_k\}_{k\in\mathbb{R}}\subset \Gamma$, such that $\lim_{k\to\infty}c_k g_k = h$, and $|c_kg_k|\leq |h|$. Then by dominated convergence theorem where $|h|$ is the dominating function, we conclude 

$$\int h d\mu = \lim_{k\to\infty}\int c_k g_k d\mu = \lim_{k\to\infty} \int c_k g_k d\nu = \int h d\nu.$$
Thus,
$$\sup_{h\in \mathrm{Lip}(1)\cap C_b(\mathbb{R})}\{\int h d(\mu-\nu)\}=0.$$ Since $\mathrm{Lip(1)}\cap C_b(\mathbb{R})$ is measure determining(see for example, \cite[Remark A.3.5]{dupell4}), then we have $\mu=\nu$. Thus the third condition is verified, and the statement $\Gamma$ is admissible is proved.

\subsubsection{Investigate $W_\Gamma(\mu,\nu)$}
In this section, we investigate $W_\Gamma(\mu,\nu) = \sup_{g\in\Gamma }\{\int g d(\mu-\nu)\}$, especially for $\mu$ and $\nu$ being two normal distribution on $\mathbb{R}$. Here throughout this section, we assume $\mu$ and $\nu$ have finite first and second moment. First we introduce two lemmas.

\begin{lemma}\label{diff_mean_infinity}
$W_\Gamma(\mu,\nu)=\infty$ if $E_\mu X\neq E_\nu X$, where $X$ represents random variables with distribution according to the subscript respectively.
\end{lemma}
\begin{proof}
When $E_\mu X\neq E_\nu X$, without loss of generality, let's assume $E_\mu X > E_\nu X$. For fixed $k>0$, we take 

\begin{equation*}
g_{n,k}(x)=\left\{
\begin{aligned}
&\frac{1}{2}k^2 + n \quad &x\geq k+\frac{n}{k}\\
&-\frac{1}{2}x^2 + (k + \frac{n}{k})x - \frac{1}{2}(\frac{n}{k})^2 \quad &\frac{n}{k}< x< k+ \frac{n}{k} \\
&kx \quad &-\frac{n}{k}\leq x \leq \frac{n}{k} \\
&\frac{1}{2}x^2 - (k + \frac{n}{k})x +\frac{1}{2}(\frac{n}{k})^2 \quad &-\frac{n}{k} - k < x< -\frac{n}{k}\\
&-\frac{1}{2}k^2 - n \quad &x\leq -k-\frac{n}{k}
\end{aligned}
\right.
\end{equation*}
which is a bounded approximation with Lipshitz-1 first derivative version of the function $kx$. Notice that $g_{k,n}(x)\to kx$ pointwise and that $\mu$ and $\nu$ has finite first moment, we have $$\lim_{n\to\infty} \int g_{n,k}(x) \mu(dx) = \int kx \mu(dx) = k E_\mu X,$$ and $$\lim_{n\to\infty} \int g_{n,k}(x) \nu(dx) = \int kx \nu(dx) = k E_\nu X.$$ Then 

\begin{align*}
W_\Gamma(\mu,\nu) &= \sup_{g\in\Gamma}\left\{\int g d(\mu-\nu)\right\}\\
&\geq \lim_{n\to\infty} \int g_{n,k} d(\mu-\nu)\\
& = k (E_\mu X - E_\nu X).
\end{align*}
Since $k>0$ can be arbitrary, and $E_\mu X > E_\nu X$, we have 
$W_\Gamma(\mu,\nu) = \infty.$
\end{proof}
In the proof of the last Lemma, we constructed a series of functions in $\Gamma$ to approximate function $kx$. Actually, similar approximation methods can be applied to a family of functions. We introduce the following lemma.
\begin{lemma}\label{approx_by_Gamma_quadratic}
For $h(x) = bx^2 + cx + d$, where $b,c,d\in\mathbb{R}$ and $|b|\leq \frac{1}{2}$, there exists a sequence of functions $\{g_n\}_{n\in\mathbb{N}}$ in $\Gamma$ as  defined by (\ref{choice_of_Gamma_diffusion}) such that for any $x\in\mathbb{R}$,
$$\lim_{n\to\infty} g_n(x) = h(x),$$
and there exists constant $M$, which only depends on $b,c,d$, such that
$$|g_n(x)|\leq \max(|h(x)|,M).$$
\end{lemma}
\begin{proof}
When $b=0$, the construction can be done similar to the one given in the proof of Lemma \ref{diff_mean_infinity}. When $b\neq 0$, we first consider $h(x) = bx^2$. For $h(x) = bx^2$ with $0<b\leq \frac{1}{2}$. We can construct the following $g_n(x)$.
\begin{equation}\label{approx_b_pos}
g_{n}(x)=\left\{
\begin{aligned}
&(2b^2 + b)n^2 \quad &x\geq (2b+1)n\\
&-\frac{1}{2}x^2 + (2b+1)nx -(b+\frac{1}{2})n^2  \quad &n< x< (2b+1)n\\
&bx^2 \quad &-n\leq x \leq n \\
&-\frac{1}{2}x^2 - (2b+1)nx -(b+\frac{1}{2})n^2 \quad &-(2b+1)n < x< -n\\
&(2b^2+b)n^2 \quad &x\leq -(2b+1)n
\end{aligned}
\right.
\end{equation}
For $h(x) = bx^2$ with $-\frac{1}{2}\leq b <0$, we can construct the following $g_n(x)$.
\begin{equation}\label{approx_b_neg}
g_{n}(x)=\left\{
\begin{aligned}
&(-2b^2 + b)n^2 \quad &x\geq (-2b+1)n\\
&\frac{1}{2}x^2 + (2b-1)nx -(b+\frac{1}{2})n^2  \quad &n< x< (-2b+1)n\\
&bx^2 \quad &-n\leq x \leq n \\
&\frac{1}{2}x^2 - (2b-1)nx -(b+\frac{1}{2})n^2 \quad &-(-2b+1)n < x< -n\\
&(-2b^2+b)n^2 \quad &x\leq -(-2b+1)n
\end{aligned}
\right.
\end{equation}
It can be checked directly the $\{g_n\}_{n\in\mathbb{N}}$ constructed above satisfies the expected condition. For general $h(x) = bx^2 + cx +d$ with $0<|b|\leq \frac{1}{2}$, we can rewrite $h(x)$ as $h(x) = b(x+\frac{c}{2b})^2 + d - \frac{c^2}{4b}$. Notice that if one denotes the construction (\ref{approx_b_pos}) or (\ref{approx_b_neg}) depending on whether $b$ is positive or negative as $\{g_n\}_{n\in\mathbb{N}}$, $\{g_n(x+\frac{c}{2b}) + d - \frac{c^2}{4b}\}_{n\in\mathbb{N}}$ is automatically an approximation to $h(x)$ which satisfies the condition of the current Lemma. 
\end{proof}

\begin{lemma}\label{Gamma0}
For $\mu,\nu\in\mathcal{P}(X)$ satisfying $E_\mu X = E_\nu X$, 
$$W_\Gamma (\mu,\nu) = W_{\Gamma_0} (\mu,\nu),$$
where $\Gamma_0=\left\{ g\in C_b(\mathbb{R}): |g''|\leq 1, g(0)=0, g'(0)=0 \right\}$.
\end{lemma}
\begin{proof}
First since $\Gamma_0\subset \Gamma$, we have 
$$W_\Gamma(\mu,\nu) = \sup_{g\in\Gamma}\{\int g d(\mu-\nu)\}\geq \sup_{g\in\Gamma_0}\{\int g d(\mu-\nu)\} = W_{\Gamma_0} (\mu,\nu).$$
Since $E_\mu X = E_\nu X$, we have for any $a,b\in\mathbb{R}$, $E_\mu(aX+b) = E_\nu(aX+b)$. Thus for any $g\in \Gamma$, we can take $f(x) = g(x)-(g'(0)x + g(0))$, and we will have 

$$\int g d(\mu-\nu) = \int f d(\mu-\nu). $$
Notice that $f$ satisfies $|f''|\leq 1$, $f(0)= 0$ and $f'(0)= 0$. Although $f\not\in C_b(\mathbb{R})$, we can take a sequence of functions $\{f_n\}\subset\Gamma_0$ such that $f_n\to f$ pointwise and $|f_n(x)|\leq |f(x)|$, similar as the construction in the proof of Lemma \ref{approx_by_Gamma_quadratic}. By dominated convergence theorem with dominating function as $|g(x)| + |g'(0)||x| + |g(0)|$, we have 
$$\lim_{n\to\infty} \int f_n d(\mu-\nu) = \int f d(\mu-\nu).$$
Then for any $g\in\Gamma$, 
\begin{align*}
W_{\Gamma_0}(\mu,\nu)  = \sup_{h\in\Gamma_0}\{\int h d(\mu-\nu)\}&\geq \lim_{n\to\infty} \int f_n d(\mu-\nu)\\ 
&= \int f d(\mu-\nu) = \int g d(\mu-\nu).    
\end{align*}
Thus 
$$W_{\Gamma_0}(\mu,\nu) \geq \sup_{g\in\Gamma}\{\int g d(\mu-\nu)\} = W_\Gamma(\mu,\nu).$$
Since we have inequalities for both directions, we have proven this lemma.
\end{proof}

Now we move on to compute the $W_\Gamma$ distance between two normal distributions on $\mathbb{R}$. Let's take $\mu = N(0,\sigma_1^2)$ and $\nu = N(0,\sigma_2^2)$. Since both $\mu$ and $\nu$ are symmetric about $0$, we have 
\begin{align*}
W_{\Gamma_0}(\mu,\nu)&=\sup_{g\in\Gamma_0}\left\{\int_{-\infty}^{\infty} g d(\mu-\nu)\right\}\\
&= 2 \sup_{g\in\Gamma_0}\left\{\int_{0}^{\infty} g d(\mu-\nu) \right\}
\end{align*}
We define new probability measures $\bar{\mu}$ and $\bar{\nu}$ on $[0,\infty)$ as the folded ones of $\mu$ and $\nu$ to the positive semi-line, i.e. $\forall 0<a<b$,
\begin{align}
\bar{\mu}([a,b])&=2\mu([a,b]),\\
\bar{\nu}([a,b])&=2\nu([a,b]).
\end{align}
Then 
\begin{align}\label{folded_normal}
W_{\Gamma_0}(\mu,\nu)= W_{\Gamma_0}(\bar{\mu},\bar{\nu}).
\end{align}
Define $c(x,y)=\frac{1}{2}|x^2-y^2|$ for $x,y>0$, and denote $\mathbb{R}_+ = [0,+\infty)$. Notice
for any $g\in\Gamma_0$, and any $z>0$, $|g'(z)| = |\int_0^z g''(w)dw| \leq \int_0^z |g''(w)|dw= \int_0^z 1 dw = z$. Thus for any $x>y>0$, 
$$|g(x)-g(y)| = |\int_y^x g'(z) dz|\leq \int_y^z |g'(z)|dz = \int_y^x zdz = \frac{1}{2}(x^2-y^2).$$
So by considering $\Gamma_0$ as the functions defined on $\mathbb{R}_+$,
$$\Gamma_0\subset \mathrm{Lip}\left(c,\mathbb{R}_+;C_b(\mathbb{R}_+)\right)\doteq \left\{g\in C_b\left(\mathbb{R}_+\right): g(x)-g(y)\leq c(x,y), \forall x,y>0\right\}.$$
Thus, 
$$W_{\Gamma_0}(\bar{\mu},\bar{\nu})\leq W_{\mathrm{Lip}\left(c,\mathbb{R}_+;C_b(\mathbb{R}_+)\right)}(\bar{\mu},\bar{\nu})\doteq \sup_{g\in \mathrm{Lip}\left(c,\mathbb{R}_+;C_b(\mathbb{R}_+)\right)}\left\{\int_0^\infty g d(\bar{\mu}-\bar{\nu}) \right\}.$$
We cite a useful lemma here. In the following lemma, $X$ is a Polish space, \\ $\mathrm{Lip}(c,X;C_b(X))\doteq\left\{f\in C_b(X): f(x)-f(y)\leq c(x,y)\quad\forall x,y\in S\right\}$ and $\Pi(\mu,\nu)$ denotes the collection of all probability measures on $X\times X$ with marginals being $\mu$ and $\nu$ on the first and second arguments respectively.

Using Theorem \ref{massdual} with $X =\mathbb{R}_+$ and 

$$Q=\left\{\min(\frac{1}{2}x^2,n), \max(-\frac{1}{2}x^2,-n):n\in\mathbb{N}\right\},$$ 
we have 

$$W_{\mathrm{Lip}\left(c,\mathbb{R}_+;C_b(\mathbb{R}_+)\right)}(\bar{\mu},\bar{\nu})=\inf_{\pi\in \Pi(\bar{\mu},\bar{\nu})}\left\{\int_{[0,+\infty)\times [0,+\infty)}c(x,y)\pi (dx,dy))\right\},$$
where $\Pi(\bar{\mu},\bar{\nu}) = \{\pi\in\mathcal{P}(\mathbb{R}\times \mathbb{R}): \pi_1 = \bar{\mu}, \pi_2 = \bar{\nu}\}$, here $\pi_1,\pi_2$ are the two marginals of $\pi$ respectively. Notice that the right hand side of the equation is the optimal transport cost between $\bar{\mu}$ and $\bar{\nu}$ with cost function $c(x,y)$.

We need a lemma for computing $W_{\mathrm{Lip}\left(c,\mathbb{R}_+;C_b(\mathbb{R}_+)\right)}(\bar{\mu},\bar{\nu})$.

\begin{lemma}\label{WGauss}
Let $\mu$, $\nu$ be the distributions $N(0,\sigma_1^2)$ and $N(0,\sigma_2^2)$, respectively. Without loss of generality, let's assume that $0<\sigma_1<\sigma_2$. Define $\bar{\mu}, \bar{\nu}$ as (\ref{folded_normal}). Then 

$$W_{\mathrm{Lip}\left(c,\mathbb{R}_+;C_b(\mathbb{R}_+)\right)}(\bar{\mu},\bar{\nu})=\frac{1}{2}(\sigma_2^2-\sigma_1^2).$$
\end{lemma}
\begin{proof}
Notice now we have the dual expression for $W_{\mathrm{Lip}\left(c,\mathbb{R}_+;C_b(\mathbb{R}_+)\right)}(\bar{\mu},\bar{\nu})$ as 
\begin{align*}
W_{\mathrm{Lip}\left(c,\mathbb{R}_+;C_b(\mathbb{R}_+)\right)}(\bar{\mu},\bar{\nu}) &= \sup_{g\in \mathrm{Lip}\left(c,\mathbb{R}_+;C_b(\mathbb{R}_+)\right)}\left\{\int_0^\infty g d(\bar{\mu}-\bar{\nu}) \right\}\\
&= \inf_{\pi\in \Pi(\bar{\mu},\bar{\nu})}\left\{\int_{[0,+\infty)\times [0,+\infty)}c(x,y)\pi (dx,dy))\right\}.
\end{align*}
We take a specific coupling of $\bar{\mu}$ and $\bar{\nu}$, $\pi^*$ as the following: for $x,y\geq 0$, 
$$\pi^*\left([0,x]\times [0,y]\right) = \min\left(\bar{\mu}([0,x]),\bar{\nu}([0,y])\right).$$
Then it can be easily checked that 
$$(x,y) \in \mathrm{supp}(\pi^*)\Longleftrightarrow \bar{\mu}([0,x])=\bar{\nu}([0,y]).$$
Since $\bar{\mu}$ and $\bar{\nu}$ are the folded version of two normal distribution where the first one has a smaller variance, we can thus conclude for $(x,y)\in\mathrm{supp}(\pi^*)$, $x<y$. Thus for $(x,y)\in\mathrm{supp}(\pi^*)$, $c(x,y) = \frac{1}{2}|x^2-y^2| = \frac{1}{2}(y^2-x^2)$. So 

\begin{align*}
W_{\mathrm{Lip}\left(c,\mathbb{R}_+;C_b(\mathbb{R}_+)\right)}(\bar{\mu},\bar{\nu})&=\inf_{\pi\in \Pi(\bar{\mu},\bar{\nu})}\left\{\int_{[0,+\infty)\times [0,+\infty)}c(x,y)\pi (dx,dy))\right\}\\
&\leq \int_{[0,+\infty)\times [0,+\infty)}c(x,y)\pi^* (dx,dy))\\
& = \int_{[0,+\infty)\times [0,+\infty)}\frac{1}{2}(y^2-x^2)\pi^* (dx,dy))\\
& = \frac{1}{2}\left(\int_{[0,+\infty)} y^2 \bar{\nu}(dy) - \int_{[0,+\infty)} x^2 \bar{\mu}(dx)\right)\\
& = \frac{1}{2}(\sigma_2^2 - \sigma_1^2).
\end{align*}
On the other hand, we can take $g_n(x) = \max (-\frac{1}{2}x^2, -n)\in \mathrm{Lip}\left(c,\mathbb{R}_+;C_b(\mathbb{R}_+)\right)$, which converges to $g^*(x) = -\frac{1}{2}x^2$. Since both $\bar{\mu}$ and $\bar{\nu}$ have finite second moment, by the  dominated convergence theorem,

$$\lim_{n\to\infty}\int_{[0,+\infty)}g_n d(\bar{\mu}-\bar{\nu}) = \int_{[0,+\infty)} g^* d(\bar{\mu}-\bar{\nu}).$$
Thus,
\begin{align*}
W_{\mathrm{Lip}\left(c,\mathbb{R}_+;C_b(\mathbb{R}_+)\right)}(\bar{\mu},\bar{\nu}) &= \sup_{g\in \mathrm{Lip}\left(c,\mathbb{R}_+;C_b(\mathbb{R}_+)\right)}\left\{\int_0^\infty g d(\bar{\mu}-\bar{\nu}) \right\}\\
&\geq \lim_{n\to\infty}\int_{[0,+\infty)}g_n d(\bar{\mu}-\bar{\nu}) \\
& = \int_{[0,+\infty)} g^* d(\bar{\mu}-\bar{\nu})\\
& = \int_{[0,+\infty)} -\frac{1}{2}x^2  \bar{\mu} (dx) + \int_{[0,+\infty)} \frac{1}{2}x^2 \bar{\nu}(dx)\\
& = \frac{1}{2}(\sigma_2^2 - \sigma_1^2).
\end{align*}
Combine both directions, we have 
$$W_{\mathrm{Lip}\left(c,\mathbb{R}_+;C_b(\mathbb{R}_+)\right)}(\bar{\mu},\bar{\nu}) = \frac{1}{2}(\sigma_2^2 - \sigma_1^2).$$
\end{proof}
Lastly, let's go back to what we start with $W_\Gamma(\mu,\nu)$. Recall $$\Gamma = \left\{g\in C_b(\mathbb{R}): g\in C^1(\mathbb{R})\ and\ g'\in\mathrm{Lip(1)}\right\}$$ and $$\Gamma_0=\left\{ g\in C_b(\mathbb{R}): |g''|\leq 1, g(0)=0, g'(0)=0 \right\}.$$
\begin{proposition}\label{W_Gamma_two_normal}
Under the same condition as in Lemma \ref{WGauss}, 
$$W_\Gamma(\mu,\nu) = \frac{1}{2}(\sigma_2^2-\sigma_1^2).$$
\end{proposition}
\begin{proof}
First, notice that when considered as functions over $\mathbb{R}_+$, 
$$\Gamma_0\subset \mathrm{Lip}\left(c,\mathbb{R}_+;C_b(\mathbb{R}_+)\right).$$
Thus,
$$W_{\Gamma_0}(\bar{\mu},\bar{\nu})\leq W_{\mathrm{Lip}\left(c,\mathbb{R}_+;C_b(\mathbb{R}_+)\right)}(\bar{\mu},\bar{\nu}) = \frac{1}{2}(\sigma_2^2 - \sigma_1^2).$$
On the other hand, one can construct $h_n\in\Gamma_0$ similar to the way in the proof of Lemma \ref{diff_mean_infinity}, such that $h_n(x)\to -\frac{1}{2}x^2$ pointwise for $x\geq 0$. Then by dominated convergence theorem and definition of $W_{\Gamma_0}$,
\begin{align*}
W_{\Gamma_0}(\bar{\mu},\bar{\nu}) & = \sup_{h\in\Gamma_0}\{\int_{[0,\infty)}h d(\bar{\mu}-\bar{\nu})\}\\
& \geq \lim_{n\to\infty} \int_{[0,\infty)}h_n d(\bar{\mu}-\bar{\nu})\\
& = \int_{[0,\infty)} -\frac{1}{2}x^2 (\bar{\mu}-\bar{\nu})(dx)\\
& = \frac{1}{2}(\sigma_2^2-\sigma_1^2).
\end{align*}
So we conclude $W_{\Gamma_0}(\bar{\mu},\bar{\nu}) = \frac{1}{2}(\sigma_2^2-\sigma_1^2)$. Lastly, by Lemma \ref{Gamma0} and (\ref{folded_normal}), we have 
$$W_\Gamma(\mu,\nu) = W_{\Gamma_0}(\mu,\nu) = W_{\Gamma_0}(\bar{\mu},\bar{\nu}) = \frac{1}{2}(\sigma_2^2-\sigma_1^2).$$
\end{proof}

\subsubsection{$\Gamma$ divergence for two Gaussians}
Now we can turn to our main target $G_\Gamma(\mu\lVert\nu)$, where $\mu$ and $\nu$ are normal distributions here. Let's assume $\mu= N(b_1,\sigma_1^2)$ and $\nu = N(b_2,\sigma_2^2)$. Recall the definition of $G_\Gamma(\mu\lVert\nu)$ from Definition \ref{def:defofV}
\begin{align}
G_{\Gamma}(\mu\lVert\nu)\doteq\sup_{g\in\Gamma}\left\{  \int_{\mathbb{R}} gd\mu-\log\int_{\mathbb{R}}
e^{g}d\nu\right\},
\end{align}
and the alternative representation from Theorem \ref{thm:main},
\begin{align}
G_\Gamma(\mu\lVert\nu) = \inf_{\gamma\in\mathcal{P}(\mathbb{R})}\left\{R(\gamma\lVert\nu) + W_\Gamma(\mu,\gamma)\right\}.
\end{align}
Since the choice of $\Gamma$ in this section is different from what is used in Section \ref{Wass}, we can't directly use theorems from Section \ref{Wass}. We will establish the following theorem to deal with this specific case. The idea from Theorem \ref{optimizer} and Theorem \ref{verif} carries over here.

\begin{theorem}\label{thm:Gamma_div_two_normal}
Let $\mu,\nu$ be two normal distributions with distribution $N(b_1,\sigma_1^2)$ and $N(b_2,\sigma_2^2)$ on $\mathbb{R}$ respectively. $\Gamma = \{g\in C_b(\mathbb{R}): g\in C^1(\mathbb{R})\ and\ g'\in\mathrm{Lip(1)}\}$. Then the following conclusions hold:\\
1) There exists a unique $\gamma^*\in\mathcal{P}(\mathbb{R})$ such that it is the optimizer for (\ref{variational_expression}).\\
2) One can get the exact $\gamma^*$ in all possible situations:\\
i) When $|\frac{1}{\sigma_1^2} - \frac{1}{\sigma_2^2}|\leq 1$, $\gamma^*=N(b_2,\sigma_1^2)$.\\
ii) When $\frac{1}{\sigma_1^2} - \frac{1}{\sigma_2^2}>1$, $\gamma^* = N(b_2,\frac{\sigma_2^2}{1+\sigma_2^2})$;\\
iii) When $\frac{1}{\sigma_1^2} - \frac{1}{\sigma_2^2} < -1$, $\gamma^* = N(b_2, \frac{\sigma_2^2}{1-\sigma_2^2})$.\\
\end{theorem}
\begin{proof}
1) First, by picking $\gamma = \mu$ in (\ref{variational_expression}), we have the inequality
$$G_\Gamma(\mu\lVert\nu) \leq R(\mu\lVert\nu) + W_\Gamma(\mu,\mu) = R(\mu\lVert\nu) < \infty.$$
For a sequence of near optimizers $\{\gamma_n\}_{n\geq 1}$, where
$$R(\gamma_n\lVert\nu) + W_\Gamma(\mu,\gamma_n)\leq G_\Gamma(\mu\lVert\nu) + \frac{1}{n},$$
we have 
$$R(\gamma_n\lVert\nu) \leq R(\gamma_n\lVert\nu) + W_\Gamma(\mu,\gamma_n)\leq G_\Gamma(\mu\lVert\nu) + \frac{1}{n} \leq R(\mu\lVert\nu) + 1.$$
Then by \cite[Lemma 1.4.3(c)]{dupell4} $\{\gamma_{n}\}_{n\geq1}$ is precompact in the weak topology, and thus
has a convergent subsequence $\left\{\gamma_{n_{k}}\right\}_{k\geq 1}$. Denote $\gamma^{\ast}\doteq
\lim_{k\rightarrow\infty}\gamma_{n_{k}}$. The following is analogous to the proof of Theorem \ref{optimizer}, where one can conclude $\gamma^*$ is the unique optimizer of (\ref{variational_expression}).\\
2) Now we try to get the exact $\gamma^*$ in different situations.\\
i) When $|\frac{1}{\sigma_1^2}-\frac{1}{\sigma_2^2}|\leq 1$, 
\begin{align*}
\log\left(\frac{d\mu}{d\nu}(x)\right) &= \log\left(\frac{\frac{1}{\sqrt{2\pi\sigma_1^2}}\exp(-\frac{(x-b_1)^2}{2\sigma_1^2})}{\frac{1}{\sqrt{2\pi\sigma_2^2}}\exp(-\frac{(x-b_2)^2}{2\sigma_2^2})}\right)\\
&=-\frac{1}{2}(\frac{1}{\sigma_1^2} - \frac{1}{\sigma_2^2})x^2+(\frac{b_1}{\sigma_1^2}-\frac{b_2}{\sigma_2^2})x - (\frac{b_1^2}{2\sigma_1^2} - \frac{b_2^2}{2\sigma_2^2})+\log\left(\frac{\sigma_2}{\sigma_1}\right)
\end{align*}
is a quadratic function with second derivative being between $[-1,1]$. Then by Lemma \ref{approx_by_Gamma_quadratic}, one can find a series of functions $g_m\in \Gamma$ and a constant $M>0$ depending on $b_1,b_2,\sigma_1,\sigma_2$, such that $g_m\to \log\left(\frac{d\mu}{d\nu}\right)$ pointwise, and $|g_m(x)|\leq \max(|\log\left(\frac{d\mu}{d\nu}(x)\right)|,M) \leq  \frac{1}{2}|\frac{1}{\sigma_1^2}-\frac{1}{\sigma_2^2}|x^2 + |\frac{b_1}{\sigma_1^2}-\frac{b_2}{\sigma_2^2}||x| + |\frac{b_1^2}{2\sigma_1^2} - \frac{b_2^2}{2\sigma_2^2}|+M$ for all $x\in\mathbb{R}$. Since $ \frac{1}{2}|\frac{1}{\sigma_1^2}-\frac{1}{\sigma_2^2}|x^2 + |\frac{b_1}{\sigma_1^2}-\frac{b_2}{\sigma_2^2}||x| + |\frac{b_1^2}{2\sigma_1^2} - \frac{b_2^2}{2\sigma_2^2}|+M$ is integrable by $\mu$, and $e^{g_m(x)}\leq \max(e^{\log(\frac{d\mu}{d\nu}(x)},1)=\max(\frac{d\mu}{d\nu}(x),1)$ which is integrable by $\nu$, by dominated convergence theorem we have
$$\lim_{m\to\infty}\left\{\int_{\mathbb{R}}g_m d\mu - \log\int_{\mathbb{R}}e^{g_m}d\nu\right\} = \int_{\mathbb{R}}\log\left(\frac{d\mu}{d\nu}\right)d\mu - \log\int_{\mathbb{R}}e^{\log\left(\frac{d\mu}{d\nu}\right)}d\nu.$$
Thus,
\begin{align*}
G_\Gamma(\mu\lVert\nu) &= \sup_{g\in\Gamma}\left\{ \int_{\mathbb{R}} gd\mu-\log\int_{\mathbb{R}}
e^{g}d\nu\right\}\\
& \geq \limsup_{m\to\infty}\left\{\int_{\mathbb{R}} g_m d\mu - \log \int_{\mathbb{R}} e^{g_m} d\nu \right\}\\
& = \int_{\mathbb{R}}\log\left(\frac{d\mu}{d\nu}\right)d\mu - \log\int_{\mathbb{R}}e^{\log\left(\frac{d\mu}{d\nu}\right)}d\nu\\
&= \int_{\mathbb{R}}\log\left(\frac{d\mu}{d\nu}\right)d\mu - \log \int_{\mathbb{R}}\frac{d\mu}{d\nu}d\nu\\
& = R(\mu\lVert\nu).
\end{align*}
On the other hand, we know from the beginning of this proof that $G_\Gamma(\mu\lVert\nu)\leq R(\mu\lVert\nu)$. Thus $G_\Gamma(\mu\lVert\nu) = R(\mu\lVert\nu)$, which is expected, since $\log\left(\frac{d\mu}{d\nu}(x)\right)$ can be approximated pointwise by functions in $\Gamma$. By taking $\gamma=\mu$, we can achieve $R(\mu\lVert\nu)$ in (\ref{variational_expression}). Since the uniqueness of the optimizer $\gamma^*$ holds, $\gamma^* = \mu$.\\
ii) When $\frac{1}{\sigma_1^2} - \frac{1}{\sigma_2^2} > 1$. Consider the potential intermediate measure $\bar{\gamma} = N(b_1,\frac{\sigma_2^2}{1+\sigma_2^2})$. Notice that  $\frac{1}{\sigma_1^2} - \frac{1}{\sigma_2^2} > 1$ if and only if $\sigma_1^2 < \frac{\sigma_2^2}{1+\sigma_2^2}$. By Proposition \ref{W_Gamma_two_normal}, we have 
\begin{align}\label{W_Gamma_ii}
W_\Gamma(\mu,\bar{\gamma}) = \frac{1}{2}(\frac{\sigma_2^2}{1+\sigma_2^2}-\sigma_1^2).
\end{align}
By (\ref{variational_expression}), 
\begin{align*}
G_\Gamma(\mu\lVert\nu) \leq R(\bar{\gamma}\lVert\nu) + W_\Gamma(\mu,\bar{\gamma}).\label{ii_upper_bound}
\end{align*}
On the other hand, taking $\bar{g}= \log\left(\frac{d\bar{\gamma}}{d\nu}\right)$, although $\bar{g}$ is not in $\Gamma$ directly, our choice of $\bar{\gamma}$ implies that the coefficient of $x^2$ in $\bar{g}(x)$ is $-\frac{1}{2}\left(\frac{1+\sigma_2^2}{\sigma_2^2}-\frac{1}{\sigma_2^2}\right)=-\frac{1}{2}$, and therefore
makes sure that $|\bar{g}''|=1$, so like (i), we can find a sequence of functions $\{g_m\}_{m\geq 1}$ in $\Gamma$ which converges to $\bar{g}$. By dominated convergence theorem similar to the one used in (i), one can get
\begin{align}
G_\Gamma(\mu\lVert\nu) &= \sup_{g\in\Gamma}\left\{ \int_{\mathbb{R}} gd\mu-\log\int_{\mathbb{R}}
e^{g}d\nu\right\}\nonumber\\
& \geq \limsup_{m\to\infty}\left\{\int_{\mathbb{R}} g_m d\mu - \log \int_{\mathbb{R}} e^{g_m} d\nu \right\}\nonumber\\
& = \int_\mathbb{R} \log\left(\frac{d\bar{\gamma}}{d\nu}\right)d\mu - \log\int_{\mathbb{R}}e^{\log\left(\frac{d\bar{\gamma}}{d\nu}\right)}d\nu \nonumber\\
&= \int_\mathbb{R} \log\left(\frac{d\bar{\gamma}}{d\nu}\right)d\mu - 0 \\
& = \left(\int_\mathbb{R} \log\left(\frac{d\bar{\gamma}}{d\nu}\right)d\mu - \int_\mathbb{R} \log\left(\frac{d\bar{\gamma}}{d\nu}\right)d\bar{\gamma}\right) + \int_\mathbb{R}\log\left(\frac{d\bar{\gamma}}{d\nu}\right)d\bar{\gamma} \label{ii_lower_bound}. 
\end{align}
Notice that $\log\left(\frac{d\bar{\gamma}}{d\nu}\right)= -\frac{1}{2}x^2 + c_1 x + c_0$ for some constant $c_1,c_0\in\mathbb{R}$, and $\mu$ and $\bar{\gamma}$ are two normal distribution with the same mean, then together with (\ref{W_Gamma_ii}) we can conclude
\begin{align}\label{ii_lower_bound_W_Gamma}
\int_\mathbb{R} \log\left(\frac{d\bar{\gamma}}{d\nu}\right)d\mu - \int_\mathbb{R} \log\left(\frac{d\bar{\gamma}}{d\nu}\right)d\bar{\gamma} = \frac{1}{2}\left(-\sigma_1^2 +  \frac{\sigma_2^2}{1+\sigma_2^2}\right)=W_\Gamma(\mu,\bar{\gamma}).
\end{align}
And
\begin{align}\label{ii_lower_bound_re}
\int_\mathbb{R} \log\left(\frac{d\bar{\gamma}}{d\nu}\right)d\bar{\gamma} = R(\bar{\gamma}\lVert\nu).
\end{align}
So putting (\ref{ii_lower_bound_W_Gamma}) and (\ref{ii_lower_bound_re}) back to (\ref{ii_lower_bound}), we have 
$$G_\Gamma(\mu\lVert\nu)\geq W_\Gamma(\mu,\bar{\gamma}) + R(\bar{\gamma}\lVert\nu).$$
Together with (\ref{ii_upper_bound}) we can conclude $G_\Gamma(\mu\lVert\nu) = R(\bar{\gamma}\lVert\nu) + W_\Gamma(\mu,\bar{\gamma})$, which tells $\bar{\gamma}$ is the optimizer $\gamma^*$ for (\ref{variational_expression}). By the uniqueness of $\gamma^*$ from 1) of the current theorem, we can conclude $\gamma^* = \bar{\gamma} = N(b_1,\frac{\sigma_2^2}{1+\sigma_2^2})$.\\
(iii) This part is similar to (ii), so the author omits the proof here.
\end{proof}
\begin{remark}
In case ii) and iii), let's compute the exact $G_\Gamma(\mu\lVert\nu)$. \\
In case ii), denoting $\sigma_3^2 = \frac{\sigma_2^2}{1+\sigma_2^2}$, then
\begin{align*}
G_\Gamma(\mu\lVert\nu) &= R(\gamma^*\lVert\nu) + W_\Gamma(\mu,\gamma^*)\\
&= \log(\frac{\sigma_2}{\sigma_3}) + \frac{\sigma_3^2 + (b_1-b_2)^2}{2\sigma_2^2}-\frac{1}{2} + \frac{1}{2}(\sigma_3^2 - \sigma_1^2)\\
& = \frac{1}{2}\log(1 + \sigma_2^2) + \frac{(b_1-b_2)^2}{2\sigma_2^2} -\frac{\sigma_2^2}{2(1+\sigma_2^2)} + \frac{1}{2}(\frac{\sigma_2^2}{1+\sigma_2^2}-\sigma_1^2)\\
& = \frac{1}{2}\log(1+\sigma_2^2) + \frac{(b_1-b_2)^2}{2\sigma_2^2} -\frac{1}{2}\sigma_1^2.
\end{align*}
When $\sigma_2^2$ is small, (which in this case also implied $\sigma_1^2$ is small,) we can do the Taylor expansion $\log(1+x) = x+ O(x^2)$ on the first term to get
$$G_\Gamma(\mu\lVert\nu) = \frac{1}{2}(\sigma_2^2-\sigma_1^2) + \frac{(b_1-b_2)^2}{2\sigma_2^2} + O(\sigma_2^4)= \frac{1}{2}|\sigma_2^2-\sigma_1^2| + \frac{(b_1-b_2)^2}{2\sigma_2^2} + O(\sigma_2^4).$$
In case iii), let's also denote $\sigma_3^2 = \frac{\sigma_2^2}{1-\sigma_2^2}$, then
\begin{align*}
G_\Gamma(\mu\lVert\nu) &= R(\gamma^*\lVert\nu) + W_\Gamma(\mu,\gamma^*)\\
&=\log(\frac{\sigma_2}{\sigma_3}) + \frac{\sigma_3^2 + (b_1-b_2)^2}{2\sigma_2^2}-\frac{1}{2} + \frac{1}{2}(\sigma_1^2 - \sigma_3^2)\\
&= \frac{1}{2}\log(1-\sigma_2^2) + \frac{(b_1-b_2)^2}{2\sigma_2^2} + \frac{1}{2}\sigma_1^2
\end{align*}
Similarly, we can do the Taylor expansion when $\sigma_2^2$ is small and get
$$G_\Gamma(\mu\lVert\nu) = \frac{1}{2}(\sigma_1^2-\sigma_2^2) + \frac{(b_1-b_2)^2}{2\sigma_2^2} + O(\sigma_2^4)= \frac{1}{2}|\sigma_1^2-\sigma_2^2| + \frac{(b_1-b_2)^2}{2\sigma_2^2} + O(\sigma_2^4).$$
So to combine both cases together, we can conclude when $|\frac{1}{\sigma_1^2}-\frac{1}{\sigma_2^2}|>1$, 
\begin{align}\label{Gamma_div_two_normal}
G_\Gamma(\mu\lVert\nu)= \frac{1}{2}|\sigma_1^2-\sigma_2^2| + \frac{(b_1-b_2)^2}{2\sigma_2^2} + O(\sigma_2^4).
\end{align}
\end{remark}

\begin{remark}\label{Gamma_k}
The choice of taking constant as $1$ in $\Gamma = \{g\in C_b(\mathbb{R}): g\in C^1(\mathbb{R})\ and\ g'\in\mathrm{Lip(1)}\}$ is somewhat arbitrary. Actually, if we denote $\Gamma_k = \{g\in C_b(\mathbb{R}): g\in C^1(\mathbb{R})\ and\ g'\in\mathrm{Lip(k)}\}$, we can use the exact same discussion within this section, and get the result similar as (\ref{Gamma_div_two_normal}): Take $\mu = N(b_1,\sigma_1^2)$, $\nu = N(b_2,\sigma_2^2)$. When $|\frac{1}{\sigma_1^2}-\frac{1}{\sigma_2^2}|> k$, 
\begin{align}\label{Gamma_div_two_normal_Lip_k}
G_{\Gamma_k}(\mu\lVert\nu) = \frac{k}{2}|\sigma_1^2 -\sigma_2^2| + \frac{(b_1-b_2)^2}{2\sigma_2^2} + O(\sigma_2^4).
\end{align}
This result is the main component for deriving uncertainty bounds in the next section.
\end{remark}

\subsection{Average Cost per Unit Time (ACUT)}
Our application is in average cost per unit time case of discrete version of (\ref{model}) and (\ref{model_perturbed}). The model and perturbed model are as following.

\begin{align}\label{discrete_model}
X_{k+1}^{(N)} = (1-\frac{a}{N})X_k^{(N)} + \sigma W_k
\end{align}
and 
\begin{align}\label{discrete_model_perturbed}
\bar{X}_{k+1}^{(N)} = (1-\frac{a}{N})\bar{X}_k^{(N)} + \sigma \bar{W}_k,
\end{align}
where $W_k \sim i.i.d. N(0,1/N)$ and $\bar{W}_k \sim N(u_k/N, v_k/N)$, $k=0,1,\dots,NT-1$. Here  $u_k, v_k$ are functions of $\bar{X}_k^{(N)}$ which satisfies Assumption \ref{assumption_of_uv}. We take cost function at any given time point as $f(x) = \frac{1}{2} qx^2$. Now we will define an auxillary set to help us.

\subsubsection{An auxillary set}
Here we take space as $\mathbb{R}$, and \\
$\Gamma = \left\{g\in C_b(\mathbb{R}): g\in C^1(\mathbb{R})\ and\ g'\in\mathrm{Lip(1)}\right\}$ as in Section \ref{new_Gamma_investigation}. We define the following anxillary set.
\begin{definition}
For a transition kernel $p$, let 
\[
\mathcal{R}(\Gamma,p) = \left\{ -\log \int_S e^{-g(y)}p(x,dy)-g(x)+\lambda: g \in \Gamma\mbox{ and }\lambda \in \mathbb{R}\right\}.
\]
\end{definition}
The usage of this definition lies in next theorem.
\begin{theorem}
\label{thm:ergcost}
Suppose that $f \in \mathcal{R}(\Gamma,p)$ for some $g$ and $\lambda$. Consider any transition kernel $q$ on $S$ and any stationary probability measure $\pi_q$ of $q$.
Then
\begin{align*}
\int_S f(x)\pi_q(dx)&\leq \int_S G_{\Gamma}(q(x,\cdot)\lVert p(x,\cdot)) \pi_q(dx) + \lambda.
\end{align*}
\end{theorem}
\begin{proof}
Since $g \in \Gamma$, $-g$ is also in $\Gamma$. Then by Theorem \ref{incredible},
\begin{align*}
g(x) &= -f(x) - \log\int_S e^{-g(y)}p(x,dy) + \lambda\\
&= -f(x) + \inf_{q(x,dy)}\left[G_\Gamma (q(x,\cdot)\lVert p(x,\cdot)) + \int_S g(y)q(x,dy)\right]+\lambda.
\end{align*}
For any given specific transition kernel $q$,
$$g(x)\leq -f(x)+ \left[G_\Gamma (q(x,\cdot)\lVert p(x,\cdot)) + \int_S g(y)q(x,dy)\right]+a.$$
Integrating both sides with respect to $\pi_q(dx)$ and using $\int_S q(x,dy)\pi_q(dx)=\pi_q(dy)$ gives the result.
\end{proof}
Before applying Theorem \ref{thm:ergcost}, we first investigate $\mathcal{R}(\Gamma,p)$ for Gaussian transition kernel.

\subsubsection{$\mathcal{R}(\Gamma,p)$ with $p(x,\cdot)\sim N(\alpha x, \sigma^2)$}
In this section, we compute what's included in $\mathcal{R}(\Gamma,p)$ when $p$ is a normal transition kernel in the type of $N(\alpha x,\sigma^2)$ for $\alpha\in (0,1)$. First, let's take $g(x) = -bx^2 - cx -d$, where $b,c,d\in\mathbb{R}$. Then with direct computation one can get when $1-2b\sigma^2 >0$,
\begin{align}
&\quad -\log\int_S e^{-g(y)}p(x,dy) - g(x) + \lambda \label{forward_input}\\
&=-\frac{b\alpha^2 x^2+c\alpha x + c^2\sigma^2/2}{1-2b\sigma^2} + \frac{1}{2} \log (1-2b\sigma^2) +bx^2+cx + \lambda\\
& = b\left(1- \frac{\alpha^2}{1-2b\sigma^2}\right)x^2 + c\left(1-\frac{\alpha}{1-2b\sigma^2}\right)x +\lambda - \frac{c^2\sigma^2}{2(1-2b\sigma^2)} + \frac{1}{2}\log(1-2b\sigma^2),\label{forward_map}
\end{align}
and when $1-2b\sigma^2\leq 0$, the above expression equals $\infty$. When $c=0$, the expression (\ref{forward_map}) simplifies to 
$$b\left(1- \frac{\alpha^2}{1-2b\sigma^2}\right)x^2 +\lambda+ \frac{1}{2}\log(1-2b\sigma^2).$$
Notice that although for our choice of $g(x)$ is not in $\Gamma$ directly, our goal is to apply Theorem \ref{thm:ergcost} to function of the form (\ref{forward_map}). As long as $b\in[-\frac{1}{2},\frac{1}{2}]\cap(-\infty,\frac{1}{2\sigma^2})$, by Lemma \ref{approx_by_Gamma_quadratic}), one can take a sequence of functions from $\Gamma$ which converge pointwise to $g$ with absolute value dominated by $\max(|g|,M)$, where $M$ is a constant depending only on the coefficient of $g$. Thus by dominated convergence theorem, we can take $f$ as (\ref{forward_map}) in Theorem \ref{thm:ergcost} and the inequality there still holds.

\subsubsection{Back to uncertainty bounds in ACUT problem}
Now let's reconsider the dynamics from (\ref{discrete_model}) and (\ref{discrete_model_perturbed}). By considering them as Markov chain, the transition kernel for for (\ref{discrete_model}) is $p_N(x,\cdot)\sim N((1-\frac{a}{N})x,\frac{\sigma^2}{N})$, and the transition kernel for (\ref{discrete_model_perturbed}) is $q_N(x,\cdot)\sim N((1-\frac{a}{N})x+\sigma\frac{u}{N},\sigma^2\frac{v^2}{N})$, where $u,v$ are functions of $x$ satisfying Assumption \ref{assumption_of_uv}. One can show the existence and uniqueness of the stationary distribution of (\ref{discrete_model_perturbed}) using similar arguments which are used for stationary distributions of (\ref{model_perturbed}). We denote the stationary distribution of (\ref{discrete_model_perturbed}) by $\pi_{q_N}$. The bound we are going to get from Theorem \ref{thm:ergcost} by taking $g(x) = bx^2$ will be
\begin{align}\label{quadratic_bound_discretized}
\int_{\mathbb{R}} b(1-\frac{(1-a/N)^2}{1-2b\sigma^2/N})x^2d\pi_{q_N}\leq \int_\mathbb{R} G_\Gamma(q_N(x,\cdot)\lVert p_N(x,\cdot))d\pi_{q_N} -\frac{1}{2} \log(1-2b\sigma^2/N).
\end{align}
To make the most tight bound, we use the set $\Gamma_{2|b|}$ as in Remark \ref{Gamma_k} instead of the plain $\Gamma$. If $v = 1$, then by Theorem \ref{thm:Gamma_div_two_normal}, 
$$G_{\Gamma_{2|b|}}\left(q_N(x,\cdot)\lVert p(x,\cdot)\right)  = R(q_N(x,\cdot)\lVert p(x,\cdot) = \frac{\sigma^2 (u/N)^2}{2\sigma^2/N} = \frac{u^2}{2N}.$$
When $v\neq 1$, then for any fixed $b$, there exists large enough $N$ such that the difference of the inverse variance of the two transition kernel $p_N(x,\cdot) $ and $q_N(x,\cdot)$, $|\frac{N}{\sigma^2} - \frac{N}{\sigma^2v^2}| = \frac{N}{\sigma^2}|1-\frac{1}{v^2}| > 2|b|$. Then by (\ref{Gamma_div_two_normal_Lip_k}), we get 
$$G_{\Gamma_{2|b|}}(q_N(x,\cdot)\lVert p_N(x,\cdot)) = |b|\frac{\sigma^2}{N}|v^2-1| + \frac{\sigma^2 (u/N)^2}{2\sigma^2/N} + O(\frac{1}{N^2})=\frac{|b|\sigma^2}{N}|v^2-1| + \frac{u^2}{2N}+O(\frac{1}{N^2}).$$
Putting these two cases together, we can conclude for $N$ large enough, 
\begin{align}\label{discretized_Gamma_diffusion}
G_{\Gamma_{2|b|}}(q_N(x,\cdot)\lVert p_N(x,\cdot)) = \frac{|b|\sigma^2}{N}|v-1| + \frac{u^2}{2N}+O(\frac{1}{N^2}). 
\end{align}
Putting (\ref{discretized_Gamma_diffusion}) back into the inequality (\ref{quadratic_bound_discretized}) with $\Gamma_{2|b|}$ substituting $\Gamma$, we have 
$$\int_\mathbb{R}b(1-\frac{(1-a/N)^2}{1-2b\sigma^2/N})x^2d\pi_{q_N}\leq \int_\mathbb{R} \left(\frac{|b|\sigma^2}{N}|v^2-1| + \frac{u^2}{2N}\right)d\pi_{q_N} -\frac{1}{2} \log(1-2b\sigma^2/N)+O(\frac{1}{N^2}).$$
Notice that one has the expansion
\begin{align*}
1-\frac{(1-a/N)^2}{1-2b\sigma^2/N} &=  \frac{1}{1-2b\sigma^2/N}\left((1-2b\sigma^2/N) - (1-a/N)^2 \right)\\
& = \frac{1}{1-2b\sigma^2/N} \left(2(a-b\sigma^2)/N - a^2/N^2\right)\\
& = 2(a-b\sigma^2)/N + O(\frac{1}{N^2}).
\end{align*}
For $b\in (0, \frac{a}{\sigma^2})$, we divide both sides by $b(1-\frac{(1-a/N)^2}{1-2b\sigma^2/N})$ and put in the expansion above, we get
\begin{align}
\int_\mathbb{R} x^2 d\pi_{q_N} &\leq \int_\mathbb{R}\left(\frac{\sigma^2|v^2-1|/N}{1-\frac{(1-a/N)^2}{1-2b\sigma^2/N}} + \frac{u^2/2N}{b(1-\frac{(1-a/N)^2}{1-2b\sigma^2/N})}\right)d\pi_{q_N} - \frac{1}{2b(1-\frac{(1-a/N)^2}{1-2b\sigma^2/N})}\cdot\\
&\quad \log(1-2b\sigma^2/N) + O(\frac{1}{N})\\
&\leq \int_\mathbb{R}\left(\frac{\sigma^2|v^2-1|}{2(a-b\sigma^2)} + \frac{u^2/2}{2b(a-b\sigma^2)}\right)d\pi_{q_N} + \frac{2b\sigma^2/N}{4b(a-b\sigma^2)/N} + O(\frac{1}{N})\\
& = \int_\mathbb{R}\left(\frac{\sigma^2|v^2-1|}{2(a-b\sigma^2)} + \frac{u^2/2}{2b(a-b\sigma^2)}\right)d\pi_{q_N} + \frac{\sigma^2}{2(a-b\sigma^2)} + O(\frac{1}{N}). \label{discretized_diffusion_bound}
\end{align}

Next we want to establish the connection between the stationary distribution of the discrete system (\ref{discrete_model_perturbed}) $\pi_{q_N}$ and the stationary distribution of the continuous system (\ref{model_perturbed}) $\pi_q$. We put a lemma here. This is standard but included for completeness.
\begin{lemma}
$\pi_{q_N}$ converges weakly to $\pi_q$.
\end{lemma}
\begin{proof}
First from (\ref{discretized_diffusion_bound}), by taking $b=\frac{a}{2\sigma^2}$, and Assumption \ref{assumption_of_uv}, we can conclude $\{\int_\mathbb{R}x^2d\pi_{q_N}\}_{N\geq 1}$ is uniformly bounded, by which the author means there exists a constant $M>0$ such that $\int_\mathbb{R}x^2d\pi_{q_N}\leq M$ for all $N$. Then by Chebyshev inequality, one can conclude $\{\pi_N\}_{N\geq 1}$ is tight. By Prohorov's theorem \cite{dupell4}[Theorem A.3.15], $\{\pi_N\}_{N\geq 1}$ is precompact. Thus for any subsequence of this sequence of measures, we can always extract a further subsequence of it which converges weakly. By an abuse of notation, let's also denote the convergent subsequence $\{\pi_{N}\}_{N\geq 1}$, and denote its weak limit by $\pi^*$. Denote $\mathcal{L}_N$ as the generator for (\ref{discrete_model_perturbed}), and $\mathcal{L}$ the generator for (\ref{model_perturbed}), where generators are defined in the following sense: For $h\in \mathcal{M}_b(\mathbb{R})$ being bounded measurable functions, 
$$\mathcal{L}h(x) \doteq \lim_{t\to 0^+} \frac{E_x[h(\bar{X}_t)] - h(x)}{t},$$
whenever the above limits exists. Here $\bar{X}_t$ follows the dynamics of (\ref{model_perturbed}) with $X_0 = x$. 
$$\mathcal{L}_N h(x) \doteq N \left(E_x[h(\bar{X}_1^{(N)})] - h(x)\right),$$
where $\bar{X}$ follows the dynamics of (\ref{discrete_model_perturbed}) with $\bar{X}_0 = x$.\\
For $h\in C^\infty_0(\mathbb{R})$ which has derivative of any order and compact support, it can be shown  that for all $x\in\mathbb{R}$,
$$\lim_{N\to\infty} \mathcal{L}_N h(x) = \mathcal{L} h(x),$$
and the convergence is uniform. So we have 
\begin{align*}
& \quad\limsup_{N\to\infty} |\int_\mathbb{R} \mathcal{L}_N h(x) d\pi_{q_N} - \int_\mathbb{R} \mathcal{L} h(x) d\pi^*| \\
&= \limsup_{N\to\infty} |\int_\mathbb{R} (\mathcal{L}_N h(x) -\mathcal{L} h(x)) d\pi_{q_N}  + \int_\mathbb{R} \mathcal{L} h(x) d\pi_{q_N}- \int_\mathbb{R} \mathcal{L} h(x) d\pi^* | \\
&\leq \limsup_{N\to\infty} |\int_\mathbb{R} (\mathcal{L}_N h(x) -\mathcal{L} h(x)) d\pi_{q_N}| + \limsup_{N\to\infty}|\int_\mathbb{R} \mathcal{L} h(x) d\pi_{q_N}- \int_\mathbb{R} \mathcal{L} h(x) d\pi^*|\\
&=0.
\end{align*}
Since $q_N$ is the stationary distribution for (\ref{discrete_model_perturbed}), $\int_\mathbb{R} \mathcal{L}_N h(x) d\pi_{q_N}=0$. Thus from the equations above, we can conclude $\int_\mathbb{R} \mathcal{L} h(x) d\pi^*=0$ for all $h\in C_0^2(\mathbb{R})$. By Echeverria's theorem \cite{ethkur}[Theorem 4.9.17], $\pi^*$ is a stationary distribution of (\ref{model_perturbed}). Since the stationary distribution of (\ref{model_perturbed}) is unique, thus $\pi^* = \pi_q$. 

Finally, we can show by contradiction that the weak limit of any convergent subsequence of $\{\pi_{q_N}\}_{N\geq 1}$. Combine this with the fact $\{\pi_{q_N}\}_{N\geq 1}$ is precompact, we prove the statement of the lemma.
\end{proof}

Then since $u,v$ are bounded and continuous functions of $x$, by letting $N\to\infty$, we have 
$$\int_\mathbb{R} x^2 d\pi_q \leq \int_\mathbb{R} \left(\frac{\sigma^2|v^2-1|}{2(a-b\sigma^2)} + \frac{u^2/2}{2b(a-b\sigma^2)}\right)d\pi_{q} + \frac{\sigma^2}{2(a-b\sigma^2)}$$
holds for any $b\in(0,\frac{a}{\sigma^2})$. Lastly, we can get the bound
\begin{align}\label{uncertainty_bound_diffusion_model}
\int_\mathbb{R} x^2 d\pi_q \leq \inf_{b\in (0,\frac{a}{\sigma^2})}\left\{\int_\mathbb{R} \left(\frac{\sigma^2|v^2-1|}{2(a-b\sigma^2)} + \frac{u^2/2}{2b(a-b\sigma^2)}\right)d\pi_{q} + \frac{\sigma^2}{2(a-b\sigma^2)}\right\}.    
\end{align}
In one specific choice of $b= \frac{a}{2\sigma^2}$, we can get the bound
$$\int_\mathbb{R} x^2 d\pi_q \leq \int_\mathbb{R} \left(\frac{\sigma^2|v^2-1|}{a} + \frac{u^2\sigma^2}{a^2}\right)d\pi_{q} + \frac{\sigma^2}{a}.$$

\begin{remark}
The stationary distribution $\pi_p$ of (\ref{model}) is $N(0,\frac{\sigma^2}{2a})$. So when there is no perturbation, 
$$\int_\mathbb{R} x^2 d\pi_p = \frac{\sigma^2}{2a},$$
which equals to the right hand side of (\ref{uncertainty_bound_diffusion_model}) when $u\equiv 0$ and $v\equiv 1$. 

If $v\equiv 1$, which means there is no perturbation on the diffusion coefficient,  (\ref{uncertainty_bound_diffusion_model}) becomes 
\begin{align}
\int_\mathbb{R} x^2 d\pi_q \leq \inf_{b\in (0,\frac{a}{\sigma^2})}\left\{\int_\mathbb{R} \left( \frac{u^2/2}{2b(a-b\sigma^2)}\right)d\pi_{q} + \frac{\sigma^2}{2(a-b\sigma^2)}\right\}.
\end{align}
Actually, one can get this same bound by substituting $G_\Gamma(q_N(x,\cdot)\lVert p_N(x,\cdot))$ with $R(q_N(x,\cdot)\lVert p_N(x,\cdot))$ in (\ref{quadratic_bound_discretized}), and take $N\to \infty$. On the other hand, when $v\not\equiv 1$, by substituting $G_\Gamma(q_N(x,\cdot)\lVert p_N(x,\cdot))$ with $R(q_N(x,\cdot)\lVert p_N(x,\cdot))$ in (\ref{quadratic_bound_discretized}), and take $N\to \infty$, the upper bound one gets is $\infty$, which is useless. So in the situation where diffustion coefficient is perturbed, by using $\Gamma$-divergence compared to relative entropy, one can get useful uncertainty bounds.
\end{remark}

\pagebreak




\singlespacing



\thispagestyle{plain}


\let\oldthebibliography=\thebibliography
\let\endoldthebibliography=\endthebibliography
\renewenvironment{thebibliography}[1]{
  \begin{oldthebibliography}{#1}
    \setlength{\itemsep}{3ex}
  }{
    \end{oldthebibliography}
  }

\bibliographystyle{abbrv}
\bibliography{biblio}




\end{document}